\documentclass{book}
\usepackage{fullpage}
\usepackage[utf8x]{inputenc}    
\usepackage[T1]{fontenc}
\usepackage{amsmath,amsfonts,latexsym,tikz, amsthm}
\usepackage{url}
\usepackage[greek, french, english]{babel}
\usepackage{epigraph}

\usepackage{color}

\usepackage{tikz}
\usetikzlibrary{decorations.pathmorphing, decorations.markings, arrows, matrix, fit, positioning, chains}

\usepackage{tikz-cd}

\renewcommand{\l}{\left[}
\renewcommand{\r}{\right] }

\newtheorem{teo}{Theorem}[section]
\newtheorem{coro}[teo]{Corollary}
\newtheorem{defin}[teo]{Definition}
\newtheorem{ex}{Example}
\newtheorem{lema}[teo]{Lemma}
\newtheorem{prop}[teo]{Proposition}
\newtheorem{rem}[teo]{Remark}

\newtheorem{nota}[teo]{Notation}
\newtheorem{slema}[teo]{Sub-lemma}
\newtheorem{hyp}[teo]{Hypothesis}

\newcommand{\N}{\mathbb{N}}
\newcommand{\Z}{\mathbb{Z}}
\newcommand{\R}{\mathbb{R}}
                  
\def\Im{{\text{Im}}} 
\DeclareMathOperator{\Id}{Id}
\DeclareMathOperator{\diag}{diag}

\newcommand{\Ch}{\mathbf{C(k)}}
\newcommand{\dg}{\mathbf{dg-cat}}
\newcommand{\sS}{\mathbf{sSet}}
\newcommand{\Set}{\mathbf{Set}}
\newcommand\Mod{\mathbf{Mod}}
\newcommand\Top{\mathbf{Top}}
\newcommand{\sC}{\mathbf{sCat}}
\newcommand{\Cat}{\mathbf{Cat}}
\newcommand{\Fact}{\mathbf{Fact}}
\newcommand{\Free}{c\mathcal{L}}
\newcommand{\FreeS}{c\mathcal{L}_\S}

\newcommand{\Dnsd}{\Delta_k(n,s,d)}
\newcommand{\D}{\Delta_k}
\newcommand{\0}{\emptyset}

\newcommand{\dgS}{\mathbf{dg-Segal}}
\newcommand{\dgSc}{\mathbf{dg-Segal_c}}
\renewcommand{\k}{\underline{k}}

\renewcommand{\S}{\mathbb{S}}

\DeclareMathOperator{\Ob}{Obj}
\DeclareMathOperator{\Mor}{Morph}
\DeclareMathOperator{\Fun}{Fun}
\DeclareMathOperator{\Map}{Map}
\DeclareMathOperator{\Ho}{Ho}
\newcommand{\op}{^{op}}
\DeclareMathOperator{\Hom}{Hom}
\DeclareMathOperator{\Adj}{Adj}
\renewcommand{\O}{\mathcal{O}}
\DeclareMathOperator{\map}{map}

\DeclareMathOperator{\coeq}{coeq}
\DeclareMathOperator{\colim}{colim}
\renewcommand{\L}{\mathbb{L}}
\DeclareMathOperator{\Int}{Int}
\DeclareMathOperator{\Aut}{Aut}

\DeclareMathOperator{\hocolim}{hocolim}
\DeclareMathOperator{\holim}{holim}
\DeclareMathOperator{\Alg}{Alg}

\renewcommand{\Re}{\operatorname{Re}}
\DeclareMathOperator{\Sing}{Sing}
\DeclareMathOperator{\sk}{sk}
\DeclareMathOperator{\cosk}{cosk}

\begin{document}

\title{A new model for dg-categories}
\author{Elena Dimitriadis Bermejo}
\date{16th December 2022}
\maketitle
\chapter*{}

%\epigraph{\begin{otherlanguage}{greek}Πάντα στον νου σου να ’χεις την Ιθάκη.
%\\Το φθάσιμον εκεί είν’ ο προορισμός σου.
%\\Αλλά μη βιάζεις το ταξίδι διόλου.
%\\Καλύτερα χρόνια πολλά να διαρκέσει·
%\\και γέρος πια ν’ αράξεις στο νησί,
%\\πλούσιος με όσα κέρδισες στον δρόμο,
%\\μη προσδοκώντας πλούτη να σε δώσει η Ιθάκη.  \end{otherlanguage}
%\\
%\\
%\\Keep Ithaka always in your mind.
%\\Arriving there is what you’re destined for.
%\\But don’t hurry the journey at all.
%\\Better if it lasts for years,
%\\so you’re old by the time you reach the island,
%\\wealthy with all you’ve gained on the way,
%\\not expecting Ithaka to make you rich.}
%{\begin{otherlanguage}{greek}---Κωσταντίνος Καβάφης, \end{otherlanguage}fragment of \begin{otherlanguage}{greek}\textit{Ιθάκη}\end{otherlanguage}}
\tableofcontents

\begin{otherlanguage}{french}

\chapter{Introduction en fran\c{c}ais}

\epigraph{\textit{"There are some things that it is better to begin than to refuse, even though the end may be dark."}}{---J.R.R. Tolkien, \textit{The Lord of the Rings}}

L'objectif de cette thèse est de définir et étudier un nouveau modèle pour la théorie homotopique des dg-catégories, qui se comportera mieux que la structure de modèles originale sur les dg-catégories, et qui sera aussi plus proche des modèles établis dans le cadre des $\infty$-catégories. Avec cela en tête, avant de commencer ce sera utile de faire un petit tour de ce qu'on sait sur les modèles existants dans les $\infty$-catégories et les dg-catégories.
\\
\\On a fait de notre mieux pour expliquer les notations à leur apparition, mais la notation des outils principaux sera recueillie dans la Section \ref{intr-fr-notations} pour être plus facilement accessible. 

\section{Modèles des $\infty$-catégories}

Commençons par les catégories supérieures. Intuitivement, on peut dire qu'une $\infty$-catégorie (ou, pour être exact, une $(\infty,1)$-catégorie) est une catégorie enrichie sur les $\infty$-groupoïdes : une catégorie avec un ensemble d'objets et, pour tout couple d'objets, un $\infty$-groupoïde entre eux. En d'autres termes, au lieu d'avoir simplement des morphismes entre les objets, on demande à avoir des $2$-morphismes entre les morphismes, des $3$-morphismes entre les $2$-morphismes, etc., tout en demandant que tous les $n$-morphismes avec $n\geq 2$ soient inversibles. Cette construction, appelée aussi "une théorie homotopique" dans la littérature, apparaît dans plusieurs situations différentes.
\\
\\Un exemple classique où les $\infty$-catégories apparaissent naturellement est celui qui arrive quand on s'intéresse à des structures qui sont classifiées à une notion plus faible que celle des isomorphismes près (les espaces topologiques à équivalence d'homotopie près, les complexes de chaînes à équivalence faible près, entre autres). On voudrait localiser ces catégories pour que cette classe de morphismes soit inversible. Ce problème a déjà été résolu par Quillen dans \cite{Quillen} avec l'apparition des catégories de modèles ; mais même si elles sont très utiles, le problème avec celles-là est que les catégories de modèles ont une vaste quantité de structure qu'on doit ajouter afin de les faire fonctionner. En plus, leur catégorie homotopique (c'est-à-dire, la localisation) ne se rappelle pas de l'information ajoutée. Pour résoudre ces soucis, Dwyer et Kan ont introduit dans \cite{DK-infini-cat} une localisation simpliciale, qui est une catégorie enrichie sur les objets simpliciaux. À partir de là, un bon nombre de chercheur.es ont travaillé pour trouver un bon modèle pour une catégorie de cette forme. Il y a quatre modèles qui ont été particulièrement importants.
\\
\\La première option est clairement celle des catégories simpliciales en elles-mêmes. Dwyer et Kan ont développé la théorie des localisations simpliciales dans ce contexte, en tant que catégories enrichies sur les espaces simpliciaux. Comme on en parlera en détail dans la Section \ref{Simplicial categories}, on ne passera pas trop de temps dessus maintenant. On ajoutera uniquement que Bergner a prouvé dans \cite{Berg-simplicial} que la catégorie des catégories simpliciales a une structure de modèles.
\\
\\Un deuxième modèle qui a été très fructueux, est celui des quasi-catégories. Définies pour la première fois par Boardman et Vogt dans \cite{Boardman-Vogt-quasicat}, cette approche définit les quasi-catégories (aussi appellées "complexes de Kan faibles") comme des ensembles simpliciaux $X$ tels que pour toute corne interne $\Lambda^k[n]\to X$ il existe un remplissage $\Delta[n]\to X$, pour tous les $n\in\N$ et $0<k<n$. En pratique, on demande que tout couple de $n$-morphismes ait un troisième morphisme, la "composition", et un $(n+1)$-morphisme reliant les deux. Cette théorie a été étendue plus tard par Joyal dans des articles comme \cite{JoyalQuasi-cat}, et par Lurie dans \cite{HTT} et \cite{LurieHA}. Joyal et Tierney ont aussi défini une structure de modèles sur les quasi-catégories dans \cite{JTquasi-Segal}.
\\
\\Notre troisième modèle est celui des catégories de Segal. Celles-ci sont une généralisation naturelle des catégories simpliciales, puisqu'on peut les voir comme des catégories simpliciales avec une composition définie uniquement à homotopie près. Elles ont été définies pour la première fois par Dwyer, Kan et Smith dans \cite{Dwyer-Kan-Smith}, et elles sont définies comme des espaces simpliciaux $X:\Delta\op\to \sS$ tels que $X_0$ est un ensemble simplicial discret et que pour tout $k\geq 1$ le morphisme de Segal
$$\phi_k:X_k\to \overbrace{X_1\times_{X_0}\ldots\times_{X_0}X_1}^{\text{$k$ fois}} $$
est une équivalence faible. La catégorie des catégories de Segal a une structure de modèles, construite pour les $n$-catégories de Segal par Hirschowitz et Simpson dans \cite{Simpson} et dans une autre version par Bergner dans \cite{Bergner-equivalences}.
\\
\\Et finalement, le quatrième modèle, et le plus intéressant pour nous, est celui des espaces de Segal. Définis pour la première fois par Rezk dans \cite{ComSegalSpacesREZK}, un espace de Segal complet est aussi un espace simplicial où le morphisme de Segal $\phi_k$ est une équivalence faible pour tout $k\geq 1$, mais au lieu de demander que $X_0$ soit discret, on demande que le morphisme 
$$X_0\to X_{hoequiv} $$
soit une équivalence faible, où $X_{hoequiv}$ est l'espace des équivalences d'homotopie. Dans le même papier, Rezk construit une structure de modèles dans laquelle les espaces de Segal complets sont les objets fibrants.
\\
\\Cela dit, on a construit quatre modèles, en disant qu'ils étaient tous des modèles de la même chose. Mais cela n'est pas évident en regardant les définitions. Oui, elles ont des choses en commun, mais elles sont aussi assez différentes les unes des autres. C'est la raison pour laquelle on a introduit les structures de modèles : en effet, Bergner a démontré dans \cite{Bergner-equivalences} qu'il existe une équivalence de Quillen (c'est-à-dire, une équivalence de catégories de modèles) entre les espaces de Segal complets et les catégories de Segal, et une autre entre les catégories de Segal et les catégories simpliciales. De leur côté, Joyal et Tierney ont démontré dans \cite{JTquasi-Segal} que les quasi-catégories et les espaces de Segal complets sont aussi Quillen équivalents. Attention : même si on sait que toutes ces catégories sont Quillen équivalentes, les équivalences de Quillen vont dans des directions opposées, ce qui veut dire qu'elles ne peuvent pas être composées pour en faire un seule équivalence de Quillen.
\\
\\Il y a eu quelques efforts pour proposer une définition axiomatique des $\infty$-catégories au fil des années, comme par exemple la "théorie des $\infty$-catégories" de Toën dans \cite{TOEN-infini} et les "$\infty$-topos" de Riehl et Verity dans \cite{Verity-Riehl}, mais on n'en parlera pas plus en détail ici.

\section{Modèles des dg-catégories}

Maintenant parlons des dg-catégories. Comme les catégories simpliciales, les dg-catégories sont définies comme des catégories enrichies sur quelque chose : dans ce cas, des complexes de cochaînes. Comme on en parlera en détail dans la Section \ref{Dg-categories}, on ne se perdra pas dans les détails ici : on dira uniquement que Tabuada a prouvé dans \cite{TAB_Fr} qu'il y a une structure de modèles sur $\dg$.
\\
\\Si les $\infty$-catégories ont une longue histoire, les dg-catégories sont encore plus vieilles : en effet, on les retrouve déjà dans des papiers des années 60, comme l'article de Kelly \cite{Kelly}. En vue de cela, on ne donnera pas une histoire complète de leur évolution ; ce n'est pas non plus l'objet de cette introduction. Il suffira de dire que même si les catégories et les $\infty$-catégories sont suffisantes pour travailler en Topologie Algébrique, la Géométrie Algébrique a souvent besoin de travailler avec une notion de linéarité qui ne marche pas toujours bien avec celles-là. En conséquence, les dg-catégories, qui ont déjà un concept de linéarité, se sont retrouvées comme un objet essentiel dans ce domaine. Un bon exemple de leur usage est celui du Programme de Langlands Géométrique, qui a été réécrit par Arinkin et Rozenblyum dans \cite{Geometric-Langlands} en termes de dg-catégories. Pour plus de détails sur cela, on pointe le lecteur vers le livre \textit{A study in Derived Algebraic Geometry}, par Gaitsgory et Rozenblyum (voir \cite{Gaitsgory-Rozenblyum-Vol-I} et \cite{Gaitsgory-Rozenblyum-Vol-II}).
\\
\\Mais même si les dg-catégories sont très utiles, elles ne sont pas parfaites. Par exemple, la catégorie des dg-catégories a une structure de modèles, et elle a aussi une structure monoïdale ; mais elles ne sont pas compatibles. Ce problème, parmi d'autres, a poussé des chercheur.es à chercher d'autres modèles pour les dg-catégories, dans le style du travail fait dans les $\infty$-catégories. On a déjà parlé de notre première comparaison : les dg-catégories sont les catégories simpliciales. Il est intéressant de remarquer, d'ailleurs, qu'une des raisons pour lesquelles la communauté a commencé à chercher d'autres modèles d'$\infty$-catégories est que les catégories simpliciales ont aussi une structure de modèles et une structure monoïdale (dans ce cas, le produit direct) qui ne sont pas compatibles.
\\
\\Commençons par les quasi-catégories. On a un résultat dû à Cohn dans \cite{COHN} qui nous dit qu'il y a une équivalence entre l'$\infty$-catégorie sous-jacente à la catégorie de modèles des dg-catégories localisée par les équivalences de Morita, et l'$\infty$-catégorie des quasi-catégories stables $k$-linéaires idempotent-complètes. Malheureusement pour nous, les quasi-catégories stables $k$-linéaires ne sont pas faciles à utiliser : par exemple, si on a un adjoint entre deux quasi-catégories linéaires, il est très difficile de prouver qu'il y a un adjoint linéaire. Aussi, si $x$ est un objet dans une quasi-catégorie $k$-linéaire, prouver que $\operatorname{End}(x)$ peut être représenté par une dg-algèbre est très compliqué dans le monde des quasi-catégories $k$-linéaires, mais découle presque de la définition dans les dg-catégories.
\\
\\Plus récemment, il y a eu une autre approche à ce sujet : Mertens a proposé dans \cite{Arne-thesis} une définition d'une quasi-catégorie enrichie sur une catégorie monoïdale cocomplète $C$ (voir la  Définition 2.2.26 dans \cite{Arne} pour plus de détails) et a construit un foncteur qui relève le dg-nerf classique $\dg\to \sS$ dans un dg-nerf linéaire
$$\dg\to S_\otimes\Mod(k) $$
où $S_\otimes\Mod(k)$ est la catégorie des modules templiciaux ("tensor simpliciaux" ; voir la Définition 2.4 dans \cite{Arne} pour plus de détails). Dans un exposé à l'Institut Poincaré en Septembre 2022 (voir \cite{conf}), Lowen a mentionné que Mertens et elle-même sont en train de travailler sur une équivalence de Quillen pour relier leurs quasi-catégories enrichies sur $\Mod(k)$ avec les dg-catégories, mais que c'est encore un travail en cours.
\\
\\Du côté des catégories de Segal, Bacard a défini dans \cite{Bacard-enriched-Segal-cat} une notion de catégorie de Segal enrichie, et en conséquence une notion de dg-catégorie de Segal, en disant que, une fois un ensemble $\O$ fixé, une dg-catégorie de Segal est un morphisme $W$-colaxe de la forme
$$F:P_{\overline{\O}}\to \Ch, $$
où $\overline{\O}$ est un groupoïde appellé "the coarse category associated to $\O$" et $P_{\overline{\O}}$ est la 2-path-category de $\overline{\O}$ (voir Définitions 4.1 et 2.7 dans \cite{Bacard-enriched-Segal-cat} pour plus de détails). Dans ce cas, aussi, il paraît que la comparaison entre ces objets et les dg-catégories est un travail en cours : en effet, dans l'introduction de ce papier Bacard dit que le premier objectif une fois les catégories de Segal enrichies bien définies est de développer la théorie homotopique des dg-catégories de Segal et de les comparer avec les dg-catégories classiques.
\\
\\On a alors une définition de dg-quasi-catégorie, et une définition de dg-catégorie de Segal, mais si on suit le schéma de la section précédente, il nous manque un modèle : on n'a pas encore parlé d'espaces de dg-Segal complets. Et effectivement, c'est à cet endroit que nos résultats vont se trouver. En suivant Rezk, on va construire une version "linéaire" des espaces de Segal complets et on va (à une certaine hypothèse près) prouver qu'il y a une chaîne d'adjonctions de Quillen qui en fait une équivalence de catégories homotopiques.
%\newpage
\begin{center}
\begin{tabular}{|c|c|}
\hline
catégories simpliciales & dg-catégories \\
\hline
quasi-catégories &  quasi-catégories $k$-linéaires \\
                 &  dg-quasi-catégories\\
\hline
catégories de Segal & dg-catégories de Segal\\
\hline
espaces de Segal complets & ??\\
\hline
\end{tabular}
\end{center}

\section{Espaces de Segal vs espaces de dg-Segal}

Dans l'article de Rezk, \cite{ComSegalSpacesREZK}, on a plusieurs étapes importants pour construire son modèle des $\infty$-catégories. Premièrement, il définit les espaces de Segal comme vu avant, comme un espace simplicial (c'est-à-dire un foncteur de $\Delta\op$ dans $\sS$) tel que certains morphismes sont des équivalences faibles. Après, il démontre qu'il existe une localisation de Bousfield de la structure projective sur les espaces simpliciaux telle que les espaces de Segal sont les objets fibrants dans celle-là. Mais cette structure n'est pas suffisante pour donner une équivalence avec les $\infty$-catégories : il y a une certaine classe de morphismes qui devraient être des équivalences mais ne le sont pas encore. Avec cela en tête, il applique une autre localisation de Bousfield à cette structure de modèles, et il trouve une structure de modèles dont les objets fibrants sont les espaces de Segal complets. Finalement, il définit une classe de morphismes dans la catégorie des espaces simpliciaux, appelés des équivalences de Dwyer-Kan, et il prouve que les équivalences faibles pour la structure de catégorie de modèles pour les espaces de Segal complets sont exactement les équivalences de Dwyer-Kan entre des espaces de Segal.
\\
\\Dans notre cas, on suivra un schéma similaire ; mais le chemin n'est pas si simple. En effet, la structure linéaire des dg-catégories va nous compliquer la tâche à plusieurs reprises. Premièrement, on choisira une catégorie de foncteurs (dans notre cas, la catégorie des foncteurs simpliciaux entre des catégories libres de type fini et les espaces simpliciaux) et on définit les espaces de dg-Segal comme des objets dans cette catégorie tels que certains morphismes soient des équivalences faibles. Déjà, la condition du morphisme de Segal ne sera pas suffisante : on aura besoin d'ajouter une condition en disant qu'on peut ajouter des termes aux complexes de modules dans les morphismes.
\\
\\Une fois que cela est fait, ce n'est pas compliqué de montrer qu'il y a une localisation de Bousfield telle  que les objets locaux sont exactement les espaces de dg-Segal qui sont fibrants pour la structure de modèles projective. Malheureusement, comme dans le cas des espaces de Segal, le foncteur nerf (qu'on appellera ici $\Sing$) ne forme pas une équivalence de Quillen, nous forçant à faire une deuxième localisation de Bousfield. On n'aura pas besoin de chercher trop loin : on définira tout simplement un foncteur "d'oubli" entre les espaces de dg-Segal et les espaces de Segal classiques et on définira les espaces de dg-Segal complets comme les espaces de dg-Segal qui deviennent complets dans l'image du foncteur oubli.
\\
\\Mais ici on se heurte encore à la linéarité. En effet, pour prouver que le nerf est une équivalence de Quillen dans cette structure de modèles, on aimerait faire comme Rezk et définir une classe de morphismes, qu'on appelle des DK-équivalences, de façon à que les équivalences faibles dans la structure de modèles des espaces de Segal complets soient exactement les DK-équivalences entre des espaces de dg-Segal. Mais pour faire cela, Rezk utilise le produit direct dans les espaces simpliciaux pour calculer une certaine exponentielle. On n'a pas ce privilège : comme on a mentionné plus tôt, la structure de modèles et la structure monoïdale ne sont pas compatibles dans $\dg$. En conséquence, pour prouver notre résultat on devra d'abord définir une structure de modèles dans nos espaces de Segal complets. À cause d'un manque de temps, on n'a pas pu le faire dans ce manuscrit, et on admettra le fait que les équivalences dans la structure des espaces de dg-Segal complets sont les DK-équivalences entre des espaces de dg-Segal comme hypothèse.
\\
\\En supposant que cette hypothèse est vraie, par contre, on peut montrer qu'il existe une équivalence entre la catégorie homotopique des dg-catégories et la catégorie homotopique des espaces de dg-Segal complets pour la structure de modèles des espaces de dg-Segal complets.

\section{Pourquoi ça nous intéresse ?}

Évidemment, si on va investir tout ce temps et tout cet effort pour construire des espaces de dg-Segal complets, on doit justifier nos choix. Pourquoi on veut ces objets ?
\\
\\La première raison a été déjà mentionnée plusieurs fois : la catégorie des dg-catégories telle qu'on la connaît a une structure de modèles et une structure monoïdale qui ne sont pas compatibles. En effet, l'objet $\D(1,0,1)$ (c'est-à-dire la dg-catégorie avec deux objets et $k$ le complexe de morphismes entre les deux) est cofibrant dans la structure de modèles des dg-catégories, mais il est facile de voir que $\D(1,0,1)\otimes\D(1,0,1)$ ne l'est pas. Cela dit, la construction même des espaces de dg-Segal complets nous oblige à définir une structure monoïdale sur ceux-là qui sera compatible avec la structure de modèles, résolvant ce souci. Ce serait, en soi, un pas important pour démontrer des versions linéaires des théorèmes classiques de théorie des catégories, comme par exemple le théorème de Barr-Beck, qui non seulement ne peut pas être déduit de sa version non-linéaire, mais ne peut même pas être défini exactement dans le cas linéaire en l'état actuel.
\\
\\En plus de cela, la construction des dg-catégories comme des espaces de dg-Segal nous donne une catégorie de foncteurs, qui est, citant Dugger dans \cite{Dugger}, "une espèce de présentation par générateurs et relations". Comme c'est le cas dans ce type de structures, une telle présentation rendrait la construction des morphismes partant des espaces de dg-Segal complets beaucoup plus facile : on aura juste à les définir sur les générateurs et s'assurer que les "relations" sont envoyées sur des équivalences faibles. En particulier, cela nous donnerait une façon plus simple de calculer les automorphismes de $\dg$.
\\
\\Cela n'est pas un nouveau développement : à la fois Toën dans \cite{TOEN-infini} et Barwick et Schommer-Pries dans \cite{infinity-n-cat} ont utilisé des telles "présentations par générateurs et relations" pour définir des axiomatisations des $(\infty,1)$-catégories et des $(\infty,n)$-catégories, respectivement ; et ensuite ils les utilisent pour montrer que le groupe des automorphismes dans les $(\infty, n)$-catégories est isomorphe au groupe discret $(\Z/2\Z)^n$. Ce n'est pas une coïncidence que le papier de Toën finisse par le résultat que toute théorie des $\infty$-catégories est équivalente à la catégorie des espaces de Segal complets, et pas un autre modèle des $\infty$-catégories.

\section{Organisation de ce texte}

Cette thèse est divisée en trois parties différentes, chacune marquée par un chapitre différent. 
\\
\\Dans le premier chapitre, "Background notions", on va rappelle au lecteur.e (ou lui présenter) des notions qu'on utilisera par la suite : catégories de modèles et adjonctions de Quillen, spécialement celles des complexes de cochaînes et des diagrammes ; catégories de modèles propres à gauche et à droite ; mapping spaces ; localisations de Bousfield à gauche ; objets simpliciaux, catégories simpliciales ; et dg-catégories. On va faire particulièrement attention à la structure de modèles des dg-catégories. Ce chapitre inclut aussi une section expliquant la construction de la "catégorie de modèles universelle" de Dugger.
\\
\\Dans le deuxième chapitre, "dg-Segal spaces", on expose le gros de nos résultats. On peut les diviser en trois parties, au niveau du contenu.
\\
\\Dans une première partie contenant la Section 1, on construit une chaîne d'adjonctions de Quillen entre la catégorie des dg-catégories et $\Fun^\S(\Free_\S\op,\sS)$, la catégorie des foncteurs simpliciaux entre les dg-catégories libres de type fini et les ensembles simpliciaux. On donnera une construction explicite de l'adjonction
$$\Re:\Fun^\S(\Free_\S\op,\sS)\rightleftharpoons \dg : \Sing.$$
\\
\\Dans une seconde partie, qui contient les Sections 2 et 3, on construira les espaces de dg-Segal. Suivant l'exemple de \cite{ComSegalSpacesREZK}, on essayera de déterminer l'image du foncteur $\Sing$ en donnant une description de ces foncteurs selon s'ils font de certains morphismes des équivalences faibles. En particulier, on dit qu'un foncteur $F$ est un \textbf{espace de dg-Segal} s'il vérifie que
\begin{enumerate}
	\item Pour tous $L, K\in\Free_\S$, $F(L\coprod K)\to F(L)\times F(K)$ est une équivalence faible.
	\item L'image de l'objet initial est un point, ou autrement dit, $F(\emptyset)\simeq *$.
	\item Soit $G$ un graphe dans $Gr(\Ch)$. Pour tout $\alpha\in Z^n(G(x,y))$, l'image de la dg-catégorie libre issue de $G(<\alpha>)$ est un pullback homotopique dans $\sS$ de la forme suivante : 
	\begin{center}
		\begin{tikzcd}
			F(L(G(<\alpha>)))\ar[r]\ar[d]& F(L(G))\ar[d]\\
			F(\D^c(1,n,1))\ar[r]& F(\D(1,n,1)).
		\end{tikzcd}
	\end{center}	
\end{enumerate}

Après, on construit une adjonction de Quillen entre $\Fun^\S(\Free_\S\op,\sS)$ et les espaces simpliciaux, et on montre qu'elle envoie les espaces de dg-Segal sur des espaces de Segal classiques. On utilise cette adjonction pour définir les espaces de dg-Segal complets en disant qu'il s'agit des espaces de dg-Segal qui sont envoyés par l'adjonction sur les espaces de Segal complets. Cela nous donne la définition suivante : $F$ est un \textbf{espace de dg-Segal complet} s'il est un espace de dg-Segal et 
$$F(k)\to F_{hoequiv} $$
est une équivalence faible. On prouve que l'image de  $\Sing$ est dans les espaces de dg-Segal complets. On construit ensuite deux structures de modèles pour $\Fun^\S(\Free_\S\op,\sS)$ de façon à ce que les espaces de dg-Segal et les espaces de dg-Segal complets, respectivement, soient leurs objets fibrants. Et finalement dans cette section, on définit des morphismes appelés des DK-équivalences et on émet l'hypothèse que les équivalences faibles entre espaces de dg-Segal sont exactement les équivalences faibles dans la structure de modèles des espaces de dg-Segal complets. 

\begin{hyp}[Hypothèse \ref{ch. 2: Re-zk}] Soit $f:F\to G$ un morphisme entre deux espaces de dg-Segal. Alors, $f$ est une DK-équivalence si et seulement si $f$ est une équivalence faible dans la structure de modèles des espaces de dg-Segal complets.
\end{hyp}

Finalement, la troisième partie de ce chapitre, qui inclut les Sections 4, 5 et 6, démontre (à l'hypothèse près) que $(\Re, \Sing)$ est bien une équivalence de Quillen dans les espaces de dg-Segal complets. Pour cela, on construit dans la Section 4 un certain type d'hyperrecouvrement, d'abord dans une catégorie de modèles générale et après dans les dg-catégories, $T_*\to T$, et on montre que la colimite homotopique d'un tel hyperrecouvrement est faiblement équivalente à l'objet d'origine, $\hocolim T_i\simeq T$.  Après, dans la Section 5 on utilise ces hyperrecouvrements pour montrer que l'adjonction est pleinement fidèle à DK-équivalence près.

\begin{teo}[Théorème \ref{ch. 2: fully faithfulness}] En supposant que l'Hypothèse \ref{ch. 2: Re-zk} est vraie, pour tout $T\in\dg$ on a $\Re_k(\Sing_k)(T)\simeq T$ et le foncteur $\Sing$ est pleinement fidèle.
\end{teo}

Dans la dernière section, Section 6, on montre que (en supposant que l'hypothèse est vraie) l'adjonction est essentiellement surjective sur les espaces de dg-Segal. Pour cela, on construit un type spécial d'hyperrecouvrements de foncteurs dans $\Fun^\S(\Free_\S\op,\sS)$ de la forme $\Sing(T_*)\to F$ et on prouve que l'image de la colimite homotopique de $T_*$ est DK-équivalente à $F$.

\begin{teo}[Théorème \ref{ch. 2: essential surjectivity}] Soit $F$ un espace de dg-Segal. Alors il existe une dg-catégorie $T$ telle qu'il existe une DK-équivalence $\Sing(T)\to F$. En conséquence, si l'Hypothèse \ref{ch. 2: Re-zk} est vraie, l'adjonction $\Re:\dgS\rightleftharpoons \dg:\Sing$ est une équivalence de Quillen.
\end{teo}

Dans le troisième chapitre, "Future Work", on explore plusieurs voies dans lesquelles on pourrait étendre le travail de cette thèse. Dans une première section, on parlera de l'hypothèse du chapitre précédent, et on expliquera quelles seraient les méthodes qu'on pourrait utiliser pour le prouver. Dans une deuxième section, on parlera de la catégorie linéaire des simplexes, $\D$. Cela est construit pour être une sorte de version linéaire de la catégorie des simplexes $\Delta$, et on pense que cela nous donnera une équivalence de Quillen entre $\Fun^\S(\Free_\S\op,\sS)$ et $\Fun^\S(\D\op,\sS)$ qui nous permettrait de définir les espaces de dg-Segal complets comme des foncteurs simpliciaux entre $\D$ et $\sS$. Finalement, on fait quelques commentaires rapides sur d'autres applications possibles de nos résultats.

\section{Notations}\label{intr-fr-notations}

Même si toutes les notations ici seront expliquées dans le chapitre suivant, on les mentionne ici pour un accès plus simple. 

\begin{itemize}
	\item On note une adjonction entre deux catégories par $F:M\rightleftharpoons N:G$, avec l'adjoint à gauche étant toujours la flèche en haut.
	\item On note la catégorie des ensembles simpliciaux par $\sS$, la catégorie des complexes de cochaînes par $\Ch$ et la catégorie des dg-catégories par $\dg$.
	\item On note par $k$ un anneau commutatif. On note le produit tensoriel sur $k$ dans $\Ch$ par $-\otimes -$ et le shift dans un complexe de cochaînes $A$ par $A[-]$.
	\item On note par $k[s]$ le complexe de cochaînes concentré en degré $s$, où il vaut $k$, et par $k^c[s]$ le complexe concentré en degrés $s$ et $s-1$, où il vaut $k$.
	\item On note par $\Fun(A,B)$ la catégorie des foncteurs entre des catégories $A$ et $B$, et par $\Fun^\S(A,B)$ la catégorie des foncteurs simpliciaux entre deux catégories simpliciales $A$ et $B$.
	\item On note par $Gr(A)$ la catégorie des graphes enrichis sur une catégorie $A$ et $Gr(\Ch)^{tf}$ la sous-catégorie pleine des graphes enrichis sur les complexes de type fini.
	\item On note par $\mathcal{L}$ la catégorie des dg-catégories libres, et par $\Free$ la sous-catégorie pleine des dg-catégories libres de type fini cofibrantes.
	
\end{itemize}

\end{otherlanguage}
 
\chapter{Introduction}

\epigraph{\textit{"There are some things that it is better to begin than to refuse, even though the end may be dark."}}{---J.R.R. Tolkien, \textit{The Lord of the Rings}}

The subject of this thesis is to define and study a new model for the homotopy theory of dg-categories, which is better behaved than the original model structure on dg-categories and also closer to the well established models of $\infty$-categories. As such, before we start, it will be useful to do a quick round of what we know about the existing models in both $\infty$-categories and dg-categories.
\\
\\We have done our best to explain the notations as they appear, but  all the main notion's notations are gathered in Section \ref{intr-en-notation} for easy access.

\section{The models of $\infty$-categories}

Let us start with higher categories. On an intuitive level, we can say that an $\infty$-category (or, to be precise, an ($\infty,1$)-category) is a category enriched over $\infty$-groupoids: a category with a set of objects, and for every two objects, an $\infty$-groupoid between them. In other words, instead of just having morphisms between objects, we ask to have $2$-morphisms between the morphisms, and $3$-morphisms between the $2$-morphisms, etc, while asking all $n$-morphisms with $n\geq 2$ to be invertible. Such a construction, also called in the literature "a homotopy theory", arises in multiple different situations. 
\\
\\A classic example in which an $\infty$-category arises from common computations is what happens when we take an interest on structures that are classified up to a notion that is less strong than the one of isomorphism (topological spaces up to homotopy equivalence, chain complexes up to weak equivalence, among others). We would want to localize the category in order to have that class of morphisms be invertible. This problem was successfully tackled by Quillen in \cite{Quillen} with the introduction of model categories; but as useful as they are, the problem with those is that model categories have a whole array of structure we need to add to make it work. Also, their homotopy category (the localization in question) fails to remember the higher order information that we had constructed. In order to solve those problems, Dwyer and Kan introduced in \cite{DK-infini-cat} a simplicial localization, which is a category enriched over simplicial sets. From there, several researchers have worked over the years in order to find a good model for such a category. Four approaches to $\infty$-categories, in particular, have been the most influential.
\\
\\The first option, of course, is simplicial categories themselves. Dwyer and Kan develop the theory of simplicial localizations in this setting, as categories enriched on simplicial sets. As this will be expanded upon in Section \ref{Simplicial categories}, we will not spend much time on it here: we will just add that Bergner proved that the category of simplicial categories has a model structure in \cite{Berg-simplicial}.
\\
\\A second model, and a very successful one at that, would be quasi-categories. First defined by Boardman and Vogt in \cite{Boardman-Vogt-quasicat}, this approach defines quasi-categories (also called "weak Kan complexes") to be simplicial sets $X$ such that every inner horn $\Lambda^k[n]\to X$, there exists a filler $\Delta[n]\to X$, for all $n\in\N$ and all $0<k<n$, effectively asking that for every two $n$-morphisms there be a third "composition" $n$-morphism and a $(n+1)$-morphism linking the two. This theory was developed further by Joyal in papers like \cite{JoyalQuasi-cat} and, of course, Lurie in \cite{HTT} and \cite{LurieHA}. Joyal and Tierney also proved that there is a model structure on it in \cite{JTquasi-Segal}.
\\
\\Our third model will be Segal categories. These are a natural generalization of simplicial categories, as they can be seen as simplicial categories with its composition written only up to homotopy. They are first defined by Dwyer, Kan and Smith in \cite{Dwyer-Kan-Smith}, and they are defined to be simplicial spaces $X:\Delta\op\to \sS$ such that $X_0$ is a discrete simplicial set and for all $k\geq 1$ the Segal map 
$$\phi_k:X_k\to \overbrace{X_1\times_{X_0}\ldots\times_{X_0}X_1}^{\text{$k$ times}} $$
is a weak equivalence. The category of Segal categories has a model structure, constructed for a general Segal $n$-category by Hirschowitz and Simpson in \cite{Simpson} (text in French) and in an alternative version by Bergner in \cite{Bergner-equivalences}.
\\
\\Lastly, the fourth and (for us) most interesting model for $\infty$-categories is that of complete Segal spaces. Defined for the first time by Rezk in \cite{ComSegalSpacesREZK}, a complete Segal space is also a simplicial space such that the Segal map $\phi_k$ is a weak equivalence for all $k\geq 1$, but instead of asking for the $X_0$ to be discrete, we ask instead that the morphism 
$$X_0\to X_{hoequiv} $$
is a weak equivalence, where $X_{hoequiv}$ is the space of homotopy equivalences. In that same paper, Rezk constructs a model structure in which the complete Segal spaces form the fibrant objects. 
\\
\\Now, we have constructed four different models, claiming that they are all models of the same thing, but that is not immediately obvious by looking at the definitions. They have things in common, yes, but they are also pretty different from each other. But that is why we have introduced the model structures on all of them: in fact, Bergner proves in \cite{Bergner-equivalences} that there is a Quillen equivalence (i.e., an equivalence in the theory of model categories) between complete Segal spaces and Segal categories, and another between Segal categories and simplicial categories. On their side, Joyal and Tierney prove in \cite{JTquasi-Segal} that quasi-categories and complete Segal spaces are also Quillen equivalent. Attention, however: even though we know that all those categories are Quillen equivalent, the Quillen adjunctions go in opposite directions, which means that they can't be composed into a single Quillen equivalence.
\\
\\There have been a few attempts to give an axiomatic definition of $\infty$-categories over the years, like Toën's  "theory of $\infty$-categories" in \cite{TOEN-infini} (text in French) and Riehl and Verity's "$\infty$-topos" in \cite{Verity-Riehl}, but we will not be talking about those in more detail in here.

\section{The models of dg-categories}

Let us now talk about dg-categories. Like simplicial categories, dg-categories are defined as being categories enriched over something: in this case, cochain complexes. As we will be discussing them in detail in Section \ref{Dg-categories}, we won't go into detail here: we're just going to comment that Tabuada proved in \cite{TAB_Fr} (text in French; see \cite{TAB} for a definition in English) that there is a model structure over $\dg$.
\\
\\Now, if $\infty$-categories already have a long and proud history, dg-categories are even older: indeed, we already find them being used on papers from the 1960s, like Kelly's 1965 paper \cite{Kelly}. In sight of that, we will restrain from procuring a full history of their evolution here: it is also not the point of this introduction. Suffice to say that although categories and $\infty$-categories are in most cases enough to work with in Algebraic Topology, Algebraic Geometry often has to contend with a notion of linearity that doesn't necessarily gel well with those concepts. In consequence, dg-categories, having an embedded notion of linearity already, have established themselves over the years as a essential tool in the field, up to the present day. A good example of that use is the Geometric Langlands Program, which was rewritten by Arinkin and Gaitsgory in \cite{Geometric-Langlands} in terms of dg-categories. For more details on that, we direct the reader towards the book \textit{A study in Derived Algebraic Geometry}, by Gaitsgory and Rozenblyum, (see \cite{Gaitsgory-Rozenblyum-Vol-I} and \cite{Gaitsgory-Rozenblyum-Vol-II}).
%where problems like the Geometric Langlands Program utilize dg-categories in an essential manner. 
\\
\\But as useful as dg-categories are, they aren't perfect. Indeed, for example, the category of dg-categories has a monoidal structure, and it has a model structure; but those two aren't compatible. That, among other things, has pushed researchers in the last few years to try and find different models of dg-categories, following the footsteps of $\infty$-categories. We have already talked about our first analogy: dg-categories would be the simplicial categories. It is interesting to notice, too, that one of the reasons that pushed people to search for alternative models for $\infty$-categories outside of simplicial categories is that, like in our case, the simplicial model structure on simplicial categories also isn't compatible with its monoidal structure (the direct product, in that context).
\\
\\Let us start with quasi-categories. We have one result by Cohn in \cite{COHN} saying that there is an equivalence between the underlying $\infty$-category of the model category of dg-categories localized at the Morita equivalences, and the $\infty$-category of small idempotent-complete $k$-linear stable quasi-categories. Unluckily for us, $k$-linear stable quasi-categories aren't really very user-friendly: for example, if we have an adjoint between two linear quasi-categories, it is very complicated to prove that there exists a linear adjoint. Also, if $x$ is an object in a $k$-linear quasi-category, proving that $\operatorname{End}(x)$ can be represented by a dg-algebra is very complicated in the world of $k$-linear quasi-categories, but it is an almost direct result of the definition on dg-categories. 
\\
\\More recently, there has been another approach to the subject: Mertens has offered in \cite{Arne-thesis} a definition of a $C$-enriched quasi-category for $C$ a cocomplete monoidal category (see definition 2.2.26 in \cite{Arne-thesis} for more details), and has constructed a functor that lifts the classical dg-nerve functor $\dg\to \sS$ into a \textit{linear} dg-nerve
$$\dg\to S_\otimes\Mod(k) $$
where $S_\otimes\Mod(k)$ is the category of templicial (meaning 'tensor simplicial'; their "$\Mod(k)$-enriched" simplicial sets) modules (see definition 2.4 in \cite{Arne} for more details). In a talk at the Institut Poincaré in September 2022 (see \cite{conf}) , Lowen mentioned that Mertens and herself are working on a Quillen equivalence to link their $\Mod(k)$-enriched quasi-categories to dg-categories, but that it is still a work in process.
\\
\\On the Segal category side of things, Bacard has defined in \cite{Bacard-enriched-Segal-cat} a notion of enriched Segal category, and consequently a notion of Segal dg-category, by saying that, fixing a set $\O$, a Segal dg-category is a $W$-colax morphism of the form
$$F:P_{\overline{\O}}\to \Ch, $$
where $\overline{\O}$ is a groupoid called the coarse category associated to $\O$ and $P_{\overline{\O}}$ is the 2-path-category of $\overline{\O}$ (see definitions 4.1 and 2.7 in \cite{Bacard-enriched-Segal-cat} for more details). In this case, too, it seems like the comparison of these objects to dg-categories is a work in progress: indeed, in the introduction of this paper Bacard says the first task now that the enriched Segal categories are defined will be to develop the homotopy theory of dg-Segal categories and compare it to our original categories.
\\
\\So we have a definition of a dg-quasi-category, and a definition of a dg-Segal category, but if we are following the same pattern as in the last section, there is one left: we have no definition of a complete dg-Segal \textit{space}. And that is where our results come in. Following the footsteps of Rezk, we will construct a "linearized version" of the complete Segal spaces and will (up to a certain unproved hypothesis) see that there is a chain of Quillen adjunctions that makes it into an equivalence of the homotopy categories.
%\newpage
\begin{center}
\begin{tabular}{|c|c|}
\hline
simplicial categories & dg-categories \\
\hline
quasi-categories & $k$-linear quasi-categories \\
                 &  dg-quasi-categories\\
\hline
Segal categories & Segal dg-categories\\
\hline
complete Segal spaces & ??\\
\hline
\end{tabular}
\end{center}

\section{dg-Segal spaces vs Segal spaces}

In Rezk's paper \cite{ComSegalSpacesREZK}, we have several important steps in order to construct his model for $\infty$-categories: first of all, of course, he defines Segal spaces as above, as a simplicial space (i.e., a functor from $\Delta\op$ to $\sS$) such that a certain amount of morphisms are weak equivalences. Then, he proves that there is a Bousfield localization from the projective structure on simplicial spaces such that Segal spaces are exactly its local objects. But that structure is not enough to prove an equivalence with $\infty$-categories; there is a certain class of morphisms that will have to be inverted still. With that in mind, he applies another Bousfield localization to his model structure, getting a model structure whose local objects will be defined as being the complete Segal spaces. Finally, he defines a class of morphisms in the category of simplicial spaces, called Dwyer-Kan equivalences, and he proves that the weak equivalences for the complete Segal model structure are exactly the Dwyer-Kan equivalences between Segal spaces. 
\\
\\In our case, we will follow a similar pattern: but it is not a straight and clear road.  Indeed, the linear structure that we now have on our categories will come up again and again to complicate things. Firstly, we will choose a certain category of functors (in our case, the slightly more complicated category of functors from free dg-categories of finite type to simplicial sets) and we define dg-Segal spaces to be the an object in that category such that a certain amount of morphisms are weak equivalences. Already, the condition on the Segal morphism won't be enough: we will need to add a condition saying that we can add a term to the complex of modules in the morphisms. 
\\
\\Once that is done, it's not complicated to prove that there is a Bousfield localization such that its local objects are exactly the dg-Segal spaces that are fibrant for the projective structure. Then, as in the case of Segal spaces, the nerve functor (that we will here call $\Sing$) will not be a Quillen equivalence for this model structure, forcing us to do a second Bousfield localization. We won't have to look very far for this: we will simply define a "forgetful" functor between dg-Segal spaces and classical Segal spaces and define complete dg-Segal spaces to be the dg-Segal spaces that become complete by the forgetful functor.
\\
\\But here we stumble upon our dear linearity again. Indeed, in order to prove that our nerve functor is a Quillen equivalence here, we would like to do as Rezk does and we define a class of morphisms that we will call DK-equivalences, such that the weak equivalences in the complete dg-Segal model structure are the DK-equivalences between dg-Segal spaces. But in order to do so, Rezk uses the direct product on simplicial spaces in order to compute a certain exponential. We don't have that privilege: as mentioned before, the monoidal structure and the model structure in $\dg$ are not compatible. In consequence, in order to prove our main result, we will first have to define a monoidal model structure on our complete dg-Segal spaces. Due to a lack of time, this hasn't been possible to achieve in this manuscript, and we leave the fact that DK-equivalences between dg-Segal spaces are the weak equivalences for the complete dg-Segal model category as a hypothesis. 
\\
\\Assuming that hypothesis to be true, though, we can effectively prove that there exists an equivalence between the homotopy category of dg-categories and the homotopy category of complete dg-Segal spaces for the complete dg-Segal model structure.

\section{Why do we care?}

Of course, if we are going to spend all this time and effort constructing these complete dg-Segal spaces, we'll need to justify our choices. Why do we want to construct these objects?
\\
\\The first reason has already been mentioned: the category of dg-categories as we know them has a monoidal structure, and also a model structure, but those two aren't compatible. Indeed, the object $\D(1,0,1)$ (i.e., the dg-category with two objects and $k$ as the complex of morphisms between the two) is cofibrant in the model category of dg-categories, but it is easy to prove that $\D(1,0,1)\otimes\D(1,0,1)$ is not. Now, the construction of the complete dg-Segal model structure itself will force us to define a monoidal structure on them that will be compatible with the model structure, solving that issue. That would probably be a big step in order to prove linear versions of classic category theorems, like the Barr-Beck theorem, which at this point in time not only cannot be deduced from its classical version, but can't even be defined properly in the linear case.
\\
\\Also, the construction of dg-categories as dg-Segal spaces will leave us with a category of functors, which are, to quote Dugger in \cite{Dugger}, "a kind of presentation by generators and relations". As it is usually the case in those types of structures, having a presentation like that makes constructing functors from the category of complete dg-Segal spaces much easier: we will just have to define it on the "generators" and make sure they take the "relations" to weak equivalences. In particular, that would give us a nicer and easier way to calculate the automorphisms in $\dg$.
\\
\\This is not a new way to go around it, either: both Toën in \cite{TOEN-infini} and Barwick and Schommer-Pries in \cite{infinity-n-cat} have utilized these "presentations by generators and relations" in order to define axiomatizations of $(\infty, 1)$-categories and $(\infty, n)$-categories respectively; and they then go on to use those to prove that the group of automorphisms of $(\infty, n)$-categories is isomorphic to the discrete group $(\Z/2\Z)^n$. It is not a coincidence that Toën's paper ends by saying that every theory of $\infty$-categories is equivalent to the category of complete Segal spaces, and not some other model.

%To give an analogy in the $\infty$-categorical context again, that is precisely the reason why Toën's axiomatic definition of a theory of $\infty$-categories in  \cite{TOEN-infini} skewes so heavily towards Segal spaces. In a similar way, complete dg-Segal spaces could potentially lead to an axiomatic definition of a "theory of dg-categories". And in any case, such a construction would be useful for the construction of any functor from dg-categories, but more concretely, we could use it to calculate the group of automorphisms, like Toën does in his paper.

\section{Organisation of the main text}

This thesis is divided in three different sections, each marked by a different chapter.
\\
\\In the first chapter, 'Background Notions', we will remind the reader (or introduce them) to the basic notions we will be working with: model categories and Quillen adjunctions, especially those of cochain complexes and of diagrams; left and right proper model categories and homotopy colimits; mapping spaces; left Bousfield localizations; simplicial objects, simplicial categories, and dg-categories. We will pay special attention to the model structure of dg-categories. This chapter also includes a section explaining the construction of Dugger's "universal model category" for a category $C$.
\\
\\In the second chapter, 'dg-Segal Spaces', we deal with the bulk of the results. We could divide it into three parts, content-wise.
\\
\\In the first part, which is exclusively section 1, we will construct a chain of Quillen adjunctions between the category of dg-categories and $\Fun^\S(\Free_\S\op,\sS)$, the category of simplicial functors between free dg-categories of finite type and simplicial sets. We will give an explicit construction of that functor,
$$\Re:\Fun^\S(\Free_\S\op,\sS)\rightleftharpoons \dg : \Sing.$$
\\
\\In the second part, which encompasses sections 2 and 3, we will discuss and construct dg-Segal spaces and complete dg-Segal spaces. Following the example of \cite{ComSegalSpacesREZK}, we will try to determine the image of the functor $\Sing$ by giving a description of its functors in terms of whether they make certain morphisms into weak equivalences. In particular, we say that a functor $F$ is \textbf{a dg-Segal space} if it satisfies

\begin{enumerate}
	\item For all $L, K\in\Free_\S$, $F(L\coprod K)\to F(L)\times F(K)$ is a weak equivalence. 
	\item The image of the initial object is a point, i.e. $F(\emptyset)\simeq *$.
	\item Let $G$ be a graph in $Gr(\Ch)$. For all $\alpha\in Z^n(G(x,y))$, the image of the free dg-category issued from $G(<\alpha>)$ is a homotopy pullback in $\sS$ of the following form:
	\begin{center}
		\begin{tikzcd}
			F(L(G(<\alpha>)))\ar[r]\ar[d]& F(L(G))\ar[d]\\
			F(\D^c(1,n,1))\ar[r]& F(\D(1,n,1)).
		\end{tikzcd}
	\end{center}	
\end{enumerate}

Then, we construct a Quillen adjunction between $\Fun^\S(\Free_\S\op,\sS)$ and simplicial spaces, and prove that it sends dg-Segal spaces to classical Segal spaces. We then use that adjunction to define complete dg-Segal spaces by saying that they are the dg-Segal spaces whose image through the linearisation functor is a complete Segal space. That gives us the following definition: $F$ is \textbf{a complete dg-Segal space} if 
$$F(k)\to F_{hoequiv} $$
is a weak equivalence. We prove that the image of the Quillen adjunction $\Sing$ is included in the complete Segal spaces. We construct two model structures for $\Fun^\S(\Free_\S\op,\sS)$  such that dg-Segal spaces and complete dg-Segal spaces, respectively, are their fibrant objects. And lastly in this section, we define a type of morphism called DK-equivalences and we hypothesize that the weak equivalences for the complete dg-Segal model structure are exactly the DK-equivalences.

\begin{hyp}[Hypothesis \ref{ch. 2: Re-zk}] Let $f:F\to G$ be a morphism between two functors satisfying the dg-Segal conditions. Then, $f$ is a DK-equivalence if and only if it is a weak equivalence in the complete dg-Segal model structure.  
\end{hyp}

Lastly, the third part of this chapter, which encompasses sections 4, 5 and 6, deals with proving that (assuming the hypothesis to be true) the $(\Re, \Sing)$ is actually a Quillen equivalence on the complete dg-Segal spaces. Firstly, in section 4 we will define and construct a certain type of hypercovers, first in a general model category and then on dg-categories, $T_*\to T$, and prove that the homotopy colimit of such a hypercover is weak equivalent to the original object, $\hocolim T_i\simeq T$. Then, in section 5 we use those hypercovers to prove that the adjunction $(\Re,\Sing)$ is fully faithful, up to DK-equivalences.

\begin{teo}[Theorem \ref{ch. 2: fully faithfulness}] Assuming Hypothesis \ref{ch. 2: Re-zk} to be true, for all $T\in\dg$ we have $\Re_k(\Sing_k)(T)\simeq T$, and the functor $\Sing_k$ is fully faithful.
\end{teo}

And lastly, in section 6 we prove that, assuming the hypothesis to be true, the adjunction is essentially surjective on the dg-Segal spaces. For that, we construct a special type of hypercovers of functors in $\Fun^\S(\Free_\S\op,\sS)$ of the form $\Sing(T_*)\to F$ and we prove that the image of the homotopy colimit of $T_*$ is weak equivalent to $F$.

\begin{teo}[Theorem \ref{ch. 2: essential surjectivity}] Let $F$ be a functor that satisfies the dg-Segal conditions. Then there exists a dg-category $T$ such that the morphism $\Sing(T)\to F$ is a DK-equivalence. So if the Hypothesis is true, the adjunction $\Re:\dgS\rightleftharpoons \dg:\Sing$ is a Quillen equivalence.
\end{teo}

In the third and last chapter, 'Future Work', we explore several paths in which the work in this thesis can be expanded on. In a first section, we will tackle the Hypothesis from the last chapter, explaining what would be the main methods in which we would prove it. Then, in the second section, we will talk about the linear simplex category $\D$. This is intended to be some kind of linearized version of the simplex category $\Delta$, and we expect it to give us some kind of Quillen equivalence between $\Fun^\S(\Free_\S\op,\sS)$ and $\Fun^\S(\D\op,\sS)$ that would allow us to define complete dg-Segal spaces directly as simplicial functors from $\D$ to $\sS$. Lastly, we make some quick comments about other possible applications of our results.

\section{Notation}\label{intr-en-notation}

Although all the following notations will be mentioned and explained in the next chapter, we add them here for easier access.

\begin{itemize}
	\item We denote an adjunction between two categories by $F:M\rightleftharpoons N:G$, with the left adjoint always being the arrow on top.
	\item We denote the category of simplicial sets by $\sS$, the category of cochain complexes by $\Ch$ and the category of dg-categories as $\dg$. 
	\item We fix $k$ to be a commutative ring. We denote the tensor product on $k$ in $\Ch$ as $-\otimes -$, and the shift on a cochain complex $A$ as $A[-]$.
	\item We denote by $k[s]$ the cochain complex concentrated in degree $s$, where it is $k$, and $k^c[s]$ to be the complex concentrated in degrees $s$ and $s-1$, where it is $k$.
	\item We denote by $\Fun(A,B)$ the category of functors between categories $A$ and $B$, and $\Fun^\S(A,B)$ the category of simplicial functors between simplicial categories $A$ and $B$.
	\item We denote by $Gr(A)$ the category of graphs enriched over a category $A$, and $Gr(\Ch)^{tf}$ the full subcategory of graphs enriched over complexes of finite type. 
	\item We denote by $\mathcal{L}$ the category of free dg-categories, and $\Free$ the full subcategory of cofibrant free dg-categories of finite type.
\end{itemize}

\chapter{Background notions}

\epigraph{\textit{"Let's start at the very beginning, \\a very good place to start"}}{---The Sound of Music, \textit{Do-Re-Mi}}

\section{Model categories}

A central tool in this work is the concept of model categories. In that vein, let us recall some results and definitions about them. As these results are all well known, unless stated otherwise the results in this section will be taken from \cite{Hovey}.

\subsection{Definitions}

The results from this section come from \cite[Sections 1.1 and 1.3]{Hovey}.

\begin{defin}Let $M$ be a category, and $f,g$ morphisms in $M$. We say that $f$ is a\textbf{ retract} of $g$ if there exists a commutative diagram of the form 
\begin{center}
	\begin{tikzcd}
		A\ar[r]\ar[d,"f"]\ar[rr,"Id_A", bend left=30]&B\ar[r]\ar[d,"g"]&A\ar[d,"f"]\\
		C\ar[r]\ar[rr, "Id_C", bend right=30]        &D\ar[r]          &C
	\end{tikzcd}
\end{center}
\end{defin}

\begin{defin}Let $M$ be a category. A \textbf{functorial factorization} is an ordered pair $(\alpha, \beta):\Mor(M)^2\to \Mor(M)^2$ of functors such that for all morphisms $f$ in $C$, we have a factorization $f=\beta(f)\circ \alpha(f)$. 
\end{defin}

\begin{defin}Let $M$ be a category and $i$ and $p$ morphisms in $M$. We say that $i$ has the \textbf{left lifting property} with respect to $p$, and that $p$ has the \textbf{right lifting property} with respect to $i$, if for every commutative square
\begin{center}
	\begin{tikzcd}
		A\ar[r]\ar[d,"i"]                  &B\ar[d,"p"]\\
		C\ar[r]\ar[ur, dotted, "\exists h"]&D         
	\end{tikzcd}
\end{center}
there exists $h$, called "a lift", which makes both triangles also commutative.
\end{defin}

\begin{defin}Let $M$ be a category. We call a \textbf{model structure} on $M$ the data consisting of three sets of morphisms in $M$, called \textbf{weak equivalences}, \textbf{fibrations} and \textbf{cofibrations}, and two functorial factorizations, $(\alpha, \beta)$ and $(\gamma, \delta)$ satisfying the following properties. Let $f$ and $g$ be morphisms in $M$. 
\begin{itemize}
	\item The 2-out-of-3 property: If we can define the composite of $f$ and $g$, and two out of $f$, $g$ and $f\circ g$ are weak equivalences, the third one is too.
	\item Closure by retracts: If $f$ is a retract of $g$ and $g$ is a weak equivalence, a fibration or a cofibration, then $f$ is too.
	\item The lifting property: We define \textbf{trivial fibrations (trivial cofibrations)} to be the morphisms that are at the same time weak equivalences and fibrations (cofibrations). Trivial cofibrations have the left lifting property with respect to fibrations, and cofibrations have the left lifting property with respect to trivial fibrations. In other words, if $i$ is a cofibration and $p$ is a fibration, and at least one of them is a weak equivalence, a lift of 
\begin{center}
	\begin{tikzcd}
		A\ar[r]\ar[d,"i"]&B\ar[d,"p"]\\
		C\ar[r]          &D         
	\end{tikzcd}
\end{center}
exists.
	\item The factorization property: the morphism $\alpha(f)$ is a cofibration, $\beta(f)$ is a trivial fibration, $\gamma(f)$ is a trivial cofibration and $\delta(f)$ is a fibration. In other words, when factorizing a morphism using the functorial factorizations, the first term is always a fibration, the second is always a cofibration, and at least one of them is a weak equivalence.
\end{itemize}
\end{defin}

\begin{defin} We call a \textbf{model category} a category which is both complete and cocomplete together with a model structure on it.
\end{defin}

\begin{rem}This definition is self-dual, which means that if $M$ is a model category, $M\op$ is also a model category, where the fibrations of $M$ are the cofibrations of $M\op$ and vice-versa. We leave to the reader to find the functorial factorizations.
\end{rem}

As we assume a model category to be complete and cocomplete, it has an initial and a final object, the limit and colimit of the empty diagram. It then gives us the following definitions:

\begin{defin}Let $M$ be a model category. We say that an object $x\in\Ob(M)$ is \textbf{fibrant} if the map from $x$ to the final object is a fibration. We say that $x$ is \textbf{cofibrant} if the map from the initial object to $x$ is a cofibration. 
\end{defin}

Using the functorial factorizations from the model structure, we see that for any object $x\in M$ in a model category, there exists a factorization of  $\emptyset\to x$ of the form $\emptyset\to Qx\to x$, where the object $Qx$ is cofibrant and $Qx\to x$ is a trivial fibration. Similarly, there exists a factorization of $x\to \ast$ of the form $x\to Rx\to \ast$ where the object $Rx$ is fibrant and $x\to Rx$ is a trivial cofibration.

\begin{defin}We call $Qx$ the \textbf{cofibrant replacement }of $x$, and $Rx$ the \textbf{fibrant replacement} of $x$. We have functors $Q,R:M\to M$, and we call them the \textbf{cofibrant (fibrant) replacement functors.}
\end{defin}

In reality, we don't need to explicitly define all the components in the model structure. In fact, we have a result that tells us that the axioms are overdetermined.

\begin{prop} Let $M$ be a model category. Then a map in $M$ is a cofibration (a trivial  cofibration) if and only if it has the left lifting property with respect to all trivial fibrations (fibrations). By duality, a map is a fibration if and only if it has the right lifting property with respect to all trivial cofibrations (cofibrations). 
\end{prop}

As a result of this, we will almost never define both the fibrations and the cofibrations, because just having one of them and the weak equivalences completely determines the other. 
\\
\\Following the tradition for defining things in category theory, now that we have our objects we need to define the morphisms between those objects. What exactly do we consider to be a morphism of model categories? We call those Quillen functors, or, equivalently, Quillen adjunctions.

\begin{defin}Let $M$ and $N$ be two model categories. 
\begin{itemize}
	\item We say that a functor $F:M\to N$ is a \textbf{left Quillen functor} if $F$ is a left adjoint and preserves cofibrations and trivial cofibrations.
	\item We say that a functor $U: N\to M$ is a \textbf{right Quillen functor} if $U$ is a right adjoint and preserves fibrations and trivial fibrations. 
	\item Let $(F,U,\phi)$ be an adjunction between $M$ and $N$. We say that $(F,U,\phi)$ is a \textbf{Quillen adjunction} or a \textbf{Quillen functor} if $F$ is a left Quillen functor. 
\end{itemize}
\end{defin}

\begin{rem}It is important to remember that neither left nor right Quillen adjunctions preserve all weak equivalences, but only the ones that are also cofibrations (left) or fibrations (right). That means that Quillen equivalences, as is, do not preserve the whole structure of a model category. There is one context in which it does work, though: we know that every left Quillen functor preserves weak equivalences between cofibrant objects, and every right Quillen functor preserves weak equivalences between fibrant objects. 
\end{rem}

\begin{rem}It might sound weird, and lacking symmetry, that we only ask for $F$ to be a left Quillen functor, and not for $U$ to be a right Quillen functor too. In fact, it is easy to prove that if we have a Quillen functor $(F,U,\phi)$, then the functor $U$ is automatically a right Quillen functor, so we don't need to.
\end{rem}

\begin{rem}\label{ch1: composition}We can of course compose left Quillen functors to form a new left Quillen functor, and the same goes for right Quillen functors. On the other hand, it isn't surprising to find that composing a left Quillen functor with a right Quillen functor doesn't necessarily give us anything useful. In particular, that means that if we have three model categories $M, N, L$ and two Quillen adjunctions $(F,U,\phi):M\to N$ and $(F',U',\psi):N\to L$ we have a Quillen adjunction between $M$ and $L$; but if we have two Quillen adjunctions $(F,U,\phi):M\to N$ and $(F',U',\psi):L\to N$ there is no reason there should be a Quillen adjunction between $M$ and $L$.
\end{rem}

\begin{rem}In practice we will seldom write a Quillen functor by giving the whole adjunction: we will usually just give one of the functors. This abuse is not of great importance, as we can prove that if $(F,U,\phi)$ is a Quillen adjunction and we have another adjunction $(F,U',\psi)$, then $(F,U',\psi)$ is also a Quillen adjunction.
\end{rem}

\begin{nota}Whenever we write the two functors that form the adjunction, we will denote them by $F:M\rightleftharpoons N: Q$, where the left adjunction is always the arrow on top.
\end{nota}

There is also a criterion that will be very useful for us later, as a way to prove something is a Quillen adjunction.
 
\begin{prop}\cite[Proposition 7.15]{JTquasi-Segal}\label{ch. 1: criteron Quillen adjunction} Let $M$ and $N$ be two model categories, and let $F:M\rightleftharpoons N:G$ be an adjunction between them. Then the adjunction is a Quillen adjunction if and only if $F$ preserves cofibrations and $G$ preserves fibrant objects.
\end{prop}

Lastly in this part, we will define what we call an equivalence in model categories. 

\begin{defin}Let $M$ and $N$ be two model categories and $(F,U,\phi)$ a Quillen adjunction. We say that $(F,U,\phi)$ is a \textbf{Quillen equivalence} if for all cofibrant objects $X$ in $M$ and all fibrant objects $Y$ in $N$, a map $f:FX\to Y$ is a weak equivalence in $N$ if and only if $\phi(f):X\to UY $ is a weak equivalence in $M$.
\end{defin}

\subsection{The homotopy category}

Historically, the main reason for the introduction of model categories was that while we can always localize a category $C$ with respect to a set of morphisms $W$, the resulting object of the Gabriel-Zisman localization is quite difficult to work with, and the construction forces us to consider higher universes. Luckily, we can prove that, if we can build a model category with the set $W$ as its weak equivalences, the localized category is equivalent to a much better behaved and smaller category. The results from this section come from \cite[Section 1.2]{Hovey} unless stated otherwise.

\begin{defin}(Gabriel-Zisman localization) Let $C$ be a category and $W$ a set of morphisms in $C$. We call the \textbf{localization of $C$ with respect to $W$} a "category" (barring size issues) $C\l W^{-1}\r$ coupled with a functor $l:C\to C\l W^{-1}\r$ that sends all morphisms in $W$ to isomorphisms in $C\l W^{-1}\r$, with the following universal property. Let $D$ be a category coupled with a functor $f:C\to D$ such that $f$ sends all morphisms in $W$ to isomorphisms in $D$; then there is a unique functor $f':C\l W^{-1}\r\to D$ such that $f'\circ l=f$.   
\end{defin}

\begin{rem}We have taken this definition from \cite[Section 2.1]{Toen-dg}. A curious reader is encouraged to go read \cite[Def. 1.2.1]{Hovey} for an explicit construction of the aforementioned localization.
\end{rem}

\begin{nota}We will sometimes denote the morphisms between objects in the homotopy category of $M$ by $\l X,Y \r_{M}$, or just $\l X,Y\r$ if the base model category is clear.
\end{nota}

\begin{nota}If we take $C$ to be a model category and $W$ its set of weak equivalences, then we will denote $C\l W^{-1}\r$ by $\Ho(C)$ and call it \textbf{the homotopy category of $C$}.
\end{nota}

And now that we have defined the homotopy category of $M$ a model category, it is time to start building that equivalent version that we promised was so much easier to use. For that, we will start by setting a notation.

\begin{nota}let $M$ be a model category. Then,
\begin{itemize}
	\item We denote the full subcategory of cofibrant objects by $M_c$.
	\item We denote the full subcategory of fibrant objects by $M_f$.
	\item We denote the full subcategory of objects who are both fibrant and cofibrant by $M_{cf}$.
\end{itemize}
\end{nota}

\begin{prop}\label{ch 1:equiv}Let $M$ be a model category. Then the inclusion functors induce equivalences of categories $\Ho (M_{cf})\to \Ho(M_c)\to \Ho (M)$ and $\Ho(M_{cf})\to \Ho(M_f)\to \Ho(M)$.
\end{prop}

So if we find a way to define $\Ho(M_{cf})$ that is way easier to understand (and also a category in this universe), we will be able to use it as an equivalent category to $\Ho(M)$.

\begin{defin}Let $M$ be a model category.
\begin{itemize}
	\item Let $B$ be an object in $M$. We call $B'$ a \textbf{cylinder object for $B$} if we have a factorization of the fold map $B\coprod B\to B$ into a cofibration $i_0+i_1: B\coprod B\to B'$ and a weak equivalence $s:B'\to B$.
	\item Let $X$ be an object in $M$. We call $X'$ a \textbf{path object for $X$} if we have a factorization of the diagonal map $X\to X\times X$ into a weak equivalence $X\to X'$ and fibration $(p_0,p_1):X'\to X\times X$.
\end{itemize}
\end{defin}

\begin{rem}It should be evident to see that such a factorization always exists: indeed, we just need to apply the functorial factorizations we defined earlier. In that case, the morphism $B'\to B$ is also a trivial fibration and $X\to X'$ is a trivial cofibration.
\end{rem}

\begin{defin}Let $M$ be a model category, $f,g:B\to X$ two morphisms in $M$. 
\begin{itemize}
	\item We call a \textbf{left homotopy from $f$ to $g$} a map $H:B'\to X$ where $B'$ is a cylinder object for $B$, such that $Hi_0=f$ and $Hi_1=g$. If such a left homotopy exists, then we say that $f$ and $g$ are \textbf{left homotopic} and we denote it by $f\sim^l g$.
	\item We call a \textbf{right homotopy from $f$ to $g$} a map $H:B\to X'$ where $X'$ is a path object for $X$, such that $p_0H=f$ and $p_1H=g$. If such a right homotopy exists, then we say that $f$ and $g$ are \textbf{right homotopic} and we denote it by $f\sim^r g$.
\end{itemize}
\end{defin}

\begin{rem}The reader will probably recognize the homotopies used in Algebraic Topology to define homotopy groups, and it is no coincidence: indeed, with the construction of the model category of topological spaces that we will provide later, the definitions of right and left homotopy both coincide with each other and with the classical definition.
\end{rem}

\begin{defin}Let $M$ be a model category, $f,g:B\to X$ two morphisms in $M$. We say that $f$ and $g$ are \textbf{homotopic} if they are both left and right homotopic, and we denote it by $f\sim g$.
\end{defin}

\begin{defin}Let $M$ be a model category, $f:B\to X$ a morphism in $M$. We say that $f$ is a \textbf{homotopy equivalence} if there exists another morphism $g:X\to B$ such that $f\circ g\sim \Id_X$ and $g\circ f\sim \Id_B$. If such a homotopy equivalence exists, we say that $B$ and $X$ are \textbf{homotopy equivalent}. 
\end{defin}

At this stage, people familiar with Algebraic Topology probably know where this is going: we are going to try and make a quotient category by taking the equivalence relation given by the homotopy equivalences. And we wish we could do that directly, but it doesn't work. Indeed, in general we have that left homotopies are compatible with composition, but only on the left; right homotopies are compatible with composition, but only on the right. In the same vein, these relations make equivalence relations only if $B$ is cofibrant (for the left homotopy) or $X$ is fibrant (for the right homotopy). We will have, then, to add a few conditions to our objects so that the relations have the correct properties.

\begin{prop}Let $M$ be a model category, $B$ a cofibrant object of $M$ and $X$ a fibrant object of $M$. Then the definitions of left and right homotopy coincide on $\Hom(B,X)$ and are equivalence relations there. In particular, the homotopy relation is an equivalence relation on $M_{cf}$ and it's compatible with composition. Hence the quotient category $M_{cf}/\sim$ exists.
\end{prop}

This is the category we would like to prove is equivalent to $\Ho(M)$. For now, we know that the functor $M\to M_{cf}/\sim$ sends homotopy equivalences to isomorphisms. Does it also send weak equivalences to isomorphisms?

\begin{prop}Let $M$ be a model category. Then a map in $M_{cf}$ is an homotopy equivalence if and only if it is a weak equivalence.
\end{prop}

\begin{coro}Let $M$ be a model category. Then the category $M_{cf}/\sim$ coupled with the functor $L:~M_{cf}\to M_{cf}/\sim$ satisfies the following universal property. Let $D$ be a category with a functor $f:M_{cf}\to D$ such that $f$ sends all weak equivalences in $M_{cf}$ to isomorphisms in $D$; then there is a unique functor $f':M_{cf}/\sim\to D$ such that $f'\circ L=f$.   
\end{coro}

We have proven, now, that the category $M_{cf}/\sim$ is a localization of $M_{cf}$ with respect to the weak equivalences. We need, though, to see that $M_{cf}/\sim$ is actually equivalent to the localization of $M$. But that is a consequence of Proposition \ref{ch 1:equiv}: we have proven that $\Ho(M_{cf})$ is equivalent to $\Ho(M)$. We summarize the entirety of this in a single theorem.

\begin{teo}Let $M$ be a model category, $\Ho(M)$ the localization of $M$ with respect to the weak equivalences and $l:M\to \Ho(M)$ the localization functor. We remind that $R$ is the fibrant replacement functor and $Q$ is the cofibrant replacement functor.
\begin{itemize}
	\item The inclusion functor $i:M_{cf}\to M$ induces an equivalence of categories
	$$(M_{cf}/\sim)\cong \Ho(M_{cf})\to \Ho(M). $$
	\item There are natural isomorphisms
	$$(\Hom_M(QRX,QRY)/\sim)\cong \Hom_{\Ho M}(l(X),l(Y))\cong (\Hom_M(RQX,RQY)/\sim). $$
	In particular, $\Ho(M)$ is a category without changing the universe.
	\item The localization functor $l:M\to \Ho(M)$ sends right and left homotopy maps to isomorphisms.
	\item If $f:A\to B$ is a map in $M$ such that $l(f)$ is an isomorphism in $\Ho (M)$, then $f$ is a weak equivalence.
\end{itemize}
\end{teo}

Lastly, now that we have a useful definition of the homotopy category of a model category $M$, we can ask the question of how we go from functors between model categories to functors between the induced homotopy categories. 

\begin{defin}Let $M$ and $N$ be two model categories, $F:M\to N$ a left Quillen functor and $U:N\to M$ a right Quillen functor.
\begin{itemize}
	\item We define the \textbf{total left derived functor $\L F:\Ho(M)\to \Ho(N)$} to be the composition $$\L F~=\Ho(F)~\circ ~\Ho(Q):\Ho(M)\to \Ho(M_{c})\to \Ho(N).$$
	\item We define the \textbf{total right derived functor $\R U:\Ho(N)\to \Ho(M)$} to be the composition $$\R U~=\Ho(U)~\circ~\Ho(R):\Ho(N)\to\Ho(N_f)\to\Ho(M)$$
\end{itemize}
\end{defin}

\begin{nota}As functors are almost never left and right Quillen functors for the same model structure, we usually don't specify whether the total derived functor is left or right. We call the act of going from a Quillen functor to a total derived functor to "derive" a functor.
\end{nota}

\begin{rem}As we can see, we don't really need $F$ and $U$ to be Quillen functors for the definition to work: any functor that preserves weak equivalences between cofibrant objects can be derived to a total left derived functor, and any functor that preserves weak equivalences between fibrant objects can be derived to a total right derived functor.
\end{rem}

On the other hand, we defined alongside the Quillen adjunctions another type of functor: the Quillen equivalences. We would love if two model categories that are Quillen equivalent would be, conceptually, "the same" model category, in the same way that two categories that are equivalent are conceptually "the same". Sadly, they are not, not exactly. It all boils down, again, to the homotopy categories.

\begin{prop}Let $M$ and $N$ be two model categories, and $F:M\rightleftharpoons N:U$ a Quillen adjunction. Then the following are equivalent:
\begin{enumerate}
	\item $(F,U,\phi)$ is a Quillen equivalence.
	\item For all $X$ cofibrant, the composite $X\to UFX\to URFX$ is a weak equivalence, and for all fibrant $Y$ the composite $FQUY\to FUY\to Y$ is also a weak equivalence.
	\item The adjunction $(\L F,\R U, R\phi)$ is an adjoint equivalence of categories.
\end{enumerate}
\end{prop}

So two model categories are Quillen equivalent if and only if their homotopy categories are equivalent.

\section{Constructions on model categories}

We now have the main definitions on model categories. But how and for what can we use them? In this section we will construct a few structures that will help us afterwards, either to get model structures for our favourite categories or to work with those favourite model categories.

\subsection{Cofibrantly generated model categories}

Proving that a category admits a model structure is quite hard. To compensate for it, we have a concept that will help us construct model categories in a "easier" way, that of cofibrantly generated model categories.. But before, we will need a few definitions. The results from this part come from \cite[Ch. 2, Ch.3]{Hovey}, unless stated otherwise.

\begin{defin}Let $C$ be a category with all small colimits, and $\lambda$ an ordinal. A \textbf{$\lambda$-sequence} is a colimit-preserving functor $X:\lambda\to C$. 
\end{defin}

As a $\lambda$-sequence preserves colimits, we have a isomorphism $\colim_{\beta<\gamma}X_\beta\to X_\gamma$ for all limit ordinals $\gamma<\lambda$. 

\begin{defin}Let $C$ be a category with all small colimits, $\lambda$ an ordinal and $X:\lambda\to C$ a $\lambda$-sequence. Then we call the map $X_0\to  \colim_{\beta<\lambda}X_\beta$ the \textbf{transfinite composition} of the $\lambda$-sequence.
\end{defin}

\begin{defin}Let $\kappa$ be a cardinal. We say that an ordinal $\alpha$ is \textbf{$\kappa$-filtered} if it is a limit ordinal and, if $A\subset \alpha$ and $|A|<\kappa$, then $\sup A<\alpha$. 
\end{defin}

\begin{defin}Let $C$ be a category with all small colimits, $I$ a collection of morphisms of $C$, $A$ an object of $C$ and $\kappa$ a cardinal. We say that $A$ is \textbf{$\kappa$-small relative to $I$} if, for all $\kappa$-filtered ordinals $\alpha$ and all $\alpha$-sequences $X$ such that every $X_\beta\to X_{\beta+1}$ is in $I$ for $\beta+1<\alpha$, the map 
$$\colim_{\beta<\alpha}\Hom_C(A,X_{\beta})\to \Hom_C(A,\colim_{\beta<\alpha}X_\beta) $$
is an isomorphism. We say that $A$ is \textbf{small relative to $I$ }if it is $\kappa$-small relative to $I$ for some cardinal $\kappa$. 
\end{defin}

\begin{defin}Let $C$ be a category and $I$ a set of maps in $C$ containing all small colimits. Then a \textbf{relative $I$-cell complex} is a transfinite composition of pushouts of elements in $I$. We denote the collection of relative $I$-cell complexes by $I$-cell.
\end{defin}

\begin{defin}Let $C$ be a category and $I$ be a set of maps in $C$. 
\begin{itemize}
	\item A map is \textbf{$I$-injective} if it has the right lifting property with respect to every map in $I$. We will denote the set of $I$-injective maps by $I-inj$.
	\item A map is \textbf{$I$-projective} if it has the left lifting property with respect to every map in $I$. We will denote the set of $I$-projective maps by $I-proj$.
	\item A map is a \textbf{$I$-cofibration} if it has the left lifting property with respect to every map in $I-inj$, i.e. if it is in $(I-inj)-proj$. We will denote the set of $I$-cofibrations by $I-cof$. 
	\item A map is a \textbf{$I$-fibration} if it has the right lifting property with respect ot every map in $I-proj$, i.e. if it is in $(I-proj)-inj$. We will denote the set of $I$-fibrations by $I-fib$.
\end{itemize}
\end{defin}

So we finally have the definitions necessary to give that tool we wanted to use to construct model categories.

\begin{defin}Let $M$ be a model category. We say that $M$ is a \textbf{cofibrantly generated model category} if we have two sets of maps, $I$ and $J$, such that 
\begin{itemize}
	\item The domains of the maps in $I$ are small relative to $I$-cell.
	\item The domains of the maps in $J$ are small relative to $J$-cell.
	\item The fibrations are the maps that have the right lifting property with respect to every map in $J$, i.e. the set of fibrations is $J-inj$.
	\item The trivial fibrations are the maps that have the right lifting property with respect to every map in $I$, i.e. the set of trivial cofibrations is $I-inj$.
\end{itemize}
We call $I$ the set of \textbf{generating cofibrations}, and $J$ the set of \textbf{generating trivial cofibrations}. 
\end{defin}

As we can see, cofibrantly generated model categories are, in particular, model categories. In fact, most model categories we work with in "real life" are of this type. As it is, we can give a way of proving something is cofibrantly generated model category without going through the main definition. It isn't easy to prove these conditions either by any measure, but they are less complicated than the other option.

\begin{teo}Let $C$ be a category with all small limits and colimits. Let $W$, $I$ and $J$ be sets of maps in $C$. Then there is a cofibrantly generated model structure on $C$ which has $W$ as its weak equivalences, $I$ as its set of generating cofibrations and $J$ as its set of generating trivial cofibrations, if and only if the following conditions are satisfied:
\begin{itemize}
	\item The set $W$ has the 2-out-of-3 property.
	\item The set $W$ is closed under retracts.
	\item The domains of the maps in $I$ are small relative to $I$-cell.
	\item The domains of the maps in $J$ are small relative to $J$-cell.
	\item We have the inclusion $J-cof\subseteq (W\cap I-cof)$.
	\item We have the inclusion $I-inj\subseteq (W\cap J-inj)$. 
	\item At least one of these inclusions is an equality.
\end{itemize}
\end{teo}

There are a few examples of model categories that we will be using a lot during this thesis: let us give the description of their model structures now. 
\\
\\Let us start with topological spaces. We will write $S^{n-1}$ the unit sphere of dimension $n$ and $D^n$ the unit disk of dimension $n$. 

\begin{defin}Let $\Top$ be the category of topological spaces and let $f:X\to Y$ be a morphism in $\Top$. 
\begin{itemize}
	\item We say that $f$ is a weak equivalence in $\Top$ if for all $x\in X$ and all $n\in \N$ the induced morphism of groups 
	$$\pi_n(f):\pi_n(X,x)\to \pi_n(Y,f(x)) $$
	is an isomorphism of groups. 
	\item We define $I$ to be the set of boundary inclusions $f:S^{n-1}\to D^n$ for all $n\in\N$.
	\item We define $J$ to be the set of all inclusions $f:D^n\to D^n\times \l0,1\r$ such that $f(x)=(x,0)$ for all $x\in X$ and $n\in \N$. 
\end{itemize}
\end{defin}

\begin{teo}There is a cofibrantly generated model structure in $\Top$ with the weak equivalences as stated above, $I$ the generating cofibrations and $J$ the generating trivial cofibrations. Moreover, every object in $\Top$ is fibrant.
\end{teo}

Now that we have topological spaces, we can get the next model category, which will be important for what follows: $\sS$, the category of simplicial sets. For that, though, we are going to need a couple of results first. 

\begin{defin}
    Let $n\in\N$ be natural number. We define a simplicial set $\Delta_n:\Delta\op\to \Set$ to be
    $$ \Delta_n([k])=\Hom_\Delta([k],[n])$$
    for all $[k]\in\Delta\op$.
\end{defin}

\begin{prop}\label{ch.1: adjunctions sSet}
    Let $C$ be a category with small colimits. Then the category of functors from $\Delta$ to $C$, $\Fun(\Delta, C)$, is equivalent to the category of adjunctions $\sS\rightleftharpoons C$. 
\end{prop}

In particular, this result means that if we want to construct an adjunction between simplicial sets and another category $C$, it suffices to construct a cosimplicial object in $C$. We are going to use that in order to construct an adjunction between simplicial spaces and topological spaces.

\begin{defin}
    There exists an adjunction between topological spaces and simplicial sets. We will denote by $\Re: \sS\rightleftharpoons \Top: \Sing$. We call $\Re$ the \textbf{geometric realization} and $\Sing$ the \textbf{singular functor}. 
\end{defin}

\begin{proof}[Sketch of construction]
    Fix $n\in\N$ a natural number. We define $\Re(\Delta_n)$ to be the convex hull of $(e_0,\ldots,e_n)\in\R^{n}$, where $e_0=(0,\ldots,0)$ and for all $1\leq i \leq n$ $e_i$ is the vector with i-th coordinate $1$ and all others $0$. In other words, 
    $$\Re(\Delta_n)=\{(x_1,\ldots,x_n)\in\R^n/\ \forall 1\leq i\leq n,\ t_i\geq 0,\ \sum t_i\leq 1\}.$$
    We then have a comsimplicial topological space $\Re(\Delta_*)$. By Proposition \ref{ch.1: adjunctions sSet}, we have an adjunction $\Re: \sS\rightleftharpoons \Top: \Sing$ and we have finished our construction.
\\\end{proof}

\begin{defin}Let $\sS$ be the category of simplicial sets. 
\begin{itemize}
	\item We say that a morphism $f:X\to Y$ of simplicial sets is a weak equivalence if and only if its geometric realization $\Re(f)$ is a weak equivalence for the model structure we have defined on $\Top$. 
	\item We define $I$ to be the inclusions $\partial \Delta_n\to \Delta_n$ for all $n\in\N$. 
	\item We define $J$ to be the set of horn inclusions $\Lambda^r_n\to \Delta_n$ for all $n\in\N$ and all $0\leq r\leq n$.
\end{itemize}
\end{defin} 

\begin{teo}There is a cofibrantly generated model structure in $\sS$ with the weak equivalences as stated above, $I$ the generating cofibrations and $J$ the generating trivial cofibrations. Moreover, every object in $\sS$ is cofibrant.
\end{teo}

Finally, we are going to give the model structure of the category $\Ch$ of cochain complexes over a commutative ring $k$. We define $k\l n\r$ to be the cochain complex concentrated in degree $n$, where it takes value $k$, and we also define $k^c\l n\r$ to be the complex concentrated in degrees $n$ and $n-1$, where it takes value $k$.

\begin{defin}Let $\Ch$ be the category of cochain complexes over $k$. 
\begin{itemize}
	\item We say that $f:X\to Y$ is a weak equivalence in $\Ch$ if for all $x\in\N$ the induced morphism of cohomology groups
	$$H^n(f):H^n(X)\to H^n(Y) $$
	is an isomorphism of groups. 
	\item We define $I$ to be the inclusions $k\l n-1\r\to k^c\l n\r$ for all $n\in\N$.
	\item We define $J$ to be the inclusions $0\to k^c\l n\r$ for all $n\in\N$, where $0$ is the complex which is zero everywhere.
\end{itemize} 
\end{defin}

\begin{teo} There is a cofibrantly generated model structure in $\Ch$ with the weak equivalences as stated above, $I$ the generating cofibrations and $J$ the generating trivial cofibrations. Moreover, every object in $\Ch$ is fibrant.
\end{teo}

A particularly interesting model category is that of diagrams (or presheafs) over a model category. As it will be instrumental to the main results of this thesis, we're going to give it some attention. From here on, the results from this section are taken from \cite[Ch. 11]{HLocal}.

\begin{defin}Let $C$ and $M$ be two categories. We define the \textbf{category of $C$-diagrams in $M$}, and we will denote it by $\Fun(C,M)$, the category with the following data:
\begin{itemize}
	\item A set of objects consisting on the functors from $C$ to $M$, i.e. the functors $F:C\to M$. 
	\item For every two objects $F, Q:C\to M\in \Ob(\Fun(C,M))$, a set $\Hom(F,Q)$ of the natural transformations between $F$ and $Q$. 
\end{itemize} 
\end{defin}

\begin{nota}In the case where we work with the dual of $C$, we will call $\Fun(C\op, M)$ the \textbf{category of $C$-presheafs on $M$}. In particular, if we have $M=\sS$, we'll say the \textbf{category of simplicial presheafs}.
\end{nota}

Is it necessary to prove each time that there exists a model structure for every category of diagrams? Thankfully, no. We can prove that if the category $M$ has a cofibrantly generated model structure, then the category of diagrams also has a model structure (in fact, it has two!)

\begin{defin}Let $C$ be a category and $M$ be a cocomplete category. Let $I$ be a set of morphisms in $M$. We will denote by $F_I^C$ the set of maps in $\Fun(C,M)$ of the form 
$$ \coprod_{C(\alpha, \cdot)}A=F_A^\alpha\to F_B^\alpha=\coprod_{C(\alpha, \cdot)}B,$$
where $A\to B$ is an element of $I$ and $\alpha\in \Ob(C)$.
\end{defin}

\begin{teo}\label{ch. 1:projective}Let $C$ be a small category and $M$ a cofibrantly generated model category with $I$ the set of generating cofibrations and $J$ the set of generating trivial cofibrations. Then the category of $C$-diagrams in $M$ is a cofibrantly generated model category with $F_I^C$ as its generating cofibrations and $F_J^C$ as its generating trivial cofibrations. In this model category, we have that 
\begin{itemize}
	\item A morphism $f:F\to Q$ is a weak equivalence in $\Fun (C,M)$ if it is objectwise a weak equivalence in $M$, i.e. if for all $\alpha\in C$ the morphism $f(\alpha)=F(\alpha)\to Q(\alpha)$ is a weak equivalence in $M$.
	\item A morphism $f:F\to Q$ is a fibration in $\Fun(C,M)$ if it is objectwise a fibration in $M$, i.e. if for all $\alpha\in C$ the morphism $f(\alpha)=F(\alpha)\to Q(\alpha)$ is a fibration in $M$.
\end{itemize}
We call this model structure the \textbf{projective model structure on $\Fun(C,M)$}.
\end{teo}

We have a model structure, and it's pretty easy: we just need to look at the morphisms objectwise. But what about the cofibrations? Are they objectwise too? Yes, but in this case we don't have an equivalence.

\begin{prop}Let $C$ be a small category and $M$ a cofibrantly generated model category. Then a cofibration in the projective model structure on $\Fun(C,M)$ is also objectwise a cofibration in $M$.
\end{prop}

\begin{rem}It is important to remember that in the projective model structure all cofibrations are objectwise cofibrations, but the other implication isn't necessarily true: just because a morphism is objectwise a cofibration doesn't mean it is a cofibration in $\Fun(C,M)$. There is another model structure for the category of diagrams, called the \textbf{injective model structure}, which is defined as having objectwise weak equivalences as weak equivalences and objectwise cofibrations as cofibrations, but in that case not all objectwise fibrations are fibrations. The injective model structure is somewhat less commonly used, as the lack of symmetry in the definition of cofibrantly generated model categories means that we need more conditions for it to exist. It does exist in the most common exemples, though.
\end{rem}

\begin{prop}\label{ch.1: pull adjunction}Let $C,M, N$ be three model categories and let $f:C\to M$ be a functor between $C$ and $M$. Then we have a Quillen adjunction 
$$f_!:\Fun(C, N)\rightleftharpoons \Fun (M,N): f^* $$
is a Quillen adjunction, where $f^*:\Fun (M,N)\to \Fun (C,N)$ is given by precomposing the morphisms $M\to N$ by $f:C\to M$.
\end{prop}

\subsection{Proper model categories and homotopy colimits}

Two important tools in category theory are pullbacks and pushouts; we would want to keep using them when we're dealing with model categories. The problem is that, in general, the pushout of a weak equivalence is not a weak equivalence. There is a certain class of model categories where weak equivalences can actually be pushed and pulled without losing their characteristics, though: proper categories. Also, we are working up to homotopy: it would be interesting to define limits and colimits only up to homotopy, and to see the links between those and classical limits and colimits. The results of this section come from \cite[Chapter 13]{HLocal} unless stated otherwise.

\begin{defin}Let $M$ be a model category, and let $\mathfrak{A}$ be a commutative square of the form
\begin{center}
	\begin{tikzcd}
		A\ar[r, "g"]\ar[d, "f"]&B\ar[d, "h"]\\
		C\ar[r, "k"]           &D	
	\end{tikzcd}
\end{center} 
\begin{itemize}
	\item We say that $M$ is \textbf{left proper} if every pushout of a weak equivalence along a cofibration is a weak equivalence. In other words, if $\mathfrak{A}$ is a pushout, $g$ is a cofibration and $f$ is a weak equivalence, then $h$ is also a weak equivalence.
	\item We say that $M$ is \textbf{right proper} if every pullback of a weak equivalence along a fibration is a weak equivalence. In other words, if $\mathfrak{A}$ is a pullback, $k$ is a fibration and $h$ is a weak equivalence, then $f$ is also a weak equivalence.
	\item We say that $M$ is \textbf{proper} if it is both left and right proper.
\end{itemize}
\end{defin}

\begin{rem}As the reader can see from the definition, even if a model category is proper, we still don't have that weak equivalences are preserved by pushouts and pullback along all morphisms: just cofibrations and fibrations. This is enough for our purposes, though.
\end{rem}

And how do we know that a model category is left or right proper? Well, in the cases where we know that all objects are fibrant/cofibrant, it is actually quite simple. 

\begin{prop}Let $M$ be a model category. 
\begin{itemize}
	\item Every pushout of a weak equivalence between cofibrant objects along a cofibration is a weak equivalence.
	\item Every pullback of a weak equivalence between fibrant objects along a fibration is a weak equivalence.
\end{itemize}
\end{prop}

\begin{coro}Let $M$ be a model category. 
\begin{itemize}
	\item If every object in $M$ is cofibrant, the model category $M$ is left proper.
	\item If every object in $M$ is fibrant, the model category $M$ is right proper.
	\item If every object in $M$ is both cofibrant and fibrant, the model category $M$ is proper.
\end{itemize}
\end{coro}

This is enough to give us that the categories $\Top$ and $\Ch$ are right proper, and that $\sS$ is left proper. But we actually have more than that.

\begin{prop}The model categories $\Top$, $\Ch$ and $\sS$ are proper.
\end{prop}

But that's not all: we can also find conditions for the category of presheaves over a category.

\begin{prop}Let $M$ be a cofibrantly generated model category and $C$ be a small category. Then if $M$ is left or right proper, the functor category $\Fun(C,M)$ is also left or right proper, respectively.
\end{prop}

Let us define now the homotopy pullback. We won't define the homotopy pushout, because we won't need it, but the construction is strictly dual. We remind the reader that if $M$ has a model structure, it includes not only fibrations, cofibrations and weak equivalences, but also two functorial factorizations. In particular, a functorial factorization $(\gamma, \delta)$ such that if $f$ is a morphism in $M$, then $\gamma(f)$ is a trivial cofibration and $\delta(f)$ is a fibration.

\begin{defin}Let $M$ be a right proper model category, and a diagram $X\xrightarrow{f}Z\xleftarrow{g}Y$. We define the \textbf{homotopy pullback} of the diagram as the pullback of the associated diagram $X'\xrightarrow{\delta(f)}Z\xleftarrow{\delta(g)}Y'$, and we denote it by $X\times^h_Z Y$.
\end{defin}

\begin{rem}We have defined the homotopy pushout only in the context of right proper model categories, as all categories we will be talking about are proper; but the definition exists even when the category $M$ is not right proper. In that case, we would need to find fibrant replacements for all the objects involved before we dis the construction from the definition.
\end{rem}

At first glance, there is no reason why this definition would be invariant by weak equivalences; but it is.

\begin{prop}Let $M$ be a right proper model category and a diagram
\begin{center}
	\begin{tikzcd}
		A\ar[r, "g"]\ar[d, "w_A"]&B\ar[d, "w_B"]&C\ar[l, "f"]\ar[d, "w_C"]\\
		A'\ar[r, "g'"]           &B'            &C'\ar[l, "f'"]
	\end{tikzcd}
\end{center} 
where $w_A, w_B, w_C$ are weak equivalences. Then the induced map between homotopy pullbacks 
$$A\times^h_B C\to A'\times^h_{B'} C' $$
is a weak equivalence.
\end{prop}

So we have a description of the homotopy pullback: but calculating the functorial factorizations of morphisms can be hard. Luckily, we have a case where we can skip that step.

\begin{prop}Let $M$ be a right proper model category and a diagram $X\xrightarrow{f}Z\xleftarrow{g}Y$. If either $f$ or $g$ is a fibration, then the homotopy pullback $X\times^h_Z Y$ is naturally weakly equivalent to the pullback $X\times_Z Y$.
\end{prop}

\begin{prop}Let $M$ be a right proper model category, a morphism $h:Y\to Z$ and two morphisms $f,g:X\to Z$ that are left or right homotopic. Then the homotopy pullback of $X\xrightarrow{f}Z\leftarrow Y$ and the one of $X\xrightarrow{g}Z\leftarrow Y$ are weak equivalent.
\end{prop}

And lastly, we will define the homotopy fiber of a morphism over a point. We recall that if we have a morphism $f$ and a point $z:* \to Z$, the fiber of $f$ over $z$ is a pullback of $f$ along $z$. We will now do the same thing with homotopy pullbacks.

\begin{defin}Let $M$ be a model category, $f:X\to Z$ a morphism in $M$ and $z:*\to Z$ a point in $M$. We call the \textbf{homotopy fiber of $f$ over $z$} a fibrant object in $M$ which is weakly equivalent to the homotopy pullback of $f$ along $z$. 
\end{defin}

\begin{rem}We haven't defined the homotopy fiber directly as the homotopy pullback of $f$ along a point because such an object isn't necessarily fibrant; but as in $\Top$ and $\Ch$ all objects are fibrant, in both those categories we could compute the homotopy fiber just as the homotopy pullback.
\end{rem}

And we end with a couple of corollaries that are trivial after the results we listed about homotopy pullbacks.

\begin{coro}\label{ch. 1:fiber}Let $M$ be a right proper model category. If $f:X\to Z$ is a fibration and $z:*\to Z$ is a point in $M$, then there exists a natural weak equivalence between the fiber of $f$ over $z$ and the homotopy fiber of $f$ over $z$.
\end{coro}

\begin{coro}Let $M$ be a right proper model category and $f:X\to Z$ a morphism in $M$. If we have two points $z:*\to Z$ and $z':*\to Z$ that are either left or right homotopic, then the homotopy fiber of $f$ over $z$ is weakly equivalent to the homotopy fiber of $f$ over $z'$. 
\end{coro}

\begin{nota}In particular, this last result means that the homotopy fiber doesn't depend on the choice of the point over which we take it, as long as those points are homotopic. By abuse of notation, we will call them "the homotopy fiber of $f$", without specifying the point.
\end{nota}

Finally, we have talked about homotopy pullbacks; there is another, related concept we will need to adress. That is the concept of homotopy colimits. The pullbacks we have defined before are actually a particular case of homotopy \textit{limits}, the dual concept. From here on, all results from this section come from \cite[Appendix A.2.8]{HTT}.

\begin{defin}Let $M$ be a model category and let $f:C\to C'$ be a functor between small categories. Let $f_!:\Fun(C,M)\rightleftharpoons\Fun(C',M)$ be the Quillen adjunction defined in Proposition \ref{ch.1: pull adjunction}. We call its total left derived functor $\L f_!$ \textbf{the homotopy left Kan extension of $f$}.
 \end{defin}
 
 \begin{defin}Let $M$ be a model category and $f:C\to *$ a functor from a small category $C$ to the terminal object in the category of small categories. Then we call the homotopy left Kan extension of $f$ \textbf{the homotopy colimit functor},
 $$\L f_!:\Ho(\Fun(C,M))\to \Ho(\Fun(*,M)). $$
 \end{defin}
 
\begin{defin}We remind the reader that taking an element in $\Fun(*,M)$ is the same as taking an object in $M$; so by this process, we take a functor $\phi:C\to M$ and we get an object in $M$ by applying the homotopy colimit functor. We will call $\L f_!(\phi)$ \textbf{the homotopy colimit of $\phi$}.
\end{defin}

\subsection{(Co)simplicial framings and mapping spaces}

In this thesis we will be working with simplicial spaces, and that means we need a way of getting (co)simplicial objects out of the objects in our category. How do we do that? Using (co)simplicial framings. Once again, the results in this section come from \cite[Chapter 5]{Hovey} and \cite[Chapter 16]{HLocal} unless stated otherwise.

\begin{defin} Let $\Delta$ be the simplex category, i.e. the category of finite ordinals and weakly monotone functions. We define $\Delta^+$ to be the subcategory of injective order-preserving maps; and $\Delta^-$ the subcategory of surjective order-preserving maps. Let $\l n\r$ be an object in $\Delta$. 
\begin{itemize}
	\item We call the \textbf{latching category of $\Delta$ in $\l n\r$} the full subcategory of $(\Delta^+\downarrow n)$ containing all objects except the identity map of $n$, i.e. the category containing all arrows in $\Delta^+$ with codomain $\l n\r$ except for the identity. We denote it by $\partial(\Delta^+\downarrow n)$.
	\item We call the \textbf{matching category of $\Delta$ in $\l n\r$} the full subcategory of $(n\downarrow \Delta^-)$ containing all objects except the identity map on $n$, i.e. the category containing all arrows in $\Delta^-$ with domain $\l n\r$ except for the identity. We denote it by $\partial(n\downarrow \Delta^-)$.
\end{itemize}
\end{defin}

\begin{nota}We will denote objects in $\Delta$ and $\Delta\op$ alternatively by $\l n\r$ and $\Delta^n$.
\end{nota}

\begin{defin} Let $M$ be a model category. We define a \textbf{cosimplicial object in $M$} to be a diagram of the form $F^*:\Delta\to M$, and we denote its terms $\Delta^n\to M$ by $F^n$. We denote the category of cosimplicial objects in $M$ by $cM$.
\end{defin}

\begin{defin} Let $M$ be a model category. We define a \textbf{simplicial object in $M$} to be a diagram of the form $F_*:\Delta\op\to M$, and we denote its terms $\Delta^n\to M$ by $F_n$. We denote the category of simplicial objects in $M$ by $sM$.
\end{defin}

\begin{defin}\label{ch.1: latching and matching} Let $M$ be a model category, $\l n\r$ an object in $\Delta$ and $X:\Delta\to M$ a cosimplicial object in $M$. By an abuse of notation we will also call $X$ the induced diagram $X:\partial(\Delta^+\downarrow n) \to M$ defined on objects by $X(\l m\r\to \l n\r)=X(m)$, and the induced diagram $X:\partial(n\downarrow \Delta^-)\to M$ defined on objects by $X(\l n\r\to \l m\r)=X(m)$. 
\begin{itemize}
	\item We define the \textbf{latching object of $X$ at $n$} by $L_nX=\colim_{\partial(\Delta^+\downarrow n)}X$, and the \textbf{latching map of $X$ at $n$} to be the natural map $L_nX\to X_n$.
	\item We define the \textbf{matching object of $X$ at $n$} by $M_nX=\lim_{\partial(n\downarrow \Delta^-)}X$, and the \textbf{matching map of $X$ at $n$} to be the natural map $X_n\to M_nX$.
\end{itemize} 
\end{defin}

\begin{rem}We have decided here to give the definitions in the most basic way to avoid unnecessary clutter, as they will be the only ones we'll be using. But matching and latching objects can be easily defined for any $C$ Reedy category, not just for $\Delta$. In fact, these last definitions are adapted from \cite[Section 15.2]{HLocal}, where they are stated for any $C$ Reedy category.  In particular, they can be defined for $X$ a simplicial object in $M$ by dualizing. In that case, we'd take $(\Delta^-)\op$ to be our $\Delta^+$ and $(\Delta^+)\op$ to be our $\Delta^-$, and everything else would work in the exact same way.
\end{rem}

\begin{defin}Let $M$ be a model category and $X$ an object in $M$. 
\begin{itemize}
	\item We define a \textbf{cosimplicial frame on $X$}, denoted by $C^*(X)$, to be a cosimplicial object in $M$, i.e. a functor $C^*(X):~\Delta\to ~M$, such that $C^0(X)$ is isomorphic to $X$, and that for all $n\in \N$, $C^n(X)$ is weak equivalent to $X$ in the model structure of $\Delta$-diagrams in $M$ and the latching map $L_ nC^*(X)\to C^n(X)$ is a cofibration in $M$.   
	\item We define a \textbf{simplicial frame on $X$}, denoted by $C_*(X)$, to be a simplicial object in $M$, i.e. a functor $C_*(X):~\Delta\op\to~M$, such that $C_0(X)$ is isomorphic to $X$, and that for all $n\in \N$, $C_n(X)$ is weak equivalent to $X$ in the model structure of $\Delta$-diagrams in $M$ and the matching map $C_n(X)\to M_nC_*(X)$ is a fibration in $M$.
\end{itemize}
\end{defin}

As it is common in all areas of mathematics, we have given a definition of objects, but we haven't worked out whether or not those objects really exist. Luckily for us, they do.

\begin{teo}Let $M$ be a model category. There exists a functorial simplicial frame $C_*(-)$ and a functorial cosimplicial frame $C^*(-)$.
\end{teo}

So now for every object in a model category $M$ we have a (co)simplicial object. We will now use these objects to get actual simplicial sets. For that, we will find an adjunction associated to every (co)simplicial object, in \cite[Section 3.1]{Hovey}. 

\begin{prop}\label{ch. 1:tensor} Let $M$ be a category with all small colimits (in particular, a model category). Then the category of cosimplicial objects in $M$, $\Fun(\Delta, M)$, is equivalent to the category of adjunctions $\Adj(\sS, M)$. In particular, if $A^*$ is a simplicial object in $M$ then the right adjunction $\Map(A^*, -):M\to \sS$ is defined to have $n$-simplices $\Hom(A^n, -)$ for all $n\in\N$. We denote the adjunction associated to a cosimplicial object $A^*$ by $A^*\otimes -:\sS\rightleftharpoons M: Map(A^*,-)$.
\end{prop}

\begin{rem} As usual, we have a simplicial analog of this result. In this case, if we have $A_*$ a simplicial object in $M$, then the right adjunction $\Map(-,A_*):M\to \sS\op$ is defined to have $n$-simplices $\Hom(-,A_n)$ for all $n\in\N$, and we denote the adjunction associated to a simplicial object $A_*$ by $(\Hom(-,A_*), \Map(-,A_*),\psi)$.
\end{rem}

\begin{coro}Let $M$ be a model category, and $A\in M$ an object of $M$. Then the functorial cosimplicial frame on $A$, $C^*(A)$, induces adjoint functors $(C^*(A)\otimes -, \Map(C^*(A),\phi)$. Dually, the functorial simplicial frame on $A$, $C_*(A),$ induces adjoint functors $\Hom(-,C_*(A)), \Map(-,C_*(A)), \psi)$.
\end{coro}

\begin{prop}\label{ch. 1:cof-to-fib}Let $M$ be a model category. Let $A,B,X,Y$ be objects in $M$.
\begin{itemize}
	\item Let $p:X\to Y$ be a fibration in $M$. Then the map $p_*:\Map(C^*(A),X)\to \Map(C^*(A),Y)$ is a fibration of simplicial sets, which is a trivial fibration if $p$ is a trivial fibration.
	\item Let $i:A\to B$ be a cofibration in $M$. Then the map $i^*:\Map(B,C_*(X))\to \Map(A,C_*(Y))$ is a fibration of simplicial sets, which is a trivial fibration if $i$ is a trivial cofibration. 
\end{itemize}
\end{prop}

Are those adjunctions Quillen adjunctions? And what is the link between those two right adjoints called $\Map$? Sadly, the adjunctions are not Quillen adjunctions in general, but they do preserve enough structure to be derived; and the total derived functors for $\Map(C^*(A), -)$ and $\Map(-,C_*(Y))$ coincide. If we want to have those conditions for the original adjunctions, though, we will need to impose some more conditions.

\begin{prop}\label{ch.1:Map Quillen} Let $M$ be a model category. 
\begin{itemize}
	\item Let $X$ be a cofibrant object in $M$. Then the functor $C^*(X)\otimes -$ preserves cofibrations and trivial cofibrations, and its right adjoint $\Map(C^*(X),-)$ preserves fibrations and trivial fibrations. In particular, if $X$ is a cofibrant object in $M$, then $C^*(X)\otimes -:\sS\rightleftharpoons M: \Map(C^*(X),-)$ is a Quillen adjunction.
	\item Let $Y$ be a fibrant object in $M$. Then the functor $\Hom(-, C_*(Y))$ preserves cofibrations and trivial cofibrations, and its right adjoint $\Map(-, C_*(Y))$ preserves fibrations and trivial fibrations. In particular, if $Y$ is a fibrant object in $M$, then $(\Hom(-,C_*(Y)), \Map(-,C_*(Y)), \psi)$ is a Quillen adjunction.
\end{itemize}
\end{prop}

\begin{prop}Let $M$ be a model category, $X$ a cofibrant object and $Y$ a fibrant object. Then there are weak equivalences 
$$\Map(C^*(X),Y)\to \diag \Map(C^*(X),C_*(Y))\leftarrow \Map(X,C_*(Y)). $$
\end{prop}

\begin{rem}We want to attract attention to the fact that this means that $\Map(C^* (X),Y)$ and $\Map(X,C_*(Y))$ will be isomorphic on the homotopy category, but there isn't a direct weak equivalence between them: we need to go through $\diag\Map(C^*(X),C_*(Y))$.
\end{rem}

\begin{rem}On the other hand, we have defined this Quillen adjunction and these weak equivalences only in the case where $X$ is cofibrant and $Y$ is fibrant. That isn't a big issue: we recall that we have the fibrant and cofibrant replacements. So we could alternatively had defined those weak equivalences for all $X, Y$ in $M$ by
$$\Map(C^*(X),R(Y))\to \diag \Map(C^*(X),C_*(Y))\leftarrow \Map(Q(X),C_*(Y)). $$
\end{rem}

\begin{nota}By abuse of notation, we will denote $\Map(X,Y)$ both simplicial sets $\Map(C^*(X),RY)$ and $\Map(QX,C_*(Y))$.
\end{nota}

\begin{defin}Let $M$ be a model category, and $X,Y$ two objects in $M$. Then we call \textbf{the mapping space between $X$ and $Y$} and we denote by $\Map(X,Y)$ the simplicial sets $\Map(C^*(X),R(Y))$ and $\Map(QX,C_*(Y))$. 
\end{defin}

And, to finish off this part, a theorem summarizing the results in the homotopy category.

\begin{teo}Let $M$ be a model category. Then the total left derived functors of $-\otimes -:M\times \sS\to M$ and $\Hom(-,-):\sS\times M\op\to M\op$ exist. We denote them by $-\otimes^\L-$ and $\R Hom(-,-)$ respectively. The total right derived functors of $\Map(-,-)$ exist and are naturally isomorphic. We denote them by $\R\Map(-,-)$. There are natural isomorphisms in $\Ho(M)$
$$\l X\otimes^\L K,Y\r\cong \l K,\R\Map(X,C^*(X),Y)\r\cong \l K, \R\Map(X,C_*(Y)) \r\cong \l X,\R\Hom(K,Y)\r. $$
so there is an adjunction on two variables $\Ho (M)\times \Ho(\sS)\to \Ho(M)$. There is also a natural isomorphism $X\otimes^\L \Delta\l 0\r\cong X$. 
\end{teo}

Before we pass to the next section, let us add one last result from \cite{HLocal}, linking mapping spaces and homotopy colimits.

\begin{teo}\cite[Th. 19.4.4]{HLocal}\label{ch. 1: Map commutes with hocolim} Let $M$ be a model category and $C$ a small category. 
\begin{enumerate}
	\item Let $X_i$ be an objectwise cofibrant $C$-diagram in $M$ and $Y$ a fibrant object in $M$. Then the mapping space functor $\Map(-,Y)$ sends homotopy colimits to homotopy limits, i.e. $\Map(\hocolim(X_i),Y)\simeq \holim (\Map(X_i,Y))$.
	\item Let $X$ be a cofibrant object in $M$ and $Y_i$ be  an objectwise fibrant $C$-diagram in $M$. Then the mapping space functor $\Map(X,-)$ sends homotopy limits to homotopy colimits, i.e. $\Map(X,\holim Y_i)\simeq \holim (Map(X,Y_i))$.
\end{enumerate}
\end{teo}

\subsection{Bousfield localizations}

In this section we will introduce a very important tool in model categories, Bousfield localizations. All results in it come from \cite[Ch. 3]{HLocal} unless stated otherwise.
\\
\\There are a couple of building bricks we will need in order to construct a Bousfield localization. Let us get those out of the way first.

\begin{rem}There isn't just one Bousfield localization: there are two, left and right. In this text we will talk exclusively about left Bousfield localizations, as they are the only ones we will be using afterwards, and as such sometimes we will omit the word "left". The construction of a right Bousfield localization is dual, even if the proof of its existence is not.
\end{rem}

\begin{defin}Let $M$ be a model category and $C$ a class of morphisms in $M$. Let $X$ be an object in $M$. We say that $X$ is \textbf{a $C$-local object} if  for every morphism $f:A\to B$ in $C$ the induced map $\Map(f,X):\Map(A,X)\to \Map(B,X)$ is a weak equivalence.
\end{defin}

\begin{rem}A reader familiar with the conventions set in \cite{HLocal} will probably have noticed that this is not exactly the definition there: indeed, Hirschhorn adds the condition of being fibrant. We have preferred this alternative version for clearness' sake. It isn't a stretch, either: this is the definition given in the context of $\infty$-categories, after all.
\end{rem}

\begin{defin}Let $M$ be a model category and $C$ a class of morphisms in $M$. Let $f:A\to B$ be a morphism in $M$. We say that $f$ is \textbf{a $C$-local equivalence} if for every $C$-local object $X$ the induced map $\Map(f,X):\Map(A,X)\to \Map(B,X)$ is a weak equivalence.
\end{defin}

Now, when we compute our Bousfield localizations, our goal is to have the $C$-local equivalences be our new weak equivalences and our $C$-local objects be our new fibrant objects. For that, we will make sure that they are well defined for that purpose.

\begin{prop}Let $M$ be a model category and $C$ be a class of morphisms in $M$. Let $X$ and $Y$ be two fibrant objects which are weakly equivalent. Then $X$ is $C$-local if and only if $Y$ is too.
\end{prop}

\begin{prop}Let $M$ be a model category and $C$ be a class of morphisms in $M$. Then the class of $C$-local equivalences satisfies the two-out-of-three property and is closed under retracts.
\end{prop}

We now have all we need in order to define our localization.

\begin{defin}Let $M$ be a model category and $C$ be a class of maps in $M$. The left Bousfield localization of $M$ with respect to $C$ is a model category structure $L_CM$ on the category $M$ such that 
\begin{enumerate}
	\item the class of weak equivalences of $L_CM$ is the class of $C$-local equivalences.
	\item the class of cofibrations of $L_CM$ is the same as the class of cofibrations in $M$.
\end{enumerate}
\end{defin}

Now, this is a definition: as usual, nothing tells us that such a model structure exists. We just know that if it does exist, we call that a left Bousfield localization. And it does not exist in every context; but luckily for us, it does exist in every context we need it to.

\begin{teo}\label{ch. 1: existence Bousfield}\cite[Th. 4.1.1.]{HLocal} Let $M$ be a left proper cellular model category and let $C$ be a class of morphisms in $M$. Then the left Bousfield localization with respect to $C$ exists and the localization $L_CM$ is a left proper cellular model category. On top of this, if $M$ is a simplicial model category (see next section), then the localization is also a simplicial model category with the inherited structure.
\end{teo}

Now, there are terms in this definition that we haven't talked about. What is a cellular model category? It isn't important for us. We only need to know if the categories we work on are as such.

\begin{prop}\label{ch. 1: sSet good Bousfield}\cite[Prop. 4.1.4]{HLocal} The category of simplicial sets $\sS$ is a left proper cellular model category.
\end{prop}

\begin{prop}\label{ch. 1: fun good Bousfield}\cite[Prop. 4.1.5]{HLocal} Let $M$ be a left proper cellular model category and $T$ a small category. Then the diagram category $\Fun(T,M)$ is also a left proper cellular model category.
\end{prop}

Now that we have this localization and the existence of it in the cases we are interested in, let us state a few results that will be important later.

\begin{prop}Let $M$ be a model category and $C$ a class of morphisms on $M$. We suppose that the left Bousfield localization of $M$ with respect to $C$, $L_CM$, exists. Then we have the following properties.
\begin{enumerate}
	\item Every weak equivalence of $M$ is a weak equivalence of $L_CM$.
	\item Every fibration in $L_CM$ is a fibration in $M$. 
	\item Every trivial cofibration in $M$ is a trivial cofibration in $L_CM$.
\end{enumerate}
\end{prop}

\begin{prop}\label{ch.1: proper Bousfield}Let $M$ be a left proper model category and $C$ a class of morphisms on $M$. If it exists, then the left Bousfield localization $L_CM$ is also left proper.
\end{prop}

And lastly, let us give a result about the fibrant objects of such a localization.

\begin{prop}\label{ch. 1: fibrant and C-local}Let $M$ be a left proper model category and $C$ a class of morphisms on $M$. We suppose that the left Bousfield localization of $M$ with respect to $C$, $L_CM$, exists. Then an object $X$ is a fibrant object in $L_CM$ if and only if it is a fibrant $C$-local object in $M$.
\end{prop}

\section{Simplicial machinery}

In this section we will introduce some machinery concerning simplicial objects and simplicial categories, and the construction of a simplicial localization. %The results from this section come from \cite[Section 2]{HAG1} unless stated otherwise.

\subsection{Simplicial objects}

Unless stated otherwise, all definitions and results in this section come from \cite{DHI_hypercoverings} and \cite{HAG1}.

\begin{defin}Let $C$ be a category and $n$ an integer. We call an \textbf{$n$-truncated simplicial object in $C$} a functor $F:\Delta_{\leq n}\to C$. We denote by $sC_{\leq n}$ the category of $n$-truncated simplicial objects in $C$.
\end{defin}

\begin{defin}Let $C$ be a category. We call an \textbf{augmented simplicial object in $C$} a functor $F:\Delta_+\op\to C$, where $\Delta_+$ is the category of possibly empty finite totally ordered sets. Equivalently, it is a simplicial object $U_*$ of $C$ coupled with a morphism $U_*\to cX$ where $cX$ is the constant simplicial object that is $X$ in every degree. We denote it by $U_*\to X$. 
\end{defin}

\begin{nota}We denote the category of augmented simplicial objects in $C$ by $s_+C$, and the category of $n$-truncated augmented objects in $C$ by $s_+C_{\leq n}$.
\end{nota}

We have already discussed a model structure that would be available for $sC$, $s_+C$ and $sC_{\leq n}$ if $C$ is a cofibrantly generated model category. Indeed, if $C$ is a cofibrantly generated model category we have the projective model structure which would be available. But in the case of simplicial objects we have another structure: $\Delta$ (and $\Delta_+$ and $\Delta_{\leq n}$ as well) is a Reedy category. We will define, then, the Reedy model structure for these categories. For that, we will be using the latching and mapping objects we defined in Definition \ref{ch.1: latching and matching}. We will define this for simplicial objects, but the definitions can be easily adapted to all the other structures we have mentioned earlier. Results from this section come from \cite[Section 15.3]{HLocal}.

\begin{defin}Let $M$ be a model category and $X,Y:\Delta\op\to M$ two simplicial objects in $M$. We say that a map $f:X\to Y$ is a \textbf{Reedy weak equivalence} if it is an objectwise weak equivalence, i.e. if for every object $\Delta^n$ the induced map $X_n\to Y_n$ is a weak equivalence in $M$.
\end{defin}

\begin{defin}Let $M$ be a model category and $X,Y:\Delta\op\to M$ two simplicial objects in $M$. We say that a map $f:X\to Y$ is a \textbf{Reedy cofibration} if for every object $\Delta^n$ the induced map 
$$X_n\coprod_{L_nX}L_nY\to Y_n $$
is a cofibration in $M$.
\end{defin}

\begin{defin}Let $M$ be a model category and $X,Y:\Delta\op\to M$ two simplicial objects in $M$. We say that a map $f:X\to Y$ is a \textbf{Reedy fibration} if for every object $\Delta^n$ the induced map 
$$X_n\to Y_n\coprod_{M_nY}M_nX $$
is a fibration in $M$.
\end{defin}

\begin{teo}Let $M$ be a model category. The category $sM$ admits a model structure, which is given by the weak equivalences, fibrations and cofibrations given above. We call this model structure the \textbf{Reedy model structure on $sM$}. In particular, if the category $M$ is left (right) proper, the category $sM$ with the Reedy model structure is also left (right) proper.
\end{teo}

But it's not just the fibrations and the cofibrations that can be defined using matching and latching objects. In fact, trivial fibrations and cofibrations can be defined in the exact same way.

\begin{teo}Let $M$ be a model category, and $f:X\to Y$ be a morphism in $sM$. Then $f$ is a Reedy trivial cofibration if 
$$X_n\coprod_{L_nX}L_nY\to Y_n $$
is a trivial cofibration in $M$. Equivalently, $f$ is a Reedy trivial fibration if 
$$X_n\to Y_n\coprod_{M_nY}M_nX $$
is a trivial fibration in $M$.
\end{teo}

So, in the case when $M$ is a cofibrantly generated model category, we have defined two different model structures for $sM$. We see that the weak equivalences are the same, which means that the homotopy categories will be identical: but the model structures are not the same. Indeed, the Reedy model structure has less cofibrations and more fibrations than the projective model structure (every Reedy cofibration is a projective cofibration, but the reverse isn't true). We still have a Quillen equivalence, though.

\begin{teo}Let $M$ be a cofibrantly generated model category. Then the identity functor in $sM$ is a left Quillen equivalence from the projective model structure to the Reedy model structure, and a right Quillen equivalence in the other direction.
\end{teo}

Now we can go back to discussing other functors that will be important when talking about simplicial objects. As expected, there is a close link between simplicial objects and truncated simplicial objects. Although the definition of these functors might be easy, they are very useful and as such they deserve to be discussed explicitly.

\begin{defin}Let $C$ be a category and $n$ an integer. We call the \textbf{$n$-skeleton functor} the forgetful functor $\sk_n:s_+C\to s_+C_{\leq n}$. We call the \textbf{$n$-skeleton of $X$} the image of an augmented simplicial object $X$ by said functor, and we denote it by $\sk_nX$.
\end{defin}

\begin{prop}Let $C$ be a category. There exists a right adjoint to the $n$-skeleton functor, $\cosk_n:s_+C_{\leq n}\to s_+C$, which is a right Quillen adjunction for the Reedy model structure. We call it the \textbf{$n$-coskeleton functor} and for an $n$-truncated simplicial object $X$, we call the image by $\cosk_n$ of $X$ the \textbf{$n$-coskeleton of $X$}.
\end{prop}

\begin{rem}\label{ch.1: cosk igualdad} By construction, it is easy to see that, if we fix an integer $n\in\N$ and a simplicial object $X$, we have $X_i=(\cosk_n\sk_n X)_i$ for all $i\leq n$. 
\end{rem}

On top of that adjunction there is another one we will be using in our definitions. 

\begin{prop}Let $M$ be a model category. Then there exists a structure of tensored category over $\sS$ on $sM$ and $s_+M$, defined as follows: let $X_*\in sM$ (or $s_+M$) be a simplicial object, and $\underline{A}\in \sS$, then the external product is given by 
$$(\underline{A}\otimes X_*)_n=\coprod_{A_n} X_n. $$
There exists a right adjoint to $-\otimes-$, that we will call \textbf{the exponential of $X$ by $\underline{A}$} and we denote by $X_*^{\underline{A}}$.
\end{prop}

\begin{defin}We denote by $X_*^A$ the 0-th level of the simplicial object $(R X)_*^{\underline{A}}$, where $R X$ is the fibrant replacement of $X$ for the Reedy model structure. This is explicitly given by the following formula: $ X_*^A=\operatorname{End}(F) $ where $F$ is given by
\begin{center}
	\begin{tikzcd}
		F:\Delta\op\times\Delta\ar[r] &sM\\
		(\l n\r,\l m\r)\ar[r,mapsto]& \coprod_{A_m} (R X)_n
	\end{tikzcd}
\end{center}
\end{defin}

\begin{rem}In principle we could have given this definition to be just the 0-th level of the simplicial set $X_*^{\underline{A}}$. If $X$ is an object in $M$, we do have an isomorphism on the homotopy categories between $X_*^A$ and $(c(X)^{\R A})_0$, the 0-th level of the derived exponential over $c(X)$ the constant simplicial object. However, we have to remember that these two constructions are not isomorphic directly, only on the homotopy categories.
\end{rem}

We can now give a few computations using this construction.

\begin{prop}\label{ch.1: computations_exp} Let $U_*\to X$ be an augmented simplicial object of $M$ a model category. We have then the following isomorphisms:
\begin{enumerate}
	\item Let $\emptyset$ be the initial object in $\sS$. Then we have the isomorphism $U_*^\emptyset\simeq X$.
	\item For all $n\in\N$ we have the isomorphism $U_*^{\Delta^n}\simeq U_n$.
	\item Let $\underline{A}$ be a simplicial set. We have an isomorphism $U_*^{\sk_n \underline{A}}\simeq (\R\cosk_n U_*)^A$. In particular, $U_*^{\partial \Delta^n}\simeq (\R\cosk_{n-1}(\sk_{n-1} U_*))_n$.
\end{enumerate}
\end{prop}

\subsection{Simplicial categories}\label{Simplicial categories}

Let us continue with the simplicial categories. The results from this part come from \cite[Section 2]{HAG1} unless stated otherwise.

\begin{defin}We define $T$ a \textbf{simplicial category} to be a category enriched over $\sS$ the category of simplicial sets. Equivalently, a simplicial category consists of the following data: 
\begin{itemize}
	\item A set of objects $\Ob(T)$.
	\item For every pair of objects in $T$, $(x,y)\in\Ob(T)^2$, a simplicial set $\Hom(x,y)\in \sS$. 
	\item For every triple of objects in $T$, $(x,y,z)\in\Ob(T)^3$ a composition morphism in $\sS$
	$$\mu: \Hom(x,y)\times\Hom(y,z)\to \Hom(x,z) $$
	with the usual associativity condition.
	\item For every object in $T$, $x\in T$, a 0-simplex $\Id_x\in \Hom(x,x)_0$ that satisfies the usual unit condition with respect to the composition stated above.
\end{itemize}
\end{defin} 

\begin{defin} Let $T$ and $T'$ be two simplicial categories. A \textbf{simplicial functor} (also called a\textbf{ morphism of simplicial categories}) is a functor $f: T\to T'$ enriched over the category of simplicial sets. Equivalently, it consists of the following data: 
\begin{itemize}
	\item A map of sets $\Ob(T)\to \Ob(T')$.
	\item For every pair of objects in $T$, $(x,y)\in Ob(T)^2$, a morphism of simplicial sets
	$$\Hom(x,y)\to \Hom(f(x), f(y)). $$
	satisfying the usual associativity and unit conditions.
\end{itemize}  
\end{defin}

\begin{nota}We denote $\sC$ the category given by simplicial categories and their morphisms.
\end{nota}

\begin{ex}\label{ch. 1:sS} The category of simplicial sets, $\sS$, is a simplicial category, where for all $x,y\in \sS$ we take $\Hom_n(x,y)=\Hom(x\times \Delta_n,y)$.
\end{ex}

\begin{nota}Let $T,T'$ be two simplicial categories. We will denote by $\Fun^\S(T,T')$ the category of simplicial functors between them.
\end{nota}

\begin{rem}Using the inclusion functor $i:\Set\to \sS$, we can see that there is an obvious inclusion $\Cat\to \sC$. Consequently, that means that we can see all categories as simplicial categories, by considering the simplicial category with the same objects and the constant simplicial set as the simplicial set of morphisms. By abuse of notation, we will still call $C$ the simplicial category associated to a category $C$, unless the difference is important to the result.
\end{rem}

For any simplicial category $T$, we can define an associated category, which we will call the homotopy category of $T$.

\begin{defin}Let $T$ be a simplicial category. We call the \textbf{homotopy category of $T$}, and we denote by $\pi_0(T)$, a category which has the same objects as $T$ and whose morphisms are given by 
$$\forall x,y\in \Ob(T),\ \Hom_{\pi_0(T)}(x,y)=\pi_0(\Hom_T(x,y)), $$
i.e. the set of connected components of the simplicial set of morphisms.
\end{defin}

\begin{rem}It is easy to check that the functor $\pi_0:\sC\to \Cat$ is a left adjoint of the inclusion functor.
\end{rem}

As it is the case with every category we have defined so far, there exists a model structure for the category of simplicial categories. We can even prove that it is cofibrantly generated. The weak equivalences and fibrations are defined as follows.

\begin{defin}\cite[Introduction (1)]{Berg-simplicial} Let $f:T\to T'$ be a morphism of simplicial categories. We say that $f$ is a \textbf{weak equivalence} if it satisfies the following conditions:
\begin{itemize}
	\item For all $x,y\in \Ob(T)$, the associated morphism $\Hom(x,y)\to \Hom(f(x),f(y))$ is a weak equivalence of simplicial sets.
	\item The induced functor of homotopy categories $\pi_0(f):\pi_0(T)\to \pi_0(T')$ is an equivalence of categories.
\end{itemize}
\end{defin}

\begin{defin}\cite[Introduction]{Berg-simplicial} Let $T$ be a simplicial category. We say that a morphism in $T$, $f\in \Hom_T(x,y)_0$, is a \textbf{homotopy equivalence} if it becomes an isomorphism $\pi_0(f)$ in $\pi_0(T)$. 
\end{defin}

\begin{defin}\cite[Introduction (2)]{Berg-simplicial} Let $f:T\to T'$ be a morphism of simplicial categories. We say that $f$ is a \textbf{fibration} if it satisfies the following conditions:
\begin{itemize}
	\item For all $x,y\in\Ob(T)$, the associated morphism $\Hom(x,y)\to\Hom(f(x), f(y))$ is a fibration of simplicial sets.
	\item For all $x\in \Ob(T)$,  $y'\in \Ob(T')$, and all $h:f(x)\to y'$ homotopy equivalence in $T'$, there exists an object $y\in \Ob(T)$ and a homotopy equivalence $g:x\to y$  such that $f(g)=h$. 
\end{itemize}
\end{defin}

\begin{teo}\cite[Th. 1.1]{Berg-simplicial} The category $\sC$ admits a model structure, which is given by the weak equivalences and fibrations defined above. In particular, this model structure is cofibrantly generated.
\end{teo}

\begin{prop}\cite[Prop. 3.5]{Berg-simplicial} With the model structure defined above, the model category $\sC$ is right proper.
\end{prop}

So the category of simplicial categories is a model category. But we could also ask, what happens when we get a simplicial category which already had a model structure? Can we say something about it? Of course we can. We take the following definition from \cite[Def. 9.1.6]{HLocal}.

\begin{defin} Let $M$ be a model category which is also a simplicial category. We say that $M$ is a \textbf{simplicial model category} if it satisfies the following conditions:
\begin{itemize}
	\item For all $X,Y\in \Ob(M)$, and for all simplicial set $K$, there are objects $X\otimes K$ and $Y^K$ in $M$ such that there are isomorphisms of simplicial sets
	$$\Hom_M(X\otimes K,Y)\simeq\Hom_\sS(K,\Hom(X,Y))\simeq \Hom_M(X,Y^K). $$
	\item If $i:A\to B$ is a cofibration in $M$ and $p:X\to Y$ is a fibration in $M$, then the map of simplcial sets
	$$\Hom(B,X)\xrightarrow{i^*\times p_*}\Hom(A,X)\times_ {\Hom(A,Y)}\Hom(B,Y) $$ 
	is a fibration that is a trivial fibration if either $i$ or $p$ is a weak equivalence.
\end{itemize}
\end{defin}

We have seen that the category of diagrams over a model category is still a model category. So it isn't absurd to wonder whether the category of diagrams over a simplicial model category is still a simplicial model category, and the answer is yes, it is. We take the construction of the simplicial model category from \cite[Section 11.7]{HLocal}

\begin{prop}\label{ch. 1:simplicial} Let $M$ be a simplicial model category and $C$ a small category. Then the category of diagrams $\Fun(C,M)$ is a simplicial category. For all pairs of diagrams $X,Y\in \Fun(C,M)$, we define the simplicial set of morphisms between $X$ and $Y$ as follows: the n-simplices of $\Hom(X,Y)_*$ are given by the maps $X\otimes\Delta_n\to Y$, where $X\otimes \Delta_n$ is defined for all objects $\alpha\in\Ob(C)$ as $(X\otimes\Delta_n)(\alpha)=X(\alpha)\otimes \Delta_n$ and for all morphisms $f\in C$ as $(X\otimes\Delta_n)(f)=X(f)\otimes \Id_{\Delta_n}$.
\end{prop}

\begin{teo}\label{ch. 1: functors simplicial} Let $M$ be a simplicial cofibrantly generated model category and $C$ a small category. Then the projective model structure of $\Fun(C,M)$ is compatible with the simplicial enrichment from Proposition \ref{ch. 1:simplicial}, and $\Fun(C,M)$ is a simplicial cofibrantly generated model category.
\end{teo}

\begin{teo} Let $M$ be a simplicial cofibrantly generated model category. Then the Reedy model structure of $sM$ is compatible with the simplicial enrichment from Proposition \ref{ch. 1:simplicial}, and $sM$ is a simplicial model category.
\end{teo}

\begin{prop}Let $M$ be a simplicial cofibrantly generated model category and $f:T\to T'$ a equivalence of simplicial categories. Then the induced Quillen adjunction
$$f_!:\Fun^\S(T,M)\rightleftharpoons \Fun^\S(T',M):f^* $$
is a Quillen equivalence. 
\end{prop}

And while we're talking about categories of simplicial functors, we will construct a simplicial equivalent of the Yoneda embedding but in the case of simplicial functors. 

\begin{nota}Let $T$ be a simplicial category. We will denote the natural simplicial enrichment of $\Fun^\S(T\op,\sS)$ by $\Fun^\S(T\op,\sS)_s$. 
\end{nota}

\begin{nota}\label{ch.1:notation h simp}Let $T$ be a simplicial category, and $x$ be an object in $T$. We denote by $\underline{h}_x:T\op\to \sS$ in $ \Fun^\S(T\op,\sS)_s$ the simplicial functor $\underline{h}_x(y)=\Hom_T(y,x)$ where we take the simplicial set of morphisms in $T$. 
\end{nota}

\begin{defin}Let $T$ be a simplicial category. We define a morphism of simplicial categories $\underline{h}:~T\to \Fun^\S(T\op,\sS)_s$ by $\underline{h}(x)=\underline{h}_x$.
\end{defin}

\begin{prop}Let $T$ be a simplicial category, $x$ an object in $T$ and $F$ a simplical functor $F:~T\to ~\sS$. There exists a canonical isomorphism of simplicial sets 
$$F(x)\simeq\Hom_{\Fun^\S(T\op,\sS)_s}(\underline{h}_x,F) $$
which is functorial on the pair $(F,x)$. In particular, $\underline{h}$ is fully faithful as a simplicial functor.
\end{prop}

So the functor $\underline{h}$ induces a functor $\pi_0(h):\pi_0(T)\to \pi_0(\Fun^\S(T\op,\sS)_s)$. But we would want a fully faithful functor that goes to the homotopy category of simplicial presheaves, without the simplicial enrichment. We have an identity functor $\Fun^\S(T\op, \sS)_s\to\Fun^\S(T\op, \sS)$ we could use, but does it induce a well-defined functor in  the homotopy categories? 

\begin{rem}We remind the reader that we are here working on two different things that are called homotopy categories: on one side we have a homotopy category as a simplicial category, and on the other a homotopy category as a model category. That means that if we take $F, G\in \Ob(\Fun^\S(T\op, \sS)_s)$ simplicial presheaves, and $f,g$ two morphisms from $F$ to $G$, on one side we have $\pi_0(\Fun^\S(T\op, \sS)_s)$, in which $f$ and $g$ are equal if they are on the same connected component (are simplicially homotopic), and on the other side we have $\Ho(\Fun^\S(T\op,\sS))$, in which $f$ and $g$ are equal if there exists a homotopy equivalence between them (up to a fibrant/cofibrant replacement). 
\end{rem}

Are these two localizations indeed compatible? Can the identity functor induce a well-defined functor in the homotopy categories? Yes, it does. We get the necessary result from \cite[Lem. 9.5.15]{HLocal}.

\begin{lema}Let $M$ be a simplicial model category and let $X,Y$ be two objects of $M$. If $f,g:X\to Y$ are simplicially homotopic maps, then $f$ and $g$ are the same map in the homotopy category of $M$ as a model category.
\end{lema}

\begin{prop}Let $T$ be a simplicial category. The identity morphism induces a well-defined functor
$$\pi_0(\Fun^\S(T\op, \sS)_s)\to \Ho(\Fun^\S(T\op,\sS)).$$
\end{prop}

And his gives us finally a homotopy version of the enriched Yoneda embedding.

\begin{prop}Let $T$ be a simplicial category, $x$ an object in $T$ and $F$ a simplicial functor \\$F:T\op\to \sS$. There exists a canonical isomorphism in $\Ho(\sS)$
$$F(x)\simeq \R\Hom_{\Fun^\S(T\op, \sS)}(\underline{h}_x,F)$$
which is functorial on the pair $(F,x)$. In particular, the functor 
$$\underline{h}: \pi_0(T)\to \pi_0(\Fun^\S(T\op,\sS)_s)\to \Ho(\Fun^\S(T\op,\sS))$$
is fully faithful. 
\end{prop}

\subsection{Simplicial localizations}

In the section about the homotopy category of a model category we have already mentioned the Gabriel-Zisman localization. But we can also give the definition of an enhanced localization that, instead of giving us a category, gives us a simplicial category. All results in this section come from \cite[2.2 and 2.3]{HAG1} unless stated otherwise.

\begin{defin}\label{Ch. 1: def localization} Let $C$ be a category and $W$ a subset of its morphisms. We call a \textbf{simplicial localization of $C$ with respect to $W$} a pair $(L_WC, l)$ where $L_WC$ is a simplicial category and $l:C\to L_WC$ is a morphism of simplicial categories, called \textbf{the localization morphism}, such that for every simplicial category $T$, the aforementioned morphism induces a equivalence of simplicial categories
$$l^*:\R\Fun^\S(L_WC,T)\simeq\R\Fun^\S_W(C,T) $$
where $\R \Fun^\S(L_W, T)$ is seen as an object of $\Ho(\sC)$, and $\R\Fun^\S_W(C,T)$ denotes the full subcategory of $\R\Fun^\S(C,T)$ in $\Ho(\sC)$ consisting of all simplicial morphisms that send $W$ to equivalences in $T$. In other words, the localization morphism is such that, for every simplicial category $T$, $l$ induces a morphism of simplicial categories
$$l^*:\R\Fun^\S(L_WC,T)\to\R\Fun^\S(C,T) $$
which is fully faithful and whose essential image consists of the functors sending the morphisms in $W$ to equivalences in $T$.
\end{defin}

\begin{rem}Once again, we have given a definition in an abstract way, with no certainty that such a thing exists. It has been proven, though, that a simplicial localization as defined above always exists, and that it is equivalent to the Dwyer-Kan simplicial localization from \cite{DK-local}. 
\end{rem}

\begin{rem}A reader familiar with \cite{HAG1} will probably remark that the universal property given here looks much stronger than the one given in that paper; indeed, we have taken it from \cite[Localization and model categories]{DAG}. We can prove, though, that those two universal properties are, in fact, equivalent.
\end{rem}

So we now have a localization which is a simplicial category instead of a classical category. It could happen that that simplicial category was always trivial, in which case we would have gained nothing from this endeavour. Luckily, that is not the case: in general the morphism spaces in $L_WC$ aren't 0-truncated. There is still a close relationship between the Gabriel-Zisman localization and this one, though.

\begin{prop}\label{ch.1: simplicial loc. and pi_0}Let $C$ be a category, $W$ a subset of its morphisms. We take the simplicial localization of $C$ with respect to $W$, $(L_WC,l)$. The localization morphism induces an equivalence between the Gabriel-Zisman localization $C\l  W^{-1}\r$ and the homotopy category of $L_WC$, $\pi_0(L_WC)$.
\end{prop}

In the case where the category $C$ is a model category, we have a useful result on top of this. Once again, we will take the result not from \cite{HAG1}, but from \cite[Localization and model categories]{DAG}.

\begin{prop} Let $M$ be a model category, and $C$ a small category. Then there exists a natural equivalence of simplicial categories 
$$L_{W_C}(\Fun(C,M))\simeq \R\Fun^\S(C,L_WM), $$
where $W$ are the weak equivalences in $M$ and $W_C$ are the weak equivalences on $\Fun(C,M)$ using the projective model structure.
\end{prop}

\begin{rem}One consequence of this result is that if $M$ is a model category, $L_WM$ has all limits and colimits, and those limits and colimits can be computed using homotopy limits and colimits.
\end{rem}

While we are talking about categories of simplicial diagrams from a localization to a model category $M$, we can ask what happens when we take $M=\sS$. In that case, we have a Quillen equivalent construction. We take $\sS$ as a simplicial model category using the structure given in Example \ref{ch. 1:sS}.

\begin{defin}Let $C$ be a simplicial category and $W$ a subset of its morphisms. We call the \textbf{model category of restricted diagrams from  $(C,W)$ to $\sS$} the left Bousfield localization of $\Fun^\S(C,\sS)$ along the set of morphisms of the form $\underline{h}_x\to \underline{h}_y$ for all $x\to y\in W$, with $\underline{h}_x$ as defined in Notation \ref{ch.1:notation h simp}. %where $h_x:C\op\to \sS$ is given by $h_x(z)=\Hom_C(y,z)$ as simplicial sets.
 We denote it by $\sS^{C,W}$.
\end{defin}

\begin{rem}By the general theory of Bousfield localizations, the fibrant objects of $\sS^{C,W}$ are those functors $f:C\to \sS$ that satisfy the following conditions:
\begin{itemize}
	\item The functor $f$ is a fibration for the projective model structure.
	\item For all $x\to y$ morphism in $W$, the induced morphism $f(x)\to f(y)$ is an equivalence in $\sS$.
\end{itemize}
\end{rem}

\begin{teo}\label{ch. 1: equivalence for functor categories}Let $C$ be a simplicial category, $W$ a subset of its morphisms. Let $(F_*C, F_*W)$ be the canonical free resolution of $(C,W)$ as simplicial categories. There exist two natural functors 
$$(C,W)\xleftarrow{p} (F_*C,F_*W)\xrightarrow{l} L_{F_*W}F_*C=L_WC $$
which induce two right Quillen functors
$$\sS^{C,W}\xrightarrow{p^*}\sS^{F_*C,F_*W}\xleftarrow{l^*}\Fun^\S(L_WC,\sS).$$
Those Quillen functors $p^*, l^*$ are Quillen equivalences. In particular, there exists a chain of Quillen equivalences between $\sS^{C,W}$ and $\Fun^\S(L_WC,\sS)$.
\end{teo}

\begin{rem}We bring to your attention that we have said "there exists a chain of Quillen equivalences" and not "there exists a Quillen equivalence". As we said in Remark \ref{ch1: composition}, we cannot compose $p_*$ and the left Quillen adjoint of $l_*$ and get a Quillen adjunction, as they aren't both right adjoints.
\end{rem}

\begin{coro}Let $C, D$ be two simplicial categories, $W$ a subset of morphisms of $C$ and $V$ a subset of morphisms of $D$. Let $f:C\to D$ be a morphism of simplicial categories such that $f(W)\subset V$. If the induced functor $Lf:L_WC\to L_VD$ is an equivalence of simplicial categories, then the Quillen adjunction 
$$f_!:\sS^{C,W}\rightleftharpoons \sS^{D,V}:f^*$$
 is a Quillen equivalence.
\end{coro}

As usual we can construct a Yoneda lemma for this type of localization too. This result comes from \cite[Theorem 4.2.3]{HAG1}. It is given there for a pseudo-model category.

\begin{teo}\label{ch. 1:fully-faith}Let $C$ a model category that is also a simplicial category and $W$ the set of weak equivalences, then the functor 
$$\R\Sing(-):\Ho(C)\to \Ho(\sS^{C,W}) $$
is fully faithful. 
\end{teo}

\subsection{The universal model category}

And lastly in this section, we will construct a "universal model category" $UC$ for every category $C$, in the sense that for all model category $M$ and all functor $\gamma:C\to M$ there exists a factorization of $\gamma$ by $UC$ which is, in some sense, unique. From here on all results come from \cite{Dugger} unless stated otherwise.

\begin{defin}Let $C$ be a category, let $M$ and $N$ be two model categories. We fix a functor $f:C\to M$. For all $g:C\to N$, we define a \textbf{factorization of $g$ through $M$} to be a triple $(L,R,\eta)$ such that 
\begin{itemize}
	\item The functors $L:M\rightleftharpoons N:R$ form a Quillen adjunction.
	\item We have a weak equivalence $\eta:L\circ f\simeq g$.
\end{itemize}
\end{defin}

\begin{defin}Let $C$ be a category, let $M$ and $N$ be two model categories. We fix a functor $f:C\to M$, and take $g:C\to N$. If we have $(L,R,\eta)$ and $(L',R',\eta')$ two factorizations of $g$ through $M$, we define a \textbf{morphism of factorizations} to be a natural transformation $F:L\to L'$ such that for all $x\in C$ the diagram
\begin{center}
	\begin{tikzcd}
		L\circ f(x)\ar[rr,"F\circ \Id_f"]\ar[rd,"\eta"]& &L'\circ f(x)\ar[ld,"\eta'"]\\
		 & g(x) &
	\end{tikzcd}
\end{center}
commutes.
\end{defin}

\begin{nota}With the above conditions, we denote $\Fact_f(g)$ the category of factorizations of $g$ through $M$ and morphisms between them.
\end{nota}

\begin{prop}\label{ch. 1: Dugger}\cite[Prop. 2.3]{Dugger} Let $C$ be a category and $M$ be a model category. There exists a closed model category $UC$ and a functor $r:C\to UC$ such that the following is true: for every functor $\gamma:C\to M$ there exists a factorization of $\gamma$ through $UC$, $(\Re,\Sing,\eta)$, 
\begin{center}
	\begin{tikzcd}
		C\ar[rr,"r"]\ar[rrdd,"\gamma" {name=A}]&&UC\ar[dd,"\Re"', shift right]\ar[Rightarrow, to=A, "\eta" above]\\
		&  &\\
		 &&M\ar[uu, "\Sing"', shift right]
	\end{tikzcd}
\end{center}
and the category of factorizations $\Fact_r(\gamma)$ is contractible.
\end{prop}

\begin{proof}[Sketch of construction]\cite[9.5 Section 3]{Dugger}
The universal model category $UC$ is no other than $UC=\Fun(C\op,\sS)$, the category of simplicial presheafs over $C$. The inclusion $i:\Set\to \sS$ induces an obvious inclusion $j:\Fun(C\op,\Set)\to \Fun(C\op,\sS)$,  which composed with the Yoneda embedding gives us the needed functor from $C$ to $\Fun(C\op,\sS)$, $r=j\circ h:C\to \Fun(C\op,\Set)\to \Fun(C\op, \sS)$. 
\\
\\Now that we have the closed model category $\Fun(C\op, \Set)$ and the functor $r:C\to \Fun(C\op, \sS)$, we take a functor $\gamma:C\to M$ to another model category. We are going to give the factorization, but we won't prove that it is in fact one. The easier adjoint to define is $\Sing$. In fact, this is the reason why we have introduced cosimplicial framings and mapping spaces. We define $\Sing$ as follows:
$$\Sing(x)=\Map( \gamma(-),x)\in \Fun(C\op, \sS). $$
The functor $\Re$ is just the left adjoint of $\Sing$. In particular, it looks as follows:
\begin{center}
	\begin{tikzcd}
		\Re(F)=(C^*\circ \gamma)\otimes_C F=\coeq\ \l\coprod_{a\to b}(C^*\circ\gamma) (a)\otimes F(b)\right.\ar[r, shift left]\ar[r, shift right] & \left.\coprod_c (C^*\circ \gamma)(c)\otimes F(c)\r
	\end{tikzcd}
\end{center}
where the tensor product is the one given in Proposition \ref{ch. 1:tensor}.
\\\end{proof}

\begin{rem}The fact that these functors are called $\Re$ and $\Sing$ is obviously not a coincidence: indeed, the idea comes from the construction of a cocomplete category from $\Delta$, which gives us the adjoints $\Re:\sS=\Fun(\Delta,\Set)\rightleftharpoons \Top:\Sing$, the geometric realization and the singular functor respectively.
\end{rem}

\begin{rem}The advantage of this construction is that $\Fun(C\op, \sS)$ is a particularly well-behaved model category, as it inherits most of $\sS$'s good properties. In particular, $\Fun(C\op, \sS)$ is simplicial, cofibrantly generated and proper. 
\end{rem}

\begin{rem}The representable functors $rX$ are always cofibrant, so the images of $\Re\circ r$ are all cofibrant. That is the reason why we have asked for the factorization to be commutative up to a certain weak equivalence: unless the functor $\gamma$ sent all objects to cofibrant objects in $M$, the diagram wouldn't commute strictly.
\end{rem}

As in the classical case, where we know that every presheaf can be expressed as a colimit of representables, we can now prove that every object in $\Fun(C\op, \sS)$ is a homotopy colimit of representables. Let us construct the representables needed for it.

\begin{defin}Let $C$ be a category, and $F$ an object in $\Fun(C\op, \Set)$, i.e. a classical presheaf. We can define a \textbf{simplicial presheaf associated to $F$} by using the following formula:
$$(\tilde{Q}F)_n=\coprod_{rX_n\to\ldots\to rX_0\to F} rX_n.$$ 
\end{defin}

\begin{defin}Let $C$ be a category, and $F_*$ an object in $\Fun(C\op, \sS)$, i.e. an arbitrary simplicial presheaf. The construction $\tilde{Q}F_*$ gives us a bisimplicial presheaf, $(\tilde{Q}F_*)_n=\tilde{Q}F_n. $ We can define the following simplicial presheaf:
$$QF=\diag (\tilde{Q}F_*), $$
the diagonal of the bisimplicial presheaf $\tilde{Q}F_*$.
\end{defin}

\begin{prop}Let $C$ be a category and $F_*$ an object in $\Fun(C\op,\sS)$, i.e. an arbitrary simplicial presheaf. Then the associated simplicial presheaf $QF$ is cofibrant, and the natural map $QF\to F_*$ is a weak equivalence. In particular, every simplicial presheaf is an homotopy colimit of representables.
\end{prop}

\section{Differential graded categories}\label{Dg-categories}

Before we get to the core result of this thesis, we need to recall some basic results and notations concerning the main object we will be talking about, dg-categories.

\subsection{Definition and model structure}

Let us start by defining our terms. What is a differential graded category, and how do we give it a model structure? The results in this section should ring a familiar bell on the reader, as both these and the simplicial categories are common examples of a more general setting, enriched categories. In consequence, the definitions will be pretty similar. Unless stated otherwise, the results from this section are taken from \cite{Toen-dg}, but \cite{Keller} is also  a very good reference for it.

\begin{defin}We define $T$ a \textbf{dg-category} (differential graded category) to be a category enriched over $\Ch$ the category of cochain complexes. Equivalently, a dg-category consists of the following data: 
\begin{itemize}
	\item A set of objects $\Ob(T)$.
	\item For every pair of objects in $T$, $(x,y)\in\Ob(T)^2$, a cochain complex $\Hom(x,y)\in \Ch$. 
	\item For every triple of objects in $T$, $(x,y,z)\in\Ob(T)^3$ a composition morphism in $\Ch$
	$$\mu: \Hom(x,y)\otimes\Hom(y,z)\to \Hom(x,z) $$
	with the usual associativity condition.
	\item For every object in $T$, $x\in T$, a morphism $e_x: k\to \Hom(x,x)$ that satisfies the usual unit condition with respect to the composition stated above, where $k$ is the dg-category with a single object and $k$ as its complex of morphisms.
\end{itemize}
\end{defin} 

\begin{defin} Let $T$ and $T'$ be two dg-categories. A \textbf{dg-functor} (also called a\textbf{ morphism of dg-categories}) is a functor $f: T\to T'$ enriched over the category of complexes. Equivalently, it consists of the following data: 
\begin{itemize}
	\item A map of sets $\Ob(T)\to \Ob(T')$.
	\item For every pair of objects in $T$, $(x,y)\in Ob(T)^2$, a morphism of complexes
	$$\Hom(x,y)\to \Hom(f(x), f(y)). $$
	satisfying the usual associativity and unit conditions.
\end{itemize}  
\end{defin}

\begin{nota} We denote $\dg$ the category of dg-categories and dg-functors.
\end{nota}

For any dg-category $T$, we can define an associated category: the homotopy category of $T$. 

\begin{defin}Let $T$ be a dg-category. We call the \textbf{homotopy category of $T$}, and we denote it by $\l T\r$, a category which has the same objects as $T$ and whose morphisms are given by 
$$\l T \r(x,y)=H^0(\Hom_T(x,y))\ \ \forall (x,y)\in \Ob(T)^2, $$
i.e. the cohomology groups of degree 0 of the complex of morphisms.
\\
\\The composition in this category is given for all $(x,y,z)\in \Ob(T)^3$ by the composition of morphisms
$$H^0(\Hom_T(x,y))\otimes H^0(\Hom_T(y,z))\to H^0(\Hom_T(x,y)\otimes\Hom_T(y,z))\to H^0(\Hom_T(x,z)). $$
\end{defin}

\begin{rem}We have already defined three different things called "the homotopy category of" something. Although there are certain links between them, it is important to remember that those definitions are not interchangeable: it is for that reason that it is essential to keep in mind what it is that we are taking the homotopy category of. A model category, a simplicial category or a dg-category?
\end{rem}

It has been proven that $\dg$ has a model structure, and even a cofibrantly generated model structure. It is defined as follows: 

\begin{defin}Let $f:T\to T'$ be a morphism of dg-categories. 
\begin{itemize}
	\item We say that $f$ is \textbf{quasi-essentially surjective} if the induced morphism of homotopy categories, $\l f\r: \l T\r\to \l T'\r$ is essentially surjective.
	\item We say that $f$ is \textbf{quasi-fully faithful} if for any two objects $(x,y)\in \Ob(T)^2$ the corresponding morphism of complexes $T(x,y)\to T(f(x),f(y))$ is a weak equivalence of complexes.
	\item We say that $f$ is a \textbf{quasi-equivalence} if it is quasi-essentially surjective and quasi-fully faithful. 
\end{itemize}
\end{defin}

\begin{defin}\cite[Not. 2.5]{TAB} Let $T$ be a dg-category. We say that a morphism in $T$, $f\in Z^0(\Hom_T(x,y))$, is a \textbf{homotopy equivalence} if it becomes an isomorphism $H^0(f)$ in $\l T\r$. 
\end{defin}

\begin{defin}\cite[Def. 2.12]{TAB} Let $f:T\to T'$ be a morphism of dg-categories. We say that $f$ is a \textbf{fibration} if 
\begin{itemize}
	\item for every two objects $(x,y)\in \Ob(T)^2$, the corresponding morphism of complexes $T(x,y)\to T'(f(x),f(y))$ is a fibration of complexes, i.e. is surjective.
	\item For all $x\in \Ob(T)$,  $y'\in \Ob(T')$, and all $h:f(x)\to y'$ homotopy equivalence in $T'$, there exists an object $y\in \Ob(T)$ and a homotopy equivalence $g:x\to y$ in $T$ such that $f(y)=y'$ and $f(g)=h$. 
\end{itemize}
\end{defin}

\begin{teo}\label{ch. 1:Tab-model} (\cite[Def. 2.14]{TAB}, see \cite[Th. 2.1]{TAB_Fr} for a proof in French) The category $\dg$ admits a model structure with the quasi-equivalences as weak equivalences and the fibrations as defined above. It is a cofibrantly generated model category, and the generating cofibrations $\{I, P(s)/\ s\in\Z\}$ are the following:
\begin{itemize}
	\item The functor $I$ is the unique dg-functor $\0\to k$, where $\0$ is the initial object in $\dg$. 
	\item For all $s\in\Z$, let $\D(1,s,1)$ be the dg-category with two objects, $0$ and $1$, where $\Hom(0,0)=\Hom(1,1)=k$, $\Hom(1,0)=0$ and $\Hom(0,1)=k[s]$; and let $\D^c(1,s,1)$ be the dg-category with two objects, $0$ and $1$, where $\Hom(0,0)=\Hom(1,1)=k$, $\Hom(1,0)=0$ and $\Hom(0,1)=k^c\l s\r$. The $P(s):\D(1,s,1)\to\D^c(1,s,1)$ are, for all $s\in\Z$, the dg-functors that send $0$ to $0$, $1$ to $1$, and $\Hom_{\D(1,s,1)}(0,1)$ to $\Hom_{\D^c(1,s,1)}(0,1)$ by the following morphism:
	\begin{center}
		\begin{tikzcd}
			\ldots\ar[r]&0\ar[r]\ar[d]&0\ar[r]\ar[d]& k\ar[r]\ar[d, "id"]& 0\ar[r]\ar[d]&\ldots\\
			\ldots\ar[r]&0\ar[r]      &k\ar[r]      &k\ar[r]             &0\ar[r]       &\ldots
		\end{tikzcd}
	\end{center}
\end{itemize}
\end{teo}

\begin{coro}\label{ch. 1:hom-cofib}With that model structure, all dg-categories are fibrant. Also, if a dg-category $T$ is cofibrant, for all $(x,y)\in\Ob(T)^2$ the complex $\Hom_T(x,y)$ is cofibrant for the model structure on $\Ch$.
\end{coro}

\begin{rem}Both Bergner's model structure on simplicial categories and Tabuada's model structure on dg-categories (among quite a few others) can be seen as particular cases of the canonical model structure on the category of categories enriched over $N$, $N\mathbf{-cat}$, where $N$ is a right proper, adequate monoidal model category with cofibrant unit and a generating set of $N$-intervals. A full construction and proof of that can be found in \cite[Th. 2.5]{general-model-cat}.
\end{rem}

\begin{prop}\cite[Th. 1.10]{general-model-cat} With the model structure described above, the model category $\dg$ is right proper.
\end{prop}

Before we go on, we will give a different, equivalent way of defining quasi-fully faithful dg-functors in $\dg$. For this, we will recall a definition that exists in general in any model category. 

\begin{defin}Let $M$ be a model category and let $f:X\to Y$ be a morphism in $M$. We say that $f$ is \textbf{a homotopy monomorphism} if for all $Z\in M$ the induced morphism of simplicial sets
$$f_*:\Map(Z,X)\to \Map(Z,Y) $$
induces an injection on the $\pi_0$ and isomorphisms on $\pi_i$ for all $i>0$ and for all base points. This definition is equivalent to asking that the morphism $X\to X\times_Y^h X$ is an isomorphism in $\Ho(M)$. 
\end{defin}

In the case of dg-categories, we haven't actually defined anything new: we can prove that those dg-functors are exactly the quasi-fully faithful dg-functors.

\begin{prop}\cite[Lem. 2.4]{Toen} Let $f:T\to T'$ be a dg-functor in $\dg$. Then $f$ is a homotopy monomorphism if and only if $f$ is quasi-fully faithful. In particular, if $f$ is a weak equivalence in $\dg$, then $f$ is a homotopy monomorphism.
\end{prop} 

\begin{coro}\cite[Cor. 2.5]{Toen} Let $f:T\to T'$ be a quasi-fully faithful morphism in $\dg$, and let $Y\in \dg$ be a dg-category. Then the image of the induced injection on $\pi_0$, 
$$\pi_0(\Map(Y,T))=\l Y,T \r\to \pi_0(\Map(Y,T'))=\l Y,T'\r ,$$
is given by the morphisms such that the induced functor $\l Y \r\to \l T'\r$ can be factored through the essential image of $\l Y\r\to \l T\r$.
\end{coro}

And now, to finish this section, we will give a couple of definitions for subcategories who will be useful later.

\begin{defin}\label{ch. 1: def free}There exists a Quillen adjunction $L:Gr(\Ch) \rightleftharpoons \dg:U$ where $Gr(\Ch)$ is the category of graphs over the complexes of modules and $U$ is the forgetful functor. We call a \textbf{free dg-category} $T$ a dg-category such that there exists a $T'\in Gr(\Ch)$ with $T=L(T')$. We denote the full subcategory of free dg-categories by $\mathcal{L}$.
\\In particular, if $X$ is a graph on $\Ch$, then we have that $L(X)$ has the same objects as $X$ and that for all objects $x,y\in \Ob(X)=\mathcal{O}$,
$$L(X)(x,y)=\bigoplus_{m\in\N}\bigoplus_{(x_1,\ldots, x_m)\in\mathcal{O}^m}(X(x,x_1)\otimes \ldots \otimes X(x_m,y)). $$
\end{defin}

But still, free dg-categories will be too big for our purposes. We will use something slightly smaller: free dg-categories of finite type. As we are working with free dg-categories, we can define these objects as graphs and then send them over with the free functor.

\begin{defin}Let $G\in Gr(\Ch)$ be a graph over the complexes of modules. We say that $G$ is \textbf{a graph of finite type} if it has a finite number of vertices and the edges between two vertices are always perfect complexes. We denote the full subcategory of graphs of finite type by $Gr(\Ch)^{tf}$.
\end{defin}

\begin{defin}Let $L=L(G)$ be a free dg-category. We say that $L$ is a \textbf{free dg-category of finite type} if the underlying graph $G$ is a graph of finite type.
\end{defin}

\subsection{Modules over dg-categories}

And as the last section of this chapter, let us talk about modules over dg-categories. Unless stated otherwise, all results in this section come from \cite[Sections 3 and 4]{Toen}.

\begin{defin}Let $T$ be a dg-category. We call a \textbf{$T$-dg-module} (or a $T$-module if it is obvious we are working on $\dg$) a dg-functor $F$ of the form $F:T\to \Ch$. If we have $F,G$ two $T$-modules, we define a morphism of $T$-modules between $F$ and $G$ to be an enriched natural transformation.
\end{defin}

\begin{nota}Let $T$ be a dg-category. We denote by $T-\Mod$ the category of $T$-modules.
\end{nota}

\begin{rem}As $\Ch$ is a cofibrantly generated model category, then $T-\Mod$ is a model category with the projective structure given in Theorem \ref{ch. 1:projective}. 
\end{rem}

We can define some kind of dg-enhanced Yoneda embedding in this case. For that, though, we will need to introduce one more definition: that of an internal category. 

\begin{defin}Let $T$ be a dg-category and $T-\Mod$ its category of $T$-modules. We define the \textbf{internal category of $T-\Mod$} to be the sub-dg-category where the objects are the fibrant and cofibrant objects of $T-\Mod$, and we will denote it by $\Int(T-\Mod)$.
\end{defin}

\begin{rem}In \cite{Toen} the definition is given for any $M$ $\Ch$-model category, but as we are only going to use it in this context we have decided to only state it in the case of $T-\Mod$. In fact, this next result is also true for all $T-\Mod$. 
\end{rem}

As we remember from the section on the homotopy category of a model category $M$, the homotopy category can be defined as a quotient of the subcategory $M^{cf}$ of fibrant and cofibrant objects; the internal category could be seen as a $\Ch$-enrichment of $T-\Mod^{cf}$. When we defined the homotopy functor in $\dg$ that sends dg-categories to classical categories we warned that although they had the same name, the homotopy category of a model category and that of a dg-category are not the same definition; although, we added, there are links between the two. Here is the link.

\begin{prop}Let $T$ be a dg-category. Then the homotopy category of $\Int(T-\Mod)$ as a dg-category is naturally isomorphic to the homotopy category of $T-\Mod$ as a model category,
$$\l\Int(T-\Mod) \r\simeq \Ho(T-\Mod). $$  
\end{prop}

And now we can define the Yoneda embedding.

\begin{nota}Let $T$ be a dg-category, and $x$ be an object in $T$. We denote by $\underline{h}_x:T\op\to \Ch$ the dg-functor $\underline{h}_x(y)=\Hom_T(y,x)$ where we take the cochain complex of morphisms in $T$. 
\end{nota}

\begin{defin}Let $T$ be a dg-category. We define a morphism of dg-categories $\underline{h}:~T\to T\op-\Mod$ by $\underline{h}(x)=\underline{h}_x$ by using the natural $\Ch$-enrichment of $T\op-\Mod$.
\end{defin}

\begin{prop}Let $T$ be a dg-category. For all $x\in T$, the object $\underline{h}_x$ in $T\op-Mod$ is fibrant and cofibrant. That defines a dg-functor
$$\underline{h}:T\to \Int(T\op-\Mod)$$
and it is quasi-fully faithful. 
\end{prop}

\begin{defin}Let $T$ be a dg-category, and let $F$ be a $T\op$-module.
\begin{itemize}
	\item We say that $F$ is \textbf{representable } if there exists an object $x\in T$ in $T$ such that $F$ is isomorphic in $\Int(T\op-\Mod)$ to $\underline{h}_x$.
	\item We say that $F$ is \textbf{quasi-representable} if there exists an object $x\in T$ in $T$ such that $F$ is weak equivalent in $\Int(T\op-\Mod)$ to $\underline{h}_x$, i.e. if $F$ is isomorphic to $\underline{h}_x$ in $\Ho(T\op-\Mod)$. 
\end{itemize}
\end{defin}

\begin{rem}As $\underline{h}_x$ is quasi-fully faithful, this definition means that there exists a weak equivalence in $\dg$ between $T$ and the subcategory of $\Int(T\op-\Mod)$ consisting of quasi-representable objects.
\end{rem}

We know that $\Ch$ has a tensor product, that we have denoted $-\otimes -$: using that, we can easily define a tensor product over the whole category $\dg$. 

\begin{defin}Let $T$ and $T'$ be two dg-categories. We define the tensor product of $T$ and $T'$ as a category $T\otimes T'$ such that
\begin{itemize}
	\item The objects in $T\otimes T'$ are the objects in $\Ob(T)\times \Ob(T')$. 
	\item For every pair of objects $(x,y), (x',y')$ in $T\otimes T'$, a cochain complex of the form
	$$T'\otimes T'((x,y), (x',y'))=T(x,x')\otimes T'(y,y'). $$
\end{itemize}
\end{defin}

\begin{prop}The tensor product defined above gives $\dg$ a closed symmetric monoidal structure on $\dg$. The unit for the monoidal structure is the dg-category with one object and $k$ as its complex of morphisms, and we will denote it by $\D(0)$ or $k$.
\end{prop}

\begin{rem}This symmetric monoidal structure is unfortunately not compatible with the model structure, which means we don't have a symmetric monoidal model category: indeed, the tensor product of two cofibrant objects isn't necessarily cofibrant itself. The tensor product does have enough good properties, though, so that it can be derived. 
\end{rem}

\begin{prop}The tensor product functor on $\dg$ can be derived into a functor $-\otimes^\L -:\Ho(\dg)\times\Ho(\dg)\to\Ho(\dg)$ as follows: let $T$ and $D$ be two dg-categories. We compute the derived tensor product as
$$T\otimes^\L D=Q(T)\otimes D $$
where $Q$ is a cofibrant replacement that is the identity on the objects. When $T$ is cofibrant there exists a natural quasi-equivalence $T\otimes^\L D\to T\otimes D$. 
\end{prop}

We can now consider the category of $(T\otimes D\op)$-modules. For any $x\in T$, there exists a natural dg-functor $j_x:D \op\to T\otimes D\op$ that sends an object $y\in D\op$ to $(x,y)\in T\otimes D\op$, and, for every pair of objects $y,z\in D\op$, the morphism of complexes $D\op(y,z)$ to $T(x,x)\otimes D\op(y,z)$ by $k\otimes \Id_{D\op}$. 

\begin{defin}Let $T, D$ be two dg-categories. Then for all $x\in T$ we define a dg-functor 
$$i_x:D\op\to Q(T)\otimes^\L D\op$$
induced by $j_x$
\end{defin}

\begin{defin}Let $T$ and $D$ be two dg-categories. An object $F\in (T\otimes^\L D\op)-\Mod$ is called \textbf{right quasi-representable} if for all $x\in T$, $i_{x}^*(F)$ is quasi-representable.
\end{defin}

\begin{rem}If we have $T=k$, it is easy to see that an object $F\in (k\otimes^L D\op)-\Mod=D\op-\Mod$ is right quasi-representable if and only if it is quasi representable.
\end{rem}

\begin{defin}Let $T$ and $D$ be two dg-categories. We denote by $F(Q(T), D)$ the (non-full) subcategory of $\dg$ defined as follows:
\begin{itemize}
	\item A set of objects given by the $Q(T)\otimes D\op$-modules that are right quasi-representable.
	\item The weak equivalences for the projective model structure of $(Q(T)\otimes D\op)-\Mod$ between right quasi-representable modules.
\end{itemize}
\end{defin}

And we can now write the results that we were aiming for:

\begin{teo}\cite[Th. 4.2]{Toen}\label{ch. 1:map k} Let $T$ and $D$ be two dg-categories. There are weak equivalences in $\sS$ such that
$$\Map(T,D)\to \diag(N(M(C^*(T),D)))\leftarrow N(F(Q(T),D))$$
where $N$ stands for the nerve of a category and for all $n\in\N$, $M(C^n(T),D)$ is the subcategory of $F(C^n(T),D)$ where the objects are $T\otimes D\op$-modules $F$ that are right quasi-representable and for all $x\in C^n(T)$ the $D\op$-module $F(x,-)$ is cofibrant in $D\op-\Mod$.
\end{teo}

\begin{coro}\label{ch. 1:map to iso classes }Let $T$ be a dg-category. Then there exists a functorial isomorphism between the set of maps $\l k,T \r$ in $\Ho(\dg)$ and the set of isomorphism classes of the category $\l T\r$.
\end{coro}

\begin{coro}\label{ch. 1:pi(k,T)}Let $T$ be a dg-category, and let $x\in T$ be an object in $T$. Then there are natural isomorphisms of groups
$$\pi_1(\Map(k,T),x)\simeq \Aut_{\l T\r}(x) $$
$$\pi_i(\Map(k,T),x)\simeq H^{1-i}(T(x,x))\ \ \forall i>1. $$
\end{coro}

\chapter{dg-Segal spaces}

\epigraph{\textit{"Once one decides on what the correct definitions and theorems are, then the results almost prove themselves."}}{---Daniel Dugger, \textit{Universal Homotopy Theories}}

And now that we have all the background information we needed, we can start defining the objects we will be using for our main results. We fix $k$ a commutative ring.

\section{Constructing the adjunction}

\begin{nota}We denote the full subcategory of cofibrant free dg-categories of finite type by $\Free\subset\mathcal{L}$. As we won't consider any other type in this text, we will most of the time omit the term "cofibrant" from our explanations.
\end{nota}

\begin{defin}Let $W$ be the set of weak equivalences on $\dg$ the category of dg-categories. We construct the simplicial localization $L_W\dg$ of dg-categories with respect to $W$, the weak equivalences, as in Definition \ref{Ch. 1: def localization}. We define the \textbf{simplicial $\Free$}, and we denote it by $\FreeS$, the full simplicial subcategory of $L_W\dg$ whose objects are the ones in $\Free$.
\end{defin}

\begin{rem}We must be careful with the definition of $\FreeS$. It is tempting to just define it as $L_W\Free$, but the two categories $L_W\Free$ and $\FreeS$ do not coincide. 
\end{rem}

\begin{teo}\label{Ch.2: construction Sing_w}There exists a chain of Quillen adjunctions of the form 
$$\Re:\Fun^\mathbb{S} (\FreeS\op, \sS)\rightleftharpoons \ldots\rightleftharpoons\dg:\Sing,$$
and it can be derived into a single adjunction
$$\Ho(\Fun^\mathbb{S} (\FreeS^{op}, \sS))\rightleftharpoons \Ho(\dg). $$
\end{teo}

\begin{proof}
Let us start from the right. As we have defined $\FreeS$ to have the simplicial structure induced by the simplicial structure of $L_W\dg$, there exists a natural simplicial inclusion functor $j:\FreeS\hookrightarrow L_W\dg$. If we take the projective model structure on the categories of simplicial functors, we have a Quillen adjunction of the form 
$$j_!:\Fun^\S(\FreeS\op, \sS)\rightleftharpoons \Fun^\S(L_W\dg\op,\sS):j^* $$
Now, by definition of the simplicial localization, there exists a functor $l:\dg\to L_W\dg$ that gives us a Quillen adjunction of the form 
$$l_!:\Fun(\dg\op, \sS)\rightleftharpoons \Fun^\S(L_W\dg\op, \sS):l^* $$
And lastly, using Dugger's construction from Proposition \ref{ch. 1: Dugger}, if we take both the category $C$ and the model category $M$ to be $\dg$, and the functor $\gamma$ to be the identity, we have a factorization $(\Re_W, \Sing_W,\eta)$,
$$\Re_W:\Fun(\dg\op, \sS)\rightleftharpoons \dg:\Sing_W $$
where we know that the right adjoint is given by $\Sing_W(X)=\Map(-,X)$.
\\
\\We have constructed a chain of Quillen adjunctions from $\dg$ to $\Fun(\FreeS\op,\sS)$. But, as we have remarked a few times before, we cannot compose right Quillen functors to left Quillen functors and still get a Quillen adjunction. And in this case, the adjunction $(l_!,l^*)$ goes in the "wrong direction". Indeed, if we write the chain of Quillen adjunctions and we write the left Quillen functor always on top, we get the following diagram:
\begin{center}
	\begin{tikzcd}
		\Fun^\S(\FreeS\op,\sS)\ar[r, "j_!", shift left]& \ar[l, "j^*", shift left]\Fun^\S(L_W\dg\op,\sS)\ar[r,  "l^*" below, shift right]& \ar[l, "l_!" above, shift right]\Fun(\dg\op,\sS)						\ar[r,"\Re_W", shift left]& \ar[l,"\Sing_W", shift left] \dg.
	\end{tikzcd}
\end{center}
At this level there is nothing we can do to fix this: but on the homotopy categories there is. Indeed, in the homotopy categories we can construct a factorization of the functor $\Sing_W$, which would bypass the need for the adjoint $l_!$ altogether. In other words, we are going to try and find a functor $\Sing'$ such that the diagram 
\begin{center}
	\begin{tikzcd}
		\Ho(\Fun^\S(\FreeS\op,\sS))\ar[r, "j_!", shift left]& \ar[l, "j^*", shift left]\Ho(\Fun^\S(L_W\dg\op,\sS))\ar[r,  "l^*" below, shift right]& \ar[l, "l_!" above, shift right]\Ho(\Fun(\dg\op,\sS))						\ar[d,"\Re_W", shift left]\\
		&& \ar[u,"\Sing_W", shift left] \Ho(\dg)\ar[ul, "\Sing'", dashed].
	\end{tikzcd}
\end{center}
 commutes. 
\\
\\By definition of the simplicial localization category, if we have a functor $F:\dg\to \sS$ that sends all morphisms in $W$ to weak equivalences in $\sS$, then it can be factorized through $l^*$ on the homotopy categories. But $\Sing_W(X)=\Map(-,X)$ is a right Quillen adjoint, so it sends weak equivalences between fibrant objects to weak equivalences, and all objects in $\dg$ are fibrant. So for all $X\in\dg$, the image $\Sing_W(X)$ can be factorized through $L_W\dg$. That gives us a functorial factorization of the form
$$\Sing_W(X)=l^*\circ \Sing'(X) $$
in the homotopy categories.
\\
\\Now we only need to prove that the functors $\Sing=j^*\circ\Sing'$ and $\Re=\Re_W\circ l^*\circ j_!$ are truly adjoints, i.e. that for all $X\in\Fun^\S(\FreeS\op,\sS)$ and for all $Y\in\dg$ there exists an functorial isomorphism of the form
$$\l X,j^*\circ \Sing'(Y) \r_{\Fun^\S(\FreeS\op,\sS)}\simeq \l\Re_W\circ l^*\circ j_!(X),Y\r_{\dg}. $$
Let us start on the left side and work our way through. We start with $\l X,j^*\circ \Sing'(Y) \r_{\Fun^\S(\FreeS\op,\sS)}$. As the functors $j_!$ and $j^*$ are adjoints, we have that 
$$\l X,j^*\circ \Sing'(Y) \r_{\Fun^\S(\FreeS\op,\sS)}\simeq \l j_!(X),\Sing'(Y)\r_{\Fun^\S(L_W\dg\op,\sS)}. $$
Because we are working on the homotopy categories, by the definition of simplicial localizations (see Definition \ref{Ch. 1: def localization}) the functor $l^*$ is fully faithful. That means, in particular, that we get the following isomorphism:
$$\l j_!(X),\Sing'(Y)\r_{\Fun^\S(L_W\dg\op,\sS)}\simeq  \l l^*\circ j_!(X),\Sing_W(Y)\r_{\Fun(\dg\op,\sS)}.$$
And finally, using the definition of an adjunction again on $\Re_W$ and $\Sing_W$, we get that
$$\l l^*\circ j_!(X),\Sing_W(Y)\r_{\Fun(\dg\op,\sS)}\simeq \l\Re_W\circ l^*\circ j_!(X),Y \r_\dg. $$
We have the isomorphism
$$\l X,\Sing(Y) \r_{\Fun^\S(\FreeS\op,\sS)}\simeq \l\Re(X),Y \r_{\dg}$$
and the pair $\Re:\Ho(\Fun^\S(\FreeS\op,\sS))\rightleftharpoons \Ho(\dg):\Sing$ is an adjunction on the homotopy categories. We have finished the proof.
\\\end{proof}

We have now an adjunction between the categories we wanted itfor. The next step is proving that this functor is an equivalence. But it is not that easy. In order to do that, we need a few background concepts and constructions.
%We have then, an adjunction between the categories we wanted. In order to use it to characterize all dg-categories, though, we still need to prove that the functor $\Sing$ is fully faithful, and characterize $\Sing(\dg)$. Let's get to doing that. 

\section{dg-Segal spaces}

 As we said in the introduction, we take our inspiration for this section from complete Segal spaces. We remind the reader that the definition of said spaces is the following.

\begin{defin}\cite[Def. 4.1]{ComSegalSpacesREZK} Let $W$ be a Reedy fibrant simplicial space. We say that $W$ is a \textbf{Segal space} if the maps 
$$W_k\to \overbrace{W_1\times_{W_0}\ldots\times_{W_0}W_1}^{\text{$k$ times}}$$
are weak equivalences for all $k\geq 2$.
\end{defin}

We can put some similar conditions on our functors, but in our case they won't be enough. Indeed, we aren't just working with simplicial sets: we have a linear structure to worry about. Consequently, we need an additional condition on the shift, and for that, we are going to define the action of adding a module to a complex of modules.

\begin{defin}Let $G\in Gr(\Ch)$ be a graph in the category of complexes, $x,y\in \Ob(G)$ two objects in $G$, and $\alpha\in Z^n(G(x,y))$ a cycle in $G(x,y)$. We define the graph $G(<\alpha>)$ to be a graph of complexes such that 
\begin{itemize}
	\item The graph $G(<\alpha>)$ has the same objects as $G$.
	\item The graph $G(<\alpha>)$ has the same morphisms as $G$ between $x',y'\in\Ob(G)$ if $(x',y')\neq (x,y)$, i.e.\\ $G(<\alpha>)(x',y')=G(x',y')$.
	\item We define $G(<\alpha>)(x,y)$ to be the complex of morphisms $G(x,y)\oplus_k \beta$ where $d\beta=\alpha$.
\end{itemize}
In other words, we have that $G(<\alpha>)$ is a pushout in the graphs over the morphism $k[n]\to k^c[n]$, where $k[n]$ is the graph with two objects, $0,1$, and $k[n]$ as $\Hom(0,1)$ the morphism between the two, and $k^c[n]$ is the graph with two objects, $0,1$, and the complex $\Hom(0,1)$ which is always zero except for the degrees $n-1$ and $n$, where it is $k$.
\begin{center}
	\begin{tikzcd}
		k[n]\ar[r, "\alpha"]\ar[d]& G\ar[d]\\
		k^c[n]\ar[r]& G(<\alpha>).
	\end{tikzcd}
\end{center}

\end{defin}

\begin{rem}It isn't hard to see that all we have done here has been adding a term in degree $n-1$ to the complex of modules $G(x,y)$.
\end{rem}

Now that we have this definition, we can apply it to finally define the conditions of our image.

\begin{defin} Let $F\in\Fun^\S(\Free_\S\op,\sS)$ be a simplicial functor from the cofibrant free dg-categories to the simplicial sets. We say that $F$ satisfies the \textbf{dg-Segal conditions} if:
\begin{enumerate}
	\item For all $L, K\in\Free_\S$, $F(L\coprod K)\to F(L)\times F(K)$ is a weak equivalence. 
	\item The image of the initial object is a point, i.e. $F(\emptyset)\simeq *$.
	\item Let $G$ be a graph in $Gr(\Ch)$ and $x,y\in\Ob(G)$. For all $\alpha\in Z^n(G(x,y))$, the image of the free dg-category issued from $G(<\alpha>)$ is a homotopy pullback in $\sS$ of the following form:
	\begin{center}
		\begin{tikzcd}
			F(L(G(<\alpha>)))\ar[r]\ar[d]\arrow[dr,phantom, "\ulcorner^h", very near start]& F(L(G))\ar[d]\\
			F(\D^c(1,n,1))\ar[r]& F(\D(1,n,1))
		\end{tikzcd}
	\end{center}	
	where $\D^c(1,n,1)=L(k^c[n])$ and $\D(1,n,1)=L(k[n])$; i.e. $F$ sends the homotopy pushouts of the previous definition to homotopy pullbacks.
\end{enumerate}
	We denote the full subcategory of $F\in\Fun^\S(\Free_\S\op,\sS)$ that satisfies the dg-Segal conditions by $\mathbf{dg-Segal}$ and call its objects \textbf{dg-Segal spaces}.
\end{defin}
Our conjecture at this point is that the image of the functor $\Sing$ we defined in Section 4.1 is formed up to weak equivalence of the functors that satisfy the dg-Segal conditions. In order to prove that, first we need to prove that every object in the image is of this form.

\begin{prop}\label{ch.2 Sing is dg-Segal}Let $T\in \dg$ be a dg-category. Then the functor $\Sing(T)$ satisfies the dg-Segal conditions.
\end{prop}

\begin{proof}
 We have to see that $T$ fulfills the three conditions of the definition.  
\\
\\1. Let $L,K\in\Free_\S$ be two cofibrant free dg-categories. As $L,K$ are cofibrant, we have that 
$$\Sing(T)(L\coprod K)=\Map(L\coprod K,T)=\Hom(L\coprod K, C_*(T)).$$
But by definition of coproduct, the condition 1. holds in this case:
$$\Hom(L\coprod K, C_*(T))=\Hom(L,C_*(T))\times \Hom(K,C_*(T))=\Map(L,T)\times \Map(K,T) $$
and we have finished.
\\
\\2. This condition is evident: $\Sing(\emptyset)=\Map(-,\emptyset)=*$ by definition of final object.
\\
\\3. We need to prove that the following diagram is a homotopy pullback
\begin{center}
		\begin{tikzcd}
			\Sing(T)(L(G(<\alpha>)))=\Map(L(G(<\alpha>)), T)\ar[r]\ar[d]& \Sing(T)(L(G))=\Map(L(G),T)\ar[d]\\
			\Sing(T)(\D^c(1,n,1))=\Map(\D^c(1,n,1), T)\ar[r]& \Sing(T)(\D(1,n,1))=\Map(\D(1,n,1),T).
		\end{tikzcd}
	\end{center}
	
All these objects are in $\sS$, which is a proper category, which means that if one of these arrows is a fibration, then we have a homotopy pullback. By Theorem \ref{ch. 1:Tab-model}, we have that $\D(1,s,1)\to \D^c(1,s,1)$ is a generating cofibration in $\dg$, and by Proposition \ref{ch. 1:cof-to-fib}, $\Sing(T)(\D^c(1,s,1))\to \Sing(T)(\D(1,s,1))$ is a fibration. We then have that the previous diagram is a homotopy pullback.
\\
\\The functor $\Sing(T)$ satisfies the dg-Segal conditions and we have finished our proof.
\\\end{proof}

\begin{rem}We draw the attention of our readers to the fact that, even though we haven't asked for dg-Segal spaces to be fibrant, by adjointness every $\Sing(T)$ is actually fibrant.
\end{rem}

We have proven that every element in the image of $\Sing$ is a dg-Segal space. We ask the reader to keep that in mind for when we have to prove the essential surjectivity. 
\\
\\But for now, let us focus on the model structure we can get for these dg-Segal spaces. In \cite{ComSegalSpacesREZK}, Rezk takes his model structure over $\Fun(\Delta\op,\sS)$ and does a Bousfield localization that makes the Segal spaces into its fibrant objects; he calls that \textbf{the Segal space model category structure}. Following his footsteps, we get a new model structure for $\Fun^\S(\Free_\S\op,\sS)$ where the fibrant objects are the dg-Segal spaces that are fibrant for the projective structure. Even though the results about Segal and complete Segal spaces in this section are due to Rezk, we will take inspiration in Rasekh's lecture notes in \cite{Nima} for their presentation.

\begin{teo}There exists a simplicial closed model structure on $\Fun^\S(\Free_\S\op,\sS)$ and a class of morphisms $C$ such that
\begin{enumerate}
	\item The cofibrations are the same as in the projective model structure.
	\item The fibrant objects are the dg-Segal spaces which are fibrant for the projective model structure.
	\item The weak equivalences are the $C$-local equivalences with respect to the class $C$.
\end{enumerate}
We call such a model structure \textbf{the dg-Segal model structure.}
\end{teo}

\begin{proof}
In order to prove this, we are going to utilize the left Bousfield localization. The first thing we need to do is find a class of morphisms $C$ such that the dg-Segal spaces are exactly the $C$-local objects, i.e. such that for every dg-Segal space $F$ and for every morphism $f:A\to B$ in $C$, the morphism 
$$\Map(B,F)\to\Map(A,F) $$
is a weak equivalence. For that, we define three classes of morphisms, one for each condition in the definition of a dg-Segal space.
\\
\\1. Our first class of morphisms is $C_1=\{\Sing(L)\coprod\Sing(L')\to \Sing(L\coprod L')/\ L,L'\in \FreeS\}$. Let us check that the $C_1$-local objects are exactly the functors that satisfy the first condition of the definition of dg-Segal spaces. Let $F$ be a functor in $\Fun^\S(\Free_\S\op,\sS)$ that is a $C_1$-local object. Then, by the Yoneda Lemma, we have that 
$$\Map(\Sing(L\coprod L'),F)\simeq F(L\coprod L')\to \Map(\Sing(L)\coprod\Sing(L'),F)\simeq F(L)\times F(L') $$
is a weak equivalence. By construction, $F$ satisfies the first dg-Segal condition. 
\\
\\2. As the second condition of the definition of a dg-Segal space is just one single weak equivalence, the class of morphisms associated to it will also have just one morphism. We consider $C_2=\{\emptyset\to \Sing(\emptyset\}$. Let $F$ be a $C_2$-local object. Then, also by the Yoneda lemma, we have that
$$\Map(\Sing(\emptyset),F)\simeq F(\emptyset)\to \Map(\emptyset,F)\simeq * $$
is a weak equivalence, and $F$ satisfies the second dg-Segal condition.
\\
\\3. Lastly, we take the class $C_3=\{\Sing(L(G))\coprod_{\Sing(\D(1,s,1))}\Sing(\D^c(1,s,1))\to \Sing(L(G(<\alpha>))/\ G\in Gr(\Ch),s\in\Z, x,y\in\Ob(G), \alpha\in Z^n(x,y)\}$. Let $F$ be a $C_3$-local object. The, by the Yoneda lemma, we have that the morphism 
$$\Map(\Sing(L(G(<\alpha>))),F)\simeq F(L(G(<\alpha>)))\to $$
$$\Map(\Sing(L(G))\coprod_{\Sing(\D(1,n,1))}\Sing(\D^c(1,n,1)), F)\simeq F(L(G))\times_{\Sing(\D(1,n,1))}F(\D^c(1,n,1)) $$
is a weak equivalence. Alternatively, that means that the diagram in condition 3 of the dg-Segal condition is a homotopy pullback and $F$ satisfies the third dg-Segal condition.
\\
\\We take the class of morphisms $C=C_1\cup C_2\cup C_3$ to be our $C$ in the Bousfield localization. If such a localization exists, its fibrant objects will be exactly the dg-Segal spaces which are fibrant for the projective model structure.
\\
\\We have by Theorem \ref{ch. 1: existence Bousfield} that if the category we are trying to localize is a left proper cellular model category, then the left Bousfield localization exists. We know from Proposition \ref{ch. 1: sSet good Bousfield} that the category of simplicial sets $\sS$ is left proper and cellular, and from Proposition \ref{ch. 1: fun good Bousfield} that the functors on it are also left proper and cellular. So this localization exists and we have finished.
\\\end{proof}

Considering how we have followed Rezk's method pretty closely, it won't be surprising to our readers to see that there is a close relationship between our dg-Segal spaces and classic Segal spaces. Indeed, there is a Quillen adjoint between the model category $\Fun^\S(\Free_\S\op,\sS)$ and the model category $\Fun(\Delta\op, \sS)$. Let us construct that.

\begin{prop}There exists a morphism, called \textbf{the linearisation of $\Delta$}, between the categories $\Delta$ and $\Free_\S$, and it defines a Quillen adjunction between the categories $\Fun^\S(\Free_\S\op,\sS)$ and $\Fun(\Delta\op, \sS)$ with their respective projective structures.
\end{prop}

\begin{proof}
Let $[n]\in\Delta$ be an object in $\Delta$. We define $j([n])=[n]\times k$ to be a free dg-category such that 
\begin{center}
	\begin{tikzcd}
		j([n])=0\ar[r,"k"]& 1\ar[r,"k"]&\ldots\ar[r,"k"]&n.
	\end{tikzcd}
\end{center}
This is a free dg-category of finite type, and it is also cofibrant (for a detailed proof of the cofibrancy, see Corollary \ref{ch. 3: Delta cofibrant}), so this morphism $j$ is well-defined as $\Delta\to \Free_\S$. We construct then the following Quillen adjunction:
$$j_!:\Fun(\Delta\op,\sS)\simeq \Fun^\S(\Delta\op, \sS)\rightleftharpoons \Fun^\S(\Free\op_\S,\sS):j^* $$
and we have finished our proof.
\\\end{proof}

\begin{nota}
    We denote the images by the linearization morphism by $j(\l n\r)=\D(n,0,1)$. 
\end{nota}

\begin{defin}
    We call the morphism $j^*: \Fun^\S(\Free\op_\S,\sS)\to \Fun^\S(\Delta\op, \sS)$ the \textbf{delinearlisation morphism}.
\end{defin}

Now, we have calculated that Quillen adjunction for the projective model structure. But we have two localizations here: let us prove that this stays a Quillen adjunction in the localizations. Let us see that this adjunction sends dg-Segal spaces to classic Segal spaces.

\begin{prop}\label{ch. 2: dg-Segal are Segal}Let $F\in\dgS$ be a dg-Segal space. Then its image by the delinearisation morphism $j^*$ is a Segal space.
\end{prop}

\begin{proof}
Let $F$ be a dg-Segal space. Then, in order to prove that its image by $j^*$ is a Segal space, by definition of Segal space we need to prove that for all $n\geq 1$, the morphism
$$j^*(F)_n\to  \overbrace{j^*(F)_1\times_{j^*(F)_0}\ldots\times_{j^*(F)_0}j^*(F)_1}^{\text{$n$ times}}$$
is a weak equivalence. If we unravel that definition, we have that for every $i\in\N$, $j^*(F)_i=j^*(F)([i])=F(j([i]))=F(\D(i,0,1))$. So proving that $j^*(F)$ is a Segal space can be rewritten as asking that for all $n\geq 1$,
$$\Phi_n:F(\D(n,0,1))\to F(\D(1,0,1))\times_{F(\D(0,0,1))}\ldots\times_{F(\D(0,0,1))}F(\D(1,0,1))$$
is a weak equivalence of simplicial spaces.
\\
\\In order to simplify the notation, we will denote $\D(i,0,1)$ by $\D^i$. We remark too that $F(\D^0)=F(k)$.
\\
\\We will prove the proposition by induction.
\\
\begin{itemize}
	\item $n=1$. This is obvious, since $F(\D^1)\simeq F(\D^1)$. There is nothing to prove.
	\item $n\geq 2$. We assume that the morphism 
	$$\Phi_{n-1}: F(\D^{n-1})\to  \overbrace{F(\D^1)\times_{F(k)}\ldots\times_{F(k)}F(\D^1)}^{\text{$n-1$ times}} $$
	is a weak equivalence. Let us prove that the morphism $\Phi_n$ is also a weak equivalence.
	\\
	\\As usual for inductions, we have to decompose $F(\D^n)$ in a way that makes $F(\D^{n-1})$ appear. In this case, we will use the properties of a dg-Segal category to do so. We define $G^0$ to be a graph of the following form:
	\begin{center}
		\begin{tikzcd}
		G^0=j([n-1])\coprod *=0\ar[r,"k"]& 1\ar[r,"k"]&\ldots\ar[r,"k"]& n-1\ar[r,"0"]&n
		\end{tikzcd}
	\end{center}
	Then, we can construct $\D^n$ as the following pushout:
	\begin{center}
		\begin{tikzcd}
		\D(1,-1,1)\ar[r]\ar[d, "\alpha"]&\D^c(1,-1,1)\ar[d]\\
		L(G^0)\ar[r]& \D^n
		\end{tikzcd}
	\end{center}
	By using the third dg-Segal condition, we can write $F(\D^n)$ in the following way:
	$$F(\D^n)\simeq F(G^0)\times_{F(\D(1,-1,1))}F(\D^c(1,-1,1))\simeq F(\D^{n-1}\coprod k)\times_{F(\D(1,-1,1))}F(k\coprod k)). $$
	By the first dg-Segal condition, we can make those coproducts commute with $F$ in the following way:
	$$F(\D^n)\simeq F(\D^{n-1}\coprod k)\times_{F(\D(1,-1,1))}F(k\coprod k))\simeq (F(\D^{n-1})\times F(k))\times_{F(\D(1,-1,1))}(F(k)\times F(k)). $$
	Now, in particular, if we take $n=1$, we get the following formula: 
	$$F(\D^1)\simeq (F(\D^0)\times F(k))\times_{F(\D(1,-1,1))}F(k)^2=F(k)^2\times_{F(\D(1,-1,1))} F(k)^2.$$ 
	We are almost there. If we add and subtract one $F(k)$ to the formula of $F(\D^n)$ we will be done.
	$$F(\D^n)\simeq (F(\D^{n-1})\times F(k))\times_{F(\D(1,-1,1))}F(k)^2\simeq F(\D^{n-1})\times_{F(k)}(F(k)\times F(k))\times_{F(\D(1,-1,1))}F(k)^2 $$
	and we have that $F(\D^n)\simeq F(\D^{n-1})\times_{F(k)}F(\D^1)$. By the induction hypothesis, we have 
	$$F(\D^n)\simeq \overbrace{F(\D^1)\times_{F(\D^0)}\ldots\times_ {F(\D^0)}F(\D^1)}^{\text{$n-1$ times}}\times_{F(\D^0)} F(\D^1).$$
\end{itemize}
	The image $j^*(F)$ is a Segal space and we have finished our proof.
\\\end{proof}

\begin{rem}
    The reader is probably wondering why we introduced the notation $\D(n,0,1)$ only to immediately simplify it. The answer is that those dg-categories are actually a particular case of some dg-categories we will denote $\D(n,s,d)$, where $s\in\Z^n$ and $d\in\N^n$, which we think might form a full subcategory of $\Free$ that would be sufficient for this construction. We will expand on this hypothesis in Section \ref{linear simplex}  of Chapter \ref{Future work}.
\end{rem}

\begin{coro}The adjunction $j_!:\Fun(\Delta\op,\sS)\rightleftharpoons \Fun^\S(\Free\op_\S,\sS):j^*$ is a Quillen adjunction for the Segal and dg-Segal model structures, respectively.
\end{coro}

\begin{proof}
This result is a direct consequence of the last proposition. Indeed, by Proposition \ref{ch. 1: criteron Quillen adjunction}, we know that if we have an adjunction between two model categories $F:M\rightleftharpoons N:G$ and we want to prove it is a Quillen adjunction, then we only need to prove that $F$ preserves cofibrations and $G$ preserves fibrant objects. 
\\
\\ Now, the left Bousfield localization doesn't change cofibrations, and our adjunction was already a Quillen adjunction on the projective model structure; $j_!$ preserves cofibrations. We only have left to prove that $j^*$ preserves fibrant objects. But the fibrant objects in the dg-Segal model structure are the dg-Segal spaces that were fibrant in the projective model structure. By the last proposition, $j^*$ sends dg-Segal spaces to Segal spaces, and because it is already a Quillen adjunction in the projective model structure, it preserves fibrant objects in that structure. So it preserves fibrant objects in the dg-Segal model structure.
\\
\\The adjunction  is a Quillen adjunction on the Segal and dg-Segal model structures, respectively, and we have finished our proof.
\\\end{proof}

\section{Complete dg-Segal spaces}

When trying to characterize $\infty$-categories using Segal spaces, we realize quickly that the definition of Segal spaces that has been already given is not enough. Indeed, there is a class of morphisms, called Dwyer-Kan morphisms, that should be equivalences but aren't. In \cite{ComSegalSpacesREZK}, in his quest to invert those, the author gets to the following result: 

\begin{defin}We define $E(1)$ to be the discrete space given at level $n\in\N$ by
$$E(1)_n=\{x,y\}^{[n]}, $$
i.e. by two non-degenerate cells on each level. Those are given by $(xy)^{n/2}$ and $(yx)^{n/2}$ if $n$ is odd and $(xy)^{(n-1)/2}$ and $(yx)^{(n-1)/2}$ if $n$ is odd.
\end{defin}

\begin{prop}\cite[Prop. 7.6]{ComSegalSpacesREZK}, \cite[Prop. 4.5]{Nima} Let $f:V\to W$ be a map between two Segal spaces. We assume that the morphism
$\Map(E(1),X)\to \Map(*,X)$
is a weak equivalence, with both $X=V$and $W$, for a certain morphism $ *\to E(1)$. Then $f$ is a Dwyer-Kan equivalence if and only if it is a weak equivalence.
\end{prop}

But that just means that the Dwyer-Kan equivalences are equivalences if they are between two $C$-local objects in $\sS$, where $C$ is a class with one object, $*\to E(1)$. So they define a new concept, that of complete Segal spaces, which are just the objects we have just defined.

\begin{defin}\cite[Def. 4.1, Section 6]{ComSegalSpacesREZK} \cite[Prop. 4.5]{Nima} Let $W$ be a Reedy fibrant simplicial space. We say that $W$ is a \textbf{complete Segal space} if the maps 
$$W_k\to \overbrace{W_1\times_{W_0}\ldots\times_{W_0}W_1}^{\text{$k$ times}} \text{  and  }\Map(E(1),W)\to\Map(*,W)$$
are weak equivalences for all $k\geq 1$. Or, equivalently, 
$$W_k\to \overbrace{W_1\times_{W_0}\ldots\times_{W_0}W_1}^{\text{$k$ times}} \text{  and  } W_0\to W_{hoequiv}$$
are weak equivalences for all $k\geq 1$ and $W_{hoequiv}$ the space of homotopy equivalences.
\end{defin}

And now Rezk uses the left Bousfield localization with $\{*\to E(1)\}$ as his $C$ and the complete Segal spaces as his $C$-local objects in order to have a model category that works for him.

\begin{teo}\cite{ComSegalSpacesREZK} There exists a simplicial closed model structure on the category of simplicial spaces, with the following properties:
\begin{enumerate}
	\item The cofibrations are precisely the monomorphisms.
	\item The fibrant objects are precisely the complete Segal spaces.
	\item The weak equivalences are precisely the Dwyer-Kan equivalences between complete Segal spaces.
\end{enumerate}
\end{teo}

In our case we will see that, mirroring the classic situation, the dg-Segal spaces we have defined are also not enough to completely characterize our dg-categories. Indeed, we can't prove that the functor $\Sing$ is fully faithful; we will be able to do so up to a certain morphism, that we will call a DK-equivalence. As such, we will need to do another left Bousfield localization in order to invert those, and then define its local objects to be our complete dg-Segal spaces.
\\
\\Instead of doing it in that order, though, we will start with the definition of the complete dg-Segal model structure. For that, we will use the linearisation functor that we defined on the last section. 

\begin{defin}Let $E(1)$ be as defined above. We define
$$E_k=\L(j_!(E(1))). $$
\end{defin}

\begin{rem}We warn our readers of the fact that $E_k$ is not the image of a dg-category through our functor $\Sing$; in fact, it isn't even a dg-Segal space!
\end{rem}

\begin{teo}There exists a class of morphisms $C$ and a simplicial closed model structure on $\Fun^\S(\Free_\S\op,\sS)$ such that
\begin{enumerate}
	\item The cofibrations are the same as in the projective model structure.
	\item The fibrant objects are the dg-Segal spaces which satisfy that $F(k)\to j^*(F)_{hoequiv}$ is a weak equivalence of simplicial spaces and are fibrant for the projective model structure.
	\item The weak equivalences are the $C$-local equivalences.
\end{enumerate}
We call such a model structure \textbf{the complete dg-Segal model structure.}
\end{teo}

\begin{proof}
The proof of this is straightforward, as we are just transporting the complete Segal model structure to our new setting via the linearisation functor. Indeed, we can construct this by using a left Bousfield localization on the dg-Segal model structure by the class $C=\{\Phi:E_k\to \Sing(k)\}=\{\Phi:j_!(E(1))\to j_!(*)\}$. We won't write down the details about whether this gives us an actual localization, as they are the exact same as in the construction of the dg-Segal model structure.
\\
\\The only thing left to do is to see that the $C$-local objects are really as we have defined them above. Let $F$ be a $C$-local object. We have that $F$ is a dg-Segal space that is fibrant for the projective structure, and also that the following morphism
$$\Map(\Sing(k),F)\to\Map(E_k,F) $$
is a weak equivalence. By the definition of an adjunction, we have that 
$$\Map(E_k,F)=\Map(j_!(E(1)),F)\simeq \Map(E(1),j^*(F))\text{ and }\Map(\Sing(k),F)=\Map(j_!(*),F)\simeq \Map(*,j^*(F)).$$
So the condition is equivalent to asking for $\Map(*,j^*(F))\to \Map(E(1),j^*(F))$ to be a weak equivalence. And that is exactly the definition of $j^*(F)$ being a complete Segal space. So if $F$ is a $C$-local object, then $F$ is a dg-Segal space such that $j^*(F)_0=F(k)\to j^*(F)_{hoequiv}$ is a weak equivalence. By Proposition \ref{ch. 1: fibrant and C-local}, a fibrant object for the complete dg-Segal model structure is such an $F$ which is also fibrant for the projective model structure.
\\
\\We have all the conditions we needed and our proof is complete.
\\\end{proof}

\begin{rem}It is important to remember that $j^*(F)_{hoequiv}$ is the subset of $j^*(F)_1$ whose 0-simplexes are homotopy equivalences. Considering we've already seen that $j^*(F)_1=F(\D(1,0,1))$, we can rewrite $j^*(F)_{hocolim}$ as the subset of $F(\D(1,0,1))$ whose 0- simplexes are homotopy equivalences, with no mentions of the delinearisation morphism.
\end{rem}

\begin{nota}Let $F$ be a dg-Segal space. We denote the subset of  $F(\D(1,0,1))$ whose 0- simplexes are homotopy equivalences by $F_{hoequiv}$.
\end{nota}

\begin{defin}Let $F$ be an object in $\Fun^\S(\Free_\S\op,\sS)$. We say that $F$ is \textbf{a complete dg-Segal space} if it is a dg-Segal space and the morphism
$$F(k)\to F_{hoequiv} $$
is a weak equivalence. 
\end{defin}

\begin{defin}We denote the full subcategory of $\dgS$ of complete dg-Segal spaces by $\dgSc$.
\end{defin}

As we did with the dg-Segal spaces, we will prove now that every element of $Sing$ is actually a complete dg-Segal space.

\begin{prop}Let $T\in\dg$ be a dg-category. Then $\Sing(T)$ is a complete dg-Segal space.
\end{prop}

\begin{proof}
This is a direct consequence of the proof of Corollary 8.7 in \cite{Toen}. Indeed, during that proof Toën proves that for all $T\in\dg$, the morphism $\Map(k,T)=\Sing(T)(k)\to \Map(\D(1,0,1),T)=\Sing(T)(\D(1,0,1))$ induces an injection on $\pi_0$ and a bijection on $\pi_i$ for all $i>0$, and that its image in the homotopy category are the morphisms of $\l T\r$ that are isomorphisms. That means that the morphism is fully faithful, and that its essential image is $\Sing(T)_{hoequiv}$.
\\
\\So for all $T$, the morphism $\Sing(T)(k)\to \Sing(T)_{hoequiv}$ is a weak equivalence, and $\Sing(T)$ is a complete dg-Segal space.
\\\end{proof}

Now, we have a nice definition of our $C$-local objects, but we still have the $C$-local equivalences defined in an abstract manner. and we said before that we were going to define something called the DK-equivalences that we needed to make into equivalences. Before we do that, though, we need one additional definition.

\begin{nota}We denote the full subcategory of cofibrant complexes of modules by $\Ch^c\subset \Ch$.
\end{nota}

\begin{defin}Let $F\in\dgS$ be a functor that satisfies the dg-Segal conditions, let $x,y\in\pi_0(F(k))$. We define $F_{(x,y)}\in\Fun(\Ch^{c,op},\sS)$ to be, for all $E\in\Ch^c$ cofibrant complexes of modules, the homotopy fiber of $F(E_{x,y})\to F(k)\times F(k)$ where $E_{x,y}$ is the free dg-category given by the graph with two objects, $x$ and $y$, and $E$ as the complex between $x$ and $y$,  
\begin{center}
	\begin{tikzcd}
		F_{(x,y)}(E)\ar[r]\ar[d]& *\ar[d]\\
		F(E_{x,y})\ar[r]& F(k)\times F(k).
	\end{tikzcd}
\end{center}
We call $F_{(x,y)}$ \textbf{the dg-mapping space of $F$ at $x,y$}.
\end{defin}

This definition is more intuitive than we may think. Indeed, these functors are always representable. 

\begin{prop}Let $F\in\dgS$ be a functor that satisfies the dg-Segal conditions, let $x,y\in\pi_0(F(k))$. There exists a unique complex of modules up to weak equivalence $F(x,y)\in\Ch$ such that $F_{(x,y)}(-)\simeq \Map(-,F(x,y))$.
\end{prop}

\begin{proof}
The uniqueness is just a consequence of the Yoneda lemma. We assume the existence of two complexes of modules, $F_1(x,y)$ and $F_2(x,y)$, such that $F_{(x,y)}(-)\simeq \Map(-,F_1(x,y))$ and also $F_{(x,y)}(-)\simeq \Map(-,F_2(x,y))$. But the Yoneda lemma tells us that a weak equivalence of representable presheafs must come from a weak equivalence on the representing objects. Which means that we have
$$F_1(x,y)\simeq F_2(x,y). $$
We have the uniqueness.
\\
\\Let us see the existence next. This is a direct consequence of Proposition 1.9 in \cite{TV-Chern}, in particular of a certain point of the proof of part 1. Indeed, we prove there that if $A$ is an $\infty$-category, then all functor $G\in\Fun(A\op,\sS)$ that commutes with homotopy colimits is representable. That result has, then, reduced our problem to proving that the functor $F_{(x,y)}$ sends all homotopy colimits to homotopy limits (we remind the reader that $F_{(x,y)}$ is a contravariant functor). In order to make the writing of this proof easier, we will ignore the fact that $F$ is contravariant and call this "commuting with homotopy limits".
\\
\\Now, proving that a functor commutes with homotopy limits can be done by proving that it commutes with all filtered homotopy limits, all homotopy pushouts, and all finite sums. As $F$ satisfies the dg-Segal conditions, it already commutes with all filtered limits along perfect objects, which means that $F_{(x,y)}$ commutes with all filtered limits; and by the third dg-Segal condition it already commutes with homotopy pushouts along generating cofibrations. As we can write every homotopy pushout as a filtered homotopy limit over a homotopy pushout along a generating cofibration (in this case, the projective model structure), that means that $F_{(x,y)}$ commutes with all homotopy pushouts. Finally, all finite sums are homotopy pushouts except for the null sum, which means that is all we have left to prove.
\\
\\Let $G$ be a graph with two objects and the null complex between them. By definition, the free dg-category associated to $G$ is the coproduct $L(G)=k\coprod k$. By the first dg-Segal condition, that means that $F(L(G))=F(k\coprod k)\simeq F(k)\times F(k)$. If we apply now the definition of $F_{(x,y)}$, we have that $F_{(x,y)}(G)$ is the homotopy fiber of $F(L(G))=F(k)\times F(k)\to F(k)\times F(k)$. As this is an isomorphism, we have that $F_{(x,y)}(G)=*$, the null object in $\sS$, and we are done. 
\\
\\We have proven, then, that $F_{(x,y)}$ sends all homotopy colimits to homotopy limits. We have that $F_{(x,y)}$ is representable, i.e. there exists a complex of modules $F(x,y)$ such that $F_{(x,y)}(-)\simeq \Map(-, F(x,y))$.
\\
\\And we have finished this proof.
\\\end{proof}

\begin{rem}By the Yoneda lemma again, it is obvious that if the functor $F$ is of the form $\Sing(T)$ with $T\in\dg$, then we have that for all $E\in\Ch^c$,
$$F_{(x,y)}(E)=\Map(E,T(x,y)). $$
\end{rem}

And one last definition. 

\begin{defin}Let $F\in\dgS$ be a dg-Segal space. We call the \textbf{homotopy category of $F$}, and we denote it by $\l F\r$, the category whose objects are the 0-simplexes of $F(k)$ and whose morphisms are given for all $x,y\in\pi_0(F(k))$ by 
$$\l F \r(x,y)=H^0(F(x,y)) $$
i.e. the cohomology groups of degree 0 of the complex of morphisms associated to the dg-mapping space at $x$ and $y$.
\end{defin}

\begin{rem}We point to the reader that we have that if $F\in\dgS$, then
$$H^0(F(x,y))=\pi_0(\map_{j^*(F)}(x,y)), $$
where $\map_{j^*(F)}$ is the mapping space in the associated Segal space, so we can use the composition law there.
\end{rem}

We are now ready to define DK-equivalences.

\begin{defin}Let $f:F\to G\in\dgS$  be a morphism between two functors satisfying the dg-Segal conditions. We say that $f$ is \textbf{a DK-equivalence} if it satisfies the following conditions:
\begin{enumerate}
	\item The induced morphism $\l f\r:\l F\r\to \l G\r$ is essentially surjective.
	\item For all objects $x,y\in\pi_0(F(k))$, the induced morphism on the dg-mapping spaces, $F_{(x,y)}\to G_{(f(x),f(y))}$, is a quasi-equivalence of functors in $\Fun(\Ch^{c,op},\sS)$.
\end{enumerate}
\end{defin}

\begin{rem}\label{ch. 2: pseudo 2oo3}Some readers might ask themselves why we have defined DK-equivalences exclusively on dg-Segal spaces, and not on any functor. The reason is that if we do that, DK-equivalences would not satisfy the two-out-of-three condition. Indeed, if we have three morphisms $f$, $g$ and $f\circ g$, and two out of the three are DK-equivalences over any functor, then the third has the second condition of DK-equivalences because of the two-out-of-three condition on quasi-equivalences; but the first condition is not always true. It is only true if we have $f,g$ DK-equivalences or $g$ and $f\circ g$ DK-equivalences. It is true if the spaces are all dg-Segal, though.
\end{rem}

\begin{rem}If $F, G$ is of the form $F=\Sing(T), G=\Sing(T')$, what we just defined is pretty much exactly a weak equivalence between dg-categories.
\end{rem}

Now, the rest of this result hasn't been finished yet. We will discuss it in more detail in the next chapter, but for now, here is it.

\begin{hyp}\label{ch. 2: Re-zk}Let $f:F\to G$ be a morphism between two functors satisfying the dg-Segal conditions. Then $f$ is a DK-equivalence if and only if it is a weak equivalence in the complete dg-Segal model structure.  
\end{hyp}

We have now everything we need in order to prove that the functor is fully faithful.

\section{Hypercovers and free dg-categories}\label{hypercovers}

The next step of our proof is proving that $\Sing$ is, in fact, fully faithful. To do that, we will prove that the restriction to $\FreeS$ injects fully-faithfully, and then go on to prove that we can write every dg-category as a particular colimit of objects in $\FreeS$ and that $\Sing_W$ commutes with those colimits. And in order to do that, we will use hypercovers.
\\
\\ Hypercovers are a useful concept, and it has been defined in several different contexts. Some traits remain, though: in all of them, a hypercover of an object $A$ is some kind of augmented object $U_*\to A$ with a similar set of conditions, and in most of them we have the property that $\hocolim U_n\simeq A$. For example, Dugger and Isaksen proved in \cite{DI-TOPhyp} that we do have that weak equivalence $\hocolim U_n\simeq A$ result in $\Top$. In our case, the classic notion of hypercover that we will use more often is the one in the category of simplicial sets, $\sS$. For that, we will give here its definition. Although any reader familiar with Lurie's work will know that the definition there is given in much greater generality, we have decided to translate it here to the language of simplicial objects because it is the only context in which we will use it.

\begin{defin}\label{ch. 2: defin hypercovers sSet}\cite[Definition 6.5.3.2, Corollary 7.2.1.15]{HTT} Let $X\in\sS$ be a simplicial set, and $U_*\to X$ an augmented simplicial object in $\sS$. We say that $U_*$ is  \textbf{a hypercover of $X$} if for all $n\in\N$ the functor 
$$U_n\to U_*^{\partial\Delta^n}$$
is an effective epimorphism. In other words, $U_*\to X$ is a hypercover if for all $n\in\N$
$$\pi_0(U_n)\to \pi_0(U_*^{\partial\Delta^n})$$
is an epimorphism.
\end{defin}

\begin{prop}\label{Ch.2: hypercovers sSet} \cite[Theorem 6.5.3.12]{HTT} Let $X\in\sS$ be a simplicial set and $U_*\to X$ a hypercover of simplicial sets. Then we have that $\hocolim U_n\simeq X$.
\end{prop}

Unluckily for us, none of the known definitions of hypercover work for the context we want, so we'll be forced to write our own.

\begin{defin}Let $M$ be a model category and $M_0$ be a subcategory of $M$. Then we define $f:T\to T'$ a morphism in $M$ to be \textbf{an $M_0$-epimorphism} if for all $X\in M_0$ the induced functor 
$$\Map(X,T)\to \Map(X,T') $$
is an effective epimorphism in $\sS$, i.e. the morphism
$$\pi_0(\Map(X,T))\to \pi_0(\Map(X,T')) $$
is surjective.
\end{defin}

\begin{defin}Let $M$ be a model category and $M_0$ be a subcategory of $M$. Let $X\in M$ be an object of $M$, and $U_*\to X$ an augmented simplicial object in $M$. We say that $U_*$ is \textbf{an $M_0$-hypercover of $X$} if for all $n\in\N$ the functor 
$$U_n\to U_*^{\partial\Delta^n}$$
is an $M_0$-epimorphism. In other words, using Proposition \ref{ch.1: computations_exp}, $U_*\to X$ is an $M_0$-hypercover if 
$$U_0\to X $$
is an $M_0$-epimorphism and for all $n\geq 1$
$$U_n\to (\R\cosk_{n-1}\sk_{n-1}U)_{n} $$
is an $M_0$-epimorphism.
\end{defin}

Now that we have a definition of an $M_0$-hypercover, the next step will be proving that for all elements of a category we can construct a hypercover made entirely of objects in $M_0$. Of course, this cannot work for a general $M_0$: we will need a few extra conditions. As usual in these cases, we will construct our hypercover level by level, so we need a definition of a $n$-truncated hypercover.

\begin{defin}Let $M$ be a model category and $M_0$ be a subcategory of $M$. Let $X\in M$ be an object in $M$. We define \textbf{an $n$-truncated $M_0$-hypercover of $X$} to be an augmented $n$-truncated simplicial set $X_*\to X$ where for all $i\leq n$
$$X_i\to X_*^{\partial\Delta^i}  $$
is an $M_0$-epimorphism.
\end{defin}

And now for the theorem:

\begin{teo}\label{Ch. 2: existence hypercovers}Let $M$ be a model category where every object is fibrant, and let $M_0$ be a subcategory of $M$ which is closed for finite coproducts. We assume that for every $X\in M$ there exists an object in $M_0$, $U\in M_0$, and a morphism $U\to X$ that is an $M_0$-epimorphism. Then there exists an $M_0$-hypercover of $X$ consisting of objects in $M_0$.
\end{teo}

\begin{proof}
We are going to prove this by induction. Let $X\in M$ be an object in $M$. We will prove that for all $n\in \N$ there exists an $n$-truncated hypercover of $X$.
\\
\\$\bullet\ \ n=0$. This is true by hypothesis. We have assumed that there exists an object $U_0\in M_0$ with a morphism $U_0\to X$ which is $M_0$-epi. We then take the $0$-truncated simplicial set which is $U_0$ on degree $0$. It is a $0$-truncated hypercover by definition.
\\
\\ $\bullet\ \ n\in\N$. By induction hypothesis there exists an $n$-truncated hypercover of $X$, named $V_*\to X$. We have to construct an $(n+1)$-truncated hypercover of $X$, that we will call $U_*\to X$. 
\\
\\We define $V_*'=\sk_{n+1}(\cosk_n V_*)$. This simplicial object is $(n+1)$-truncated, but the term $n+1$ is not necessarily in $M_0$. By hypothesis, there exists a morphism $U'_{n+1}\to V_{n+1}'$ that is a $M_0$-epi and such that $U'_ {n+1}\in M_0$. We define the $(n+1)$-truncated $M_0$-hypercover $U_*\to X$ to be $U_i=V_i'=V_i$ for all $i\leq n$ and 
$$U_{n+1}=U'_{n+1}\coprod_{i\leq n} V_i$$
for $n+1$, where we take the colimit over $i\leq n$. As the set of elements $i\leq n$ is finite and we have assumed that $M_0$ is closed for finite colimits, we still have that $U_{n+1}\in M_0$.
\\
\\We need to prove that $U_*\to X$ is an augmented simplicial set, and also that we have the $M_0$-hypercover property. Let us construct the morphisms $U_m\to U_{n+1}$ and $U_{n+1}\to U_m$ for all morphism $[n+1]\to [m]$ and $[m]\to [n+1]$, with $m\leq n+1$.
\\
\\These morphisms are straightforward: in one direction we just take the composite
$$U_{n+1}\to U'_{n+1}\to V'_{n+1}\to V'_{m}=U_m. $$
and in the other sense we just take the morphism given by the colimit.
\\
\\And lastly, we need to prove that for all $i\leq n+1$ the morphism $U_i\to U_*^{\partial\Delta^i} $ is a $M_0$-epimorphism.
\\
\\For $i\leq n$ this is true by induction hypothesis. Indeed, we have defined $U_*$ in such a way that $U_i=V'_i$, and by Remark \ref{ch.1: cosk igualdad} we have that $V'_i=(\sk_{n+1}(\cosk_n V_*))_i=V_i$ for all $i\leq n$. We need then to prove that 
$$U_i=V'_i=V_i\to (\cosk_{i-1}\sk_{i-1}U_*)_i=(\cosk_{i-1}\sk_{i-1}V_*)_i $$
 is an $M_0$-epimorphism. Since $V_*$ is an $n$-truncated $M_0$-hypercover, this condition is verified. We only need to prove this for $n+1$.
\\
\\ We need to prove, then, that the morphism
$$U_ {n+1}=U'_{n+1}\coprod V_i\to (U_*)^{\partial\Delta^n}=(\cosk_n\sk_n U_*)_{n+1}=(\cosk_n V_*)_{n+1}=V'_{n+1} $$
is an $M_0$-epimorphism. Let $W\in M_0$ be an object in $M_0$, we are going to prove that 
$$\Map(W,U'_{n+1}\coprod V_i)\to \Map(W,V'_{n+1})$$
is surjective on the $\pi_0$. Every object in $M$ is fibrant by hypothesis, which means that we can write $\Map(W,U'_{n+1}\coprod V_i)$ as $\Hom(C^*(W),U'_{n+1}\coprod V_i)$, and similarly for $\Map(W,V'_{n+1})$. We are, then, going to prove that
$$\pi_0(\Map(W,U'_{n+1}\coprod V_i))=\Hom(W,U'_{n+1}\coprod V_i)/\sim\to \pi_0(\Map(W,V'_{n+1}))=\Hom(W, V'_{n+1})/\sim $$
is surjective, where $\sim$ is the homotopy equivalence relation.
\\
\\We will need an auxiliary morphism for this. By definition of a simplicial set, for all $i\leq n$ there exists a morphism $V_i=V'_i\to V'_{n+1}$. Also, by the way we have constructed $U_*$, there exists a morphism $u:U'_{n+1}\to V'_{n+1}$, which is an $M_0$-epimorphism. In consequence, by definition of a coproduct, there exists a morphism $f:U'_{n+1}\coprod V_i\to V'_{n+1}$, and a factorization $g:U'_{n+1}\to U'_{n+1}\coprod V_i$ such that $f\circ g=u$.
\\
\\And now let us prove the surjectivity. Let $F:W\to V'_{n+1}$ be a morphism. As $u:U'_{n+1}\to V'_{n+1}$ is an $M_0$-epimorphism, that means that it exists, up to homotopy, a morphism $F':W\to U'_{n+1}$ such that $u\circ F'=F$. But we have said that $u$ factorizes through $f$. We can then, compute a morphism 
$$g\circ F':W\to U'_{n+1}\to U'_{n+1}\coprod V_i $$
such that $f\circ g\circ F'=u\circ F'=F$ up to homotopy. We have that the morphism
$$\pi_0(\Map(W,U'_{n+1}\coprod V_i))=\Hom(W,U'_{n+1}\coprod V_i)/\sim\to \pi_0(\Map(W,V'_{n+1}))=\Hom(W, V'_{n+1})/\sim $$
is an $M_0$-epimorphism. The $(n+1)$-truncated simplicial set $U_*\to X$ is an $(n+1)$-truncated $M_0$-hypercover of $X$ and we have finished the proof.
\\\end{proof}

As the reader can imagine, we decided to use the free dg-categories of finite type expecting them to be well-behaved enough that we can work with $\FreeS$-hypercovers. But even though it would be possible to construct $\FreeS$-hypercovers directly, the number of objects in it would explode quite quickly, and we don't want that. So instead of working on the free dg-categories of finite type directly, we will first fix the objects. Once that is done, it will be the moment to make good on our word: let us prove that that is actually true.

\begin{prop}\label{ch. 2: free hyper} Let $X$ be a dg-category, and let $\O$ be its set of objects. Let $M$ be $\dg_\O$ the model category of dg-categories with $\O$ as a set of objects,  and $M_0=\Free_{\S\O}$ be the full subcategory of free dg-categories of finite type with $\O$ as a set of objects. There exists an $M_0$-hypercover $U_*\to X$ such that $U_i\in M_0$ for all $i\in\N$.
\end{prop}

\begin{proof}
We just need to prove that the model category $M$ and the subcategory $M_0$ fulfill the conditions of Theorem \ref{Ch. 2: existence hypercovers}, i.e. that all object in $M$ is fibrant, that $M_0$ is closed for finite coproducts and that for all object $X\in M$ there exists an object $U\in M_0$ and a $M_0$-epimorphism $U\to X$. 

\begin{itemize}
	\item The first condition is Corollary \ref{ch. 1:hom-cofib}.
	\item Let $X,Y$ be two free dg-categories over $\O$. Then, by definition, there exists two graphs $X',Y'\in Gr(\Ch)^{tf}$ such that $L(X')=X$ and $L(Y')=Y$. But $L$ is a left adjoint, and a finite coproduct is a special case of a colimit: as such, we know that $L$ commutes with finite coproducts. That gives us that $X\coprod Y=L(X')\coprod L(Y')=L(X'\coprod Y')$, and we have that $X\coprod Y$ is a free category of finite type, coming from the coproduct of graphs $X'\coprod Y'$. We have that $\FreeS$ is closed for finite coproducts and we have finished.
	\item Let $X$ be a dg-category. We are going to prove that the morphism $LU(X)\to X$ is the $M_0$-epimorphism we need. First of all, by definition of a free dg-category, $LU(X)$ has the same objects as $X$ and it is an object in $\Free_{\S\O}$. Let $A$ be a free dg-category of finite type over $\O$. We need to prove, then, that 
	$$\Map(A,LU(X))\to \Map(A,X) $$
	is surjective on the $\pi_0$. But this morphism is already surjective. Indeed, if we take $f:C^*(A)\to X$, we can define a morphism $f':C^*(A)\to LU(X)$ that gives us the right result. Indeed, $X$ and $LU(X)$ have the same objects, so that won't change. We only need to define the morphisms of complexes.
	\\
	\\Let $x,y\in C^*(A)$. Then we have a morphism $\phi:C^*(A)(x,y)\to X(f(x),f(y))$. Now, we remind the reader that, by definition \ref{ch. 1: def free}, the morphisms in $L(U(X))$ are defined as follows:
$$L(U(X))(f(x),f(y))=\bigoplus_{m\in\N}\bigoplus_{x_1,\ldots, x_m\in\mathcal{O}}(U(X)(f(x),f(x_1))\otimes \ldots \otimes U(X)(f(x_m),f(y))). $$	
	 We define the morphism $\phi':C^*(A)(x,y)\to LU(X)(f(x),f(y)) $ such that for all $g\in C^*(A)(x,y)$, the image by $\phi'$ is the direct sum with $\phi(g)$ in the zero component and zero everywhere else.
\end{itemize} 
We have all three conditions for the existence of an $M_0$-hypercover. For all $X$ there exists a $\FreeS$-hypercover of $X$ composed of objects in $\FreeS$.
\end{proof}

\begin{rem}There is one important thing that we need to highlight from this construction. If we have a $\FreeS$-hypercover $U_* \to X$ defined like that then for all $x,y\in \Ob(X)$ and for all $n\in\N$ we have that $U_n(x,y)\to U_*(x,y)^{\partial\Delta^n}$ is a split epimorphism.
\end{rem}

This second remark gives us a concept we will need for the next result.

\begin{defin} We define a \textbf{split hypercover} of $E$ in $\Ch$ to be an augmented simplicial complex $E_*\to E$ such that for all $n\in\N$ the morphism $E_n\to E_*^{\partial\Delta^n}$ is a split epimorphism.
\end{defin}

Once this definition is set, there is also another thing we have to keep in mind.

\begin{coro}\label{ch. 2 split hyper} If we have an augmented simplicial object $T_*\to T$ in $\dg_O$ such that for all $x,y\in\O$ the associated augmented simplicial complex $T_*(x,y)\to T(x,y)$ is a split hypercover in $\Ch$, then $T_*\to T$ is a hypercover in $\dg_\O$.
\end{coro}

We have now a definition of hypercover on dg-categories. When we described them at first, we gave one important property that those constructions tended to have: their homotopy colimits being quasi-equivalent to the original object. And we would like for it to happen in this context too. Let us work on that. In order to do so, though, we will need an auxiliary result beforehand.

\begin{lema}\label{ch.2: hyper complexes} Let $M=\dg$ and $M_0=\FreeS$. Let $T$ be an object in $\dg$ and $T_*\to T$ a $\FreeS$-hypercover constructed using Theorem \ref{Ch. 2: existence hypercovers} and Proposition \ref{ch. 2: free hyper}. Then for all $x,y\in \Ob(T)$ we have that $\hocolim_{\Delta\op}(T_i(x,y))\simeq T(x,y)$ in $\Ch$.
\end{lema}

\begin{proof}

As we said in the last remark, if $T_*\to T$ is a $\FreeS$-hypercover constructed using the aforementioned theorem, then for all $x,y\in\O$, $T_*(x,y)\to T(x,y)$ is a split epimorphism. We will then prove that if $E_*\to E$ is a split hypercover of complexes, then $\hocolim_{\Delta\op}E_i \simeq E$.
\\
\\We start with the connective case. Let us assume that $E_*, E\in\Ch^{\leq 0}$, and $E_*\to E$ is a split hypercover in $\Ch^{\leq 0}$. We can use the Dold-Kan equivalence, $DK:\sS\rightleftharpoons \Ch^{\leq 0}:DK^{-1}$, and $DK^{-1}(E_*)\to DK^{-1}(E)$ is surjective over $\pi_0$ in $\sS$, i.e. it is a hypercover of simplicial sets. By Proposition \ref{Ch.2: hypercovers sSet} we have that $\hocolim_{\Delta\op}DK^{-1}(E_i)\simeq DK^{-1}(E)$ in $\sS$. But the Dold-Kan equivalence isn't just an equivalence: it is a model equivalence, meaning that it is an equivalence of categories that also preserves the model structure (\cite[Chapter III, Section 2]{JG}). In particular, if we apply $DK$ we have that $\hocolim_{\Delta\op}E_i\simeq E$ and we have finished.
\\
\\We have then that the connective case is true. We will now reduce the general case to the connective case, using the naïve truncation $\tau:\Ch\to \Ch^{\leq 0}$ such that if $E\in\Ch$, $H^i(\tau(E))=H^i(E)$ for all $i\leq 0$ and  $H^i(\tau(E))=0$ for all $i>0$. But knowing that $E_*\to E$ is a split hypercover in $\Ch$ doesn't assure us that $\tau(E_*)\to \tau(E)$ will also be a split hypercover in $\Ch^{\leq 0}$. To prove that, we need the following result.

\begin{slema}\label{ch.2: h puissance} Let $E_*\to E$ be a split hypercover in $\Ch$ and $K$ a finite simplicial object. Then  we have the equivalence
$$H^i(E_*^K)\simeq H^i(E_*)^K. $$ 
\end{slema}
\begin{proof}
We are going to prove this by induction over the dimension of $K$. Let us prove that for all $n\in\N$, if $\dim K=n$, then $H^i(E_*^K)\simeq H^i(E_*)^K$.
\begin{itemize}
	\item $n=0$. This is almost immediate. Indeed, if $\dim K=0$, then $K=\coprod_p *$, and $E_*^K=E_0^p$. And as cohomology commutes with products, we have that $H^i(E_*^K)\simeq H^i(E_*)^K$. 
	\item $n\geq 1$. We assume that for every $K'$ with $\dim K'< n$ we have $H^i(E_*^{K'})\simeq H^i(E_*)^{K'}$. Let $K$ be a simplicial objet of dimension $n$. Then we have the following homotopy coproduct
\begin{center}
	\begin{tikzcd}
		\coprod\partial\Delta^n\ar[r]\ar[d]\arrow[dr,phantom, "\lrcorner^h", very near end]&\coprod \Delta^n\ar[d]\\
		K_{\leq n-1}\ar[r]&K
	\end{tikzcd}
\end{center}
where $K_{\leq n-1}$ is the simplicial subobject of $K$ of dimension $n-1$. In turn, that square gives us the following homotopy products 
\begin{center}
	\begin{tikzcd}
		E_*^K\ar[r]\ar[d]\arrow[dr,phantom, "\ulcorner^h", very near start]&E_*^{K_{\leq n-1}}\ar[d]\\
		\prod E_*^{\Delta^n}\ar[r]&\prod E_*^{\partial\Delta^n}
	\end{tikzcd}
%\end{center}
%\begin{center}
	\begin{tikzcd}
		H^i(E_*)^K\ar[r]\ar[d]\arrow[dr,phantom, "\ulcorner^h", very near start]&H^i(E_*)^{K_{\leq n-1}}\ar[d]\\
		\prod H^i(E_*)^{\Delta^n}\ar[r]&\prod H^i(E_*)^{\partial\Delta^n}.
	\end{tikzcd}
\end{center}
Now, we would like to have that if we take the $H^i$ on the first square we still have a homotopic product. That is not true in general. But we know something extra about this square: indeed, as $E_*\to E$ is a split hypercover, we know that $\prod E_*^{\Delta^n}\to\prod E_*^{\partial\Delta^n}$ is a split epimorphism, and in that case we do have that the square 
\begin{center}
	\begin{tikzcd}
		H^i(E_*^K)\ar[r]\ar[d]\arrow[dr,phantom, "\ulcorner^h", very near start]&H^i(E_*^{K_{\leq n-1}})\ar[d]\\
		\prod H^i(E_*^{\Delta^n})\ar[r]&\prod H^i(E_*^{\partial\Delta^n})
	\end{tikzcd}
\end{center}
is a homotopic product. By induction, we have $H^i(E_*)^{K_{\leq n-1}}\simeq H^i(E_*^{K_{\leq n-1})}$ and $\prod H^i(E_*)^{\partial\Delta^n}\simeq \prod H^i(E_*^{\partial\Delta^n})$ (because $\dim K_{\leq n-1}=\dim \partial\Delta^n=n-1$); and by definition we have that $H^i(E_*)^{\Delta^n}\simeq H^i(E_n)\simeq H^i(E_*^{\Delta^n})$. We have then the following cubic diagram, where both the front and the back square are homotopic products and three out of four front-to-back arrows are equivalences.
\begin{center}
	\begin{tikzcd}[row sep=scriptsize,column sep=scriptsize]
		& H^i(E_*^K)\arrow[dl]\arrow[rr]\arrow[dd] & & H^i(E_*^{K_{\leq n-1}})\arrow[dl, "\sim"]\arrow[dd] 
		\\H^i(E_*)^K\arrow[rr,crossing over]\arrow[dd] & & H^i(E_*)^{K_{\leq n-1}} \\
		& \prod H^i(E_*^{\Delta^n})\arrow[dl, "\sim"]\arrow[rr] &  & \prod H^i(E_*^{\partial\Delta^n})\arrow[dl, "\sim"] \\
		\prod H^i(E_*)^{\Delta^n}\arrow[rr] & & \prod H^i(E_*)^{\partial\Delta^n}\arrow[from=uu,crossing over]\\
	\end{tikzcd}
\end{center}
We have that the fourth arrow is also an equivalence, $H^i(E_*^K)\simeq H^i(E_*)^K$, and we have finished. 
\end{itemize}
\end{proof}
Now that we have that,  we can prove that if $E_*\to E$ is a split hypercover, then $\tau(E_*)\to \tau(E)$ is too. Indeed,  by Sublemma \ref{ch.2: h puissance}, we have that for all $i\leq 0$,
$$H^i(\tau(E_*)^{\partial\Delta^n})\simeq H^i(\tau(E_*))^{\partial\Delta^n}\simeq H^i(E_*)^{\partial\Delta^n}\simeq H^i(\tau(E_*^{\partial\Delta^n})) $$
and $\tau(E_*)^{\partial\Delta^n}\simeq \tau(E_*^{\partial\Delta^n})$. We then have that for all $n\in\N,$ $\tau(E)_n\to \tau(E_*)^{\partial\Delta^n}\simeq \tau (E_*^{\partial\Delta^n})$ is a split epimorphism, and $\tau(E_*)\to \tau(E)$ is a split epimorphism.
\\
\\Now, if $E_*\to E$ is a split hypercover, we have that for all $i\in\N$, $E_*[-i]\to E[-i]$ is also a split hypercover and $\tau(E_*[-i])\to \tau(E[-i])$ is too. By the connective case, we have that $\hocolim \tau(E_n[-i])\simeq \tau(E[-i])$. As this is true for all $i\in\N$, we have that $\hocolim E_n\simeq E$.
\\
\\In conclusion, if $T_*\to T$ is a hypercover constructed using Theorem \ref{Ch. 2: existence hypercovers}, we have that $\hocolim(T_i(x,y))\simeq T(x,y)$ for all $x,y\in\Ob(T)$ in $\Ch$ and we have finished. 
\\ \end{proof}

We're almost ready to prove that if $T_*\to T$ is a $\FreeS$-hypercover constructed as instructed, then $\hocolim T_i\to T$ is a weak equivalence. We have proven that the complexes of morphisms have that condition. But does it transfer well from complexes to dg-categories and vice-versa?

\begin{lema}\label{ch. 2: forget commutes} The forgetful functor $U:\dg_\mathcal{O}\to Gr(\Ch)_\mathcal{O}$ commutes with homotopy colimits over $\Delta\op$.
\end{lema}

\begin{proof}
This is a direct consequence of \cite[Lemma 4.1.8.13.]{LurieHA}, in particular of its proof. Indeed, in that result we say that if we have a combinatorial monoidal model category $A$ and a small category $C$ such that its nerve $N(C)$ is sifted, and we have on one hand that $A$'s monoidal structure is symmetric and satisfies the monoid axiom, and on the other hand that $A$ is left proper and its cofibrations are generated by cofibrations between cofibrant objects; then, the forgetful functor $\Alg(A)\to A$ commutes with homotopy colimits over $C$.
\\
\\Now, we take the monoidal structure on the graphs with fixed objects to be as follows: let $G,G'\in Gr(\Ch)_\O$ and $x,y\in\O$, we define
$$(G\otimes G')(x,y)=\oplus_z G(x,z)\otimes G'(z,y). $$
The category of dg-categories with fixed objects is the category of algebras over $Gr(\Ch)_\O$ with this monoidal structure, $\Alg (Gr(\Ch)_\O)$. As the category of graphs with this monoidal structure satisfies the above conditions and $N(\Delta)$ is sifted, we have that the forgetful functor $\dg_\O\to Gr(\Ch)_\O$ commutes with homotopy colimits and we have finished our proof.
\\\end{proof}

And now we're ready to prove our result.

\begin{prop}\label{ch. 2: free colimitant} Let $T$ be a dg-category with fixed objects, and $T_*\to T$ a $\FreeS$-hypercover in $\dg_\O$, the category of dg-categories with fixed objects $\O=\Ob(T)$. Then we have that $\hocolim T_i\simeq T$ in $\dg$.
\end{prop}

\begin{proof}
Let $\phi:\hocolim T_*\to T$ be the morphism from the homotopy colimit to $T$. We need to prove that this morphism is a weak equivalence in $\dg_\O$. But as we have the same objects, the quasi-essential surjectivity is automatic. We only need to prove that $\phi$ is quasi-fully faithful. By definition, that means that for all $x,y\in\O$, we need to prove that $(\hocolim T_i)(x,y)\simeq T(x,y)$. But as we know from Lemma \ref{ch. 2: forget commutes} that the forgetful functor commutes with $\hocolim$, that is equivalent to asking that for all $x,y\in\O$, we have $\hocolim(T_i(x,y))\simeq T(x,y)$. By Lemma \ref{ch.2: hyper complexes}, that is true. 
\\
\\We have that $\phi:\hocolim T_i\simeq T$ in $\dg_\O$. If we can prove that the functor $\Phi:\dg_\O\to \dg$ commutes with homotopy colimits over $\Delta\op$, we have finished the proof. Let $X=\hocolim_{\Delta\op} T_i$ in $\dg_\O$, and let us call the forgetful morphisms $\Phi:\dg_\O\to \dg$ and $\Xi:\dg_\O\to \coprod_\O k/ \dg$, where for all $X\in \dg_\O$ $\Xi(X)$ gives us a morphism from $\coprod_\O k$ to $\Phi(X)$.
\\
\\ For all $T'\in\dg$ we then have a morphism of the form
$$\Map(\Phi(X), T')\to \Map(\coprod k,T')\simeq \prod\Map(k,T'). $$
By getting the fiber of this morphism and doing the same thing with $\hocolim T_i$, we get the following diagram:
\begin{center}
	\begin{tikzcd}
		\Map(X,T')\ar[r]\ar[d,"\sim"]&\Map(\Xi(X),T')\ar[r]\ar[d, "f"]&\Map(\Phi(X),T')\ar[r]\ar[d]&\prod\Map(k,T')\ar[d,"="]\\
		\holim\Map(T_i,T')\ar[r]&\holim\Map(\Xi(T_i),T')\ar[r]&\holim\Map(\Phi(T_i),T')\ar[r]&\prod\Map(k,T').
	\end{tikzcd}
\end{center}

If we can prove that $f$ is a weak equivalence, we have finished. For that, we are going to use the far-left square of this diagram. Indeed, if the morphisms $\Map(\Xi,T')$ and $\holim \Map(\Xi,T')$ are weak equivalences, by 2-out-of-3 then $f$ will be a weak equivalence too. If we can prove that $\Xi$ is fully faithful, we will have everything we need. 
\\
\\We have an adjunction $\Xi:\dg_\O\rightleftharpoons \coprod k/\dg:\Gamma$ where the right adjoint $\Gamma$ is such that for all $(K,\{x_\alpha\}_{\alpha\in\O}, x_\alpha\in\Ob(K))\in\coprod k/\dg$, the objects of $\Gamma(K)$ are $\O$ and for every $\alpha,\beta\in\O$, $\Gamma(K)(\alpha,\beta)=K(x_\alpha, x_\beta)$. It is easy to see that $\Gamma\Xi=\Id_{\dg_\O}$. We have then that $\Xi$ is fully faithful and then that $\Map(\Xi,T')$ is a weak equivalence.
\\
\\As such, we have that 
\begin{center}
	\begin{tikzcd}
		\Map(X,T')\ar[r, "\sim"]\ar[d,"\sim"]&\Map(\Xi(X),T')\ar[d, "f"]\\
		\holim\Map(T_i,T')\ar[r,"\sim"]&\holim\Map(\Xi(T_i),T')
	\end{tikzcd}
\end{center}
and $f$ is a weak equivalence. So $\Phi$ commutes with colimits, and since $\hocolim T_i\simeq T$ in $\dg_\O$, we have $\hocolim T_i\simeq T$ in $\dg$ and we have finished this proof.
\\\end{proof}

We can now try to prove the full faithfullness of our adjoint functor $\Sing$.

\section{The functor $Sing$ is fully faithful}

Let us prove that $\Sing$ is, in fact, fully faithful. To do that, we will prove that the restriction to $\FreeS$ injects fully-faithfully, and then go on to prove that we can write every dg-category as a certain colimit of elements in $\FreeS$ and that $\Sing$ commutes with those colimits. The first part is fairly obvious.

\begin{prop}\label{ch. 2: fully faithful for free}With the construction from Theorem \ref{Ch.2: construction Sing_w}, the functor
$$\Sing^{\Free}:\Ho(\Free)\to \Ho(\Fun^\S(\Free_\S\op,\sS)) $$
is fully faithful.
\end{prop}

\begin{proof}
For this, we use Theorems \ref{ch. 1: equivalence for functor categories} and \ref{ch. 1:fully-faith}, taking the model category to be $\Free_\S$ and $W$ to be the weak equivalences in the full subcategory of $\dg$ of free categories of finite type. Indeed, if we go down to the homotopy category, we can factorize $\Sing^{\Free}$ as
$$\Ho(\Free)\to \Ho(\sS^{\Free,W})\to \Ho(\Fun^\S(L_W\Free_\S\op, \sS))=\Ho(\Fun^\S(\Free_\S\op,\sS)). $$
But we know from Theorem \ref{ch. 1:fully-faith} that the first morphism here is fully faithful, and from Theorem \ref{ch. 1: equivalence for functor categories} that the second one is an equivalence. So the composition of the two is fully faithful, and we have finished our proof.
\\\end{proof}

We have finally everything needed in order to use our free hypercovers. Let us prove that we have the necessary DK-equivalences.

\begin{lema}\label{ch.2 hocolim dg-Segal}Let $T_*\to T$ be a $\FreeS$-hypercover of $T$ constructed as in Theorem \ref{Ch. 2: existence hypercovers} and Proposition \ref{ch. 2: free hyper}, with $T\in\dg$ a dg-category. Then the homotopy colimit of its image by $\Sing_k$, $\hocolim(\Sing_k(T_i))$, satisfies the dg-Segal conditions.
\end{lema}

\begin{proof}
In order to prove this result, we need to prove that the homotopy colimit fulfills the three conditions of the definition of dg-Segal space. The first two are easy, and hinge on the fact that the homotopy colimit in this case commutes with finite products and all images of $\Sing$ satisfy the dg-Segal conditions.
\\
\\1. Let $L,K$ be two free dg-categories of finite type. By the first dg-Segal condition on the image of $\Sing$, 
$$\hocolim(\Sing(T_i))(K\coprod L)=\hocolim(\Sing(T_i)(K\coprod L))\simeq \hocolim(\Sing(T_i)(K)\times\Sing(T_i)(L)).$$
And as the homotopy colimits over $\Delta\op$ commute with finite products, 
$$\hocolim(\Sing(T_i))(K\coprod L)\simeq \hocolim(\Sing(T_i)(K))\times\hocolim(\Sing(T_i)(L))$$
and we have finished.
\\
\\2. The second property is even easier. Indeed,
$$\hocolim(\Sing(T_i))(\emptyset)\simeq\hocolim(\Sing(T_i)(\emptyset))\simeq\hocolim(*)\simeq *. $$
We arrive now to the third condition. For an issue of generality and also in order to lighten our notation, we will do the computations in a slightly larger context, with the help of this sublemma.

\begin{slema}Let $F\in\Fun^\S(\Free_\S^{c,op},\sS)$ be a functor that satisfies conditions 1 and 2 of the dg-Segal conditions, and also an additional condition as follows:
\begin{enumerate}
 \setcounter{enumi}{3}
	\item For all $L, K\in\Free_\S^c$, $F(L\coprod_{k\coprod k} K)\to F(L)\times_{F(k)\times F(k)} F(K)$ is a weak equivalence. In other words, $F$ transforms coproducts into products.
\end{enumerate}
Then we have that $F$ satisfies condition 3 of the dg-Segal conditions. In other words, for all $G\in Gr(\Ch)$ a graph of finite type, $x,y\in \Ob(G)$ two objects in $G$, and $\alpha\in Z^n(G(x,y))$ a cycle in $G(x,y)$, the following diagram is a homotopy pullback
	\begin{center}
		\begin{tikzcd}
			F(L(G(<\alpha>)))\ar[r]\ar[d]\arrow[dr,phantom, "\ulcorner^h", very near start]& F(L(G))\ar[d]\\
			F(\D^c(1,n,1))\ar[r]& F(\D(1,n,1)).
		\end{tikzcd}
	\end{center}	
\end{slema}

\begin{proof}
In order to make this proof easier to follow, we start with the case where $G$ has two objects and move up from there. Assume $G$ is a free dg-category with two objects. There is a morphism $\gamma: F(G)\to \Map(\O_G,\O)$, where $\O_G=\Ob(G)$ is the set of objects of $G$ and $\O=\pi_0(F(k))$. We know that $\pi_0(\Map(\O_G,\O))$ is the set of morphisms from $\O_G$ to $\O$. Then, if we take $f:\O_G\to \O$, the fiber of $\gamma$ is $F_{(f(x),f(y))}(G)$.  By conflating the two points on $G$ with the two points on $\D^c(1,s,1)$ and $\D(1,s,1)$, we can reduce the problem to asking that for all $f:\O_G\to \O$ the following diagram is a homotopy pullback
	\begin{center}
		\begin{tikzcd}
			F_{(f(x), f(y))}(G(<\alpha>))\ar[r]\ar[d]& F_{(f(x),f(y))}(G)\ar[d]\\
			F_{(f(x),f(y))}(\D^c(1,n,1))\ar[r]& F_{(f(x),f(y))}(\D(1,n,1))
		\end{tikzcd}
	\end{center}
	 But we already know that for all dg-Segal space $F$, $F_{(f(x),f(y))}(E_{x,y})\simeq \Map(E,F(f(x),f(y)))$, where $E_{x,y}$ is the free category given by two points and $E$ as the morphism of complexes from $x$ to $y$. So the question ends up being whether
	  	\begin{center}
		\begin{tikzcd}
			\Map(G<\alpha>,F(f(x),f(y)))\ar[r]\ar[d]& \Map(G,F(f(x),f(y)))\ar[d]\\
			\Map(k^c[n],F(f(x),f(y)))\ar[r]& \Map(k[n],F(f(x),f(y)))
		\end{tikzcd}
	\end{center}
	is a homotopy pullback. But this diagram is in $\sS$, which is a proper model category, meaning that if one of the morphisms of this diagram is a fibration, we have finished. And indeed, the morphism $k[n]\to k^c[n]$ is a generating cofibration of $\Ch$, which means that the morphism $\Map(k^c[n],F(f(x),f(y)))\to \Map(k[n],F(f(x),f(y)))$ is a fibration. The diagram is a homotopy pullback and we have finished. If $G$ has two objects, $F$ satisfies the third dg-Segal condition on it.
	\\
	\\For the passage from two elements to more, it is a question of noticing that condition 3 is merely a local condition: intuitively, the only change being made in it is in relation to the morphisms between $x$ and $y$. Outside of that, it changes nothing whether the original free dg-category $G$ had two objects or three thousand. Following that logic, we will decompose $G$ in two sections: one that changes and one that does not.
	\\
	\\Let $G$ be a free dg-category of finite type, and let $x,y$ be two objects in $G$. It is easy to see that we can write it as $G=G^0\coprod_{k\coprod k}G_{x,y}$, where $G_{x,y}$ is a free dg-category with two objects and $G(x,y)$ as the morphism complex between those objects; and $G^0$ is a free dg-category such that $G^0$ has the same objects as $G$ and $G^0(x',y')=G(x',y')$ if $(x',y')\neq (x,y)$ and $G^0(x,y)=0$. As $G_{x,y}$ is a free dg-category with two objects, we have that $F$ satisfies the third dg-Segal condition on $G_{x,y}$. We can then construct a tower of homotopy pullbacks of the following form:
	\begin{center}
		\begin{tikzcd}
			F(G^0)\times_{F(k)\times F(k)}	F(G_{x,y}(<\alpha>))\ar[r]\ar[d]\arrow[dr,phantom, "\ulcorner^h", very near start]& F(G^0)\times_{F(k)\times F(k)}	F(G_{x,y})\ar[d]\\
			F(G_{x,y}(<\alpha>))\ar[r]\ar[d]\arrow[dr,phantom, "\ulcorner^h", very near start]& F(G_{x,y})\ar[d]\\
			F(\D^c(1,n,1))\ar[r]& F(\D(1,n,1))
		\end{tikzcd}
	\end{center}	
and the outside square is a homotopy pullback. By condition 4, we have that 
$$F(G^0)\times_{F(k)\times F(k)}	F(G_{x,y}(<\alpha>))\simeq F(G^0\coprod_{k\coprod k}G_{x,y}<\alpha>)\simeq F(G(<\alpha>)$$ 
and 
$$F(G^0)\times_{F(k)\times F(k)}	F(G_{x,y})\simeq F(G^0\coprod_{k\coprod k}G_{x,y})\simeq F(G),$$ 
 so we have our condition 3 for all $G$ free dg-categories of finite type and we have finished our proof. 
\\\end{proof}
We now want to apply this sublemma to our homotopy colimit. We have already proven that $\hocolim(\Sing(T_i))$ satisfies conditions 1 and 2 of the dg-Segal conditions; we just need to prove that it also satisfies condition 4 of the sublemma. Let $K, L$ be two free dg-categories of finite type. In a similar way than the proof of Proposition \ref{ch.2 Sing is dg-Segal}, we have that 
$$\Sing(K\coprod_{k\coprod k} L)\simeq \Sing(K)\times_{\Sing(k)^2}\Sing(L).$$
But as long as the base of the finite product has a finite $\pi_0$, we have that the homotopy colimit commutes with it. Which means that
$$\hocolim(\Sing(T_i))(K\coprod_{k\coprod k} L)\simeq \hocolim(\Sing(T_i)(K)\times_{\Sing(k)^2}\Sing(T_i)(L))\simeq$$
$$\simeq \hocolim(\Sing(T_i)(K))\times_{\hocolim(\Sing(k))^2}\hocolim(\Sing(T_i)(L)).$$
Hence the functor $\hocolim(\Sing(T_i))$ satisfies the condition 4 of the sublemma, and in consequence it satisfies the third condition of a dg-Segal space.
\\
\\The functor $\hocolim(\Sing(T_i))$ satisfies all three dg-Segal conditions, and we have finished our proof.
\\\end{proof}

\begin{rem}We have proven here that if we have the first two dg-Segal conditions and also that $F(L\coprod_{k\coprod k} K)\to F(L)\times_{F(k)\times F(k)} F(K)$, we have a dg-Segal space; and also we proved in Proposition \ref{ch. 2: dg-Segal are Segal} that if we have a dg-Segal space, then at the very least $F(\D^n)\simeq F(\D^1)\times_{F(k)}\ldots\times_{F(k)}F(\D^1)$. This is starting to resemble strongly the definition of a Segal space. Indeed, we think it is probable that we could rewrite the definition of dg-Segal spaces to reflect way more closely that definition. We have decided not to do so because we think this definition emphasizes better the linear structure. In any case, this is not a minimal construction anyway: see Chapter 3 for more details.
\end{rem}

Now that we know that $\hocolim(\Sing(T_i))$ is a dg-Segal space, we can see whether the induced morphism $\hocolim(\Sing(T_i))\to \Sing(T)$ is a DK-equivalence in $\dgS$.

\begin{lema}\label{ch.2 hocolim DK}Let $T_*\to T$ be a $\FreeS$-hypercover of $T$ constructed as in Theorem \ref{Ch. 2: existence hypercovers} and Proposition \ref{ch. 2: free hyper}. Then the morphism $\hocolim(\Sing(T_i))\to \Sing(T)$ is a DK-equivalence.
\end{lema}

\begin{proof}
In order to prove this, we need to make sure this morphism satisfies both conditions of the definition. The first is easy: indeed, we have made sure by the construction of our hypercover that all terms of the hypercover have the same elements. So $\l \hocolim(\Sing(T_i))\r\to \l \Sing(T)\r$ is essentially surjective.
\\
\\As for the condition on the morphisms, we have done most of the work already. We know that for all $(x,y)\in\pi_0(\Sing(T)(k))$, $\Sing(T)_{(x,y)}(E)=\Map(E,T(x,y))$. So we can rewrite this condition as being whether
$$\hocolim(\Map(-,T_i(x,y))\simeq \Map(-,T(x,y)). $$

But we know that $T_*\to T$ is a $\FreeS$-hypercover, which in particular means that for all $x,y\in T$, $T_n(x,y)\to T_*(x,y)^{\partial\Delta^n}$ is a split epimorphism. In consequence, we have that for all $E\in\Ch$,
$$\pi_0(\Map(E,T_n(x,y)))\to \pi_0(\Map(E,T_*(x,y)^{\partial\Delta^n}))\simeq \pi_0(\Map(E,T_*(x,y))^{\partial\Delta^n}) $$
is an epimorphism. But we have seen that this is exactly the definition of a hypercover in $\sS$. So we have that $\Map(E,T_*(x,y))$ is a hypercover of $\Map(E,T(x,y))$ in the simplicial sets, and by Proposition \ref{Ch.2: hypercovers sSet} that means that for all $E$ we have $\hocolim\Map(E,T_*(x,y))\simeq \Map(E,T(x,y))$. That is, by definition, the same as saying that 
$$\hocolim(\Map(-,T_i(x,y))\to \Map(-,T(x,y)) $$
is a weak equivalence.
\\
\\The second condition is fulfilled and the morphism $\hocolim(\Sing(T_i))\to \Sing(T)$ is a DK-equivalence. 
%But by Lemma \ref{ch.2: hyper complexes}, we have that
%$$ \hocolim(\Map(-,T_i(x,y))\simeq \Map(-, \hocolim(T_i(x,y)))\simeq \Map(-, T(x,y)).$$
%So the second condition is fulfilled and the morphism $\hocolim(\Sing(T_i))\to \Sing(T)$ is a DK-equivalence. 
\\\end{proof}

Finally, we have enough information to prove the full faithfullness of our functor $\Sing$.

\begin{teo}\label{ch. 2: fully faithfulness} Assuming Hypothesis \ref{ch. 2: Re-zk} to be true, for all $T\in\dg$ we have $\Re(\Sing)(T)\simeq T$, and the functor $\Sing$ is fully faithful.
\end{teo}

\begin{proof}
Let $T$ be a dg-category. We construct, with the aforementioned methods, a $\FreeS$-hypercover of $T$, $T_i\to T$. By Lemma \ref{ch.2 hocolim DK}, we know that the image by $\Sing$ of this morphism is a DK-equivalence. We then have the following diagram
	\begin{center}
		\begin{tikzcd}
			\Re(\hocolim(\Sing(T_i))\ar[rr,"f"]\ar[d,"g"]&& \Re(\Sing(T))\\
			\hocolim(\Re(\Sing(T_i)))\ar[d,"\phi"]&& \\
			\hocolim(T_i)\ar[uurr, "\zeta"]& &
		\end{tikzcd}
	\end{center}	
Assuming Hypothesis \ref{ch. 2: Re-zk} to be true, we know that the morphism $\hocolim(\Sing(T_i))\to \Sing(T)$ is a weak equivalence for the complete dg-Segal model structure. So $f$ is a weak equivalence on the homotopy categories. On the other hand, because $\Re$ is a left hand adjoint, we know we can commute with the homotopy colimit, so $g$ is also a weak equivalence. Lastly, we have proven in Proposition \ref{ch. 2: fully faithful for free} that $\Sing$ is fully faithful over free dg-categories, and as all $T_i$ are free dg-categories by construction, we have that $\Re(\Sing(T_i))\simeq T_i$ for all $T_i$ and the morphism $\phi$ is also a weak equivalence. By the two out of three condition on model categories, that means that the morphism $\zeta$ is also a weak equivalence.
\\
\\That gives us the following diagram:
	\begin{center}
		\begin{tikzcd}
			\hocolim(T_i)\ar[r,"\zeta"]\ar[d,"~"]&\Re(\Sing(T_i))\ar[dl]\\
			T
		\end{tikzcd}
	\end{center}	
The vertical morphism is a weak equivalence by Proposition \ref{ch. 2: free colimitant}, and $\zeta$ is a weak equivalence too. By the two out of three condition, we have that $\Re(\Sing(T))\to T$ is a weak equivalence. The functor $\Sing$ is fully faithful and we have finished out proof.
\\\end{proof}

We finally know the functor is fully faithful! We can now go back to proving it is essentially surjective.

\section{The functor $Sing$ is essentially surjective}

All that is left for us to do is to prove that every functor that satisfies the dg-Segal conditions is isomorphic to an object of the form $\Sing(T)$. In order to do that, we are going to use hypercovers again: for all functor $F$ that satisfies the dg-Segal conditions, we will construct a hypercover of a dg-category $T$ whose image by the $Sing$ functor is a hypercover of $F$. So first we need to define what a hypercover of $F$ is. 

\begin{defin}Let $F, G\in\dgS$. We say that a morphism $f:F\to G$ is a \textbf{dg-Segal epimorphism} if we have the following two conditions:
\begin{enumerate}
	\item The morphism $f$ is an isomorphism on the objects, $\pi_0(F(k))\to \pi_0(G(k))$.
	\item For all $x,y,\in\pi_0(F(k))$, the induced morphism $F_{(x,y)}\to G_{(x,y)}$ is a split epimorphism.
\end{enumerate}
\end{defin}

Now that we have our definition of what epimorphism we want, we can define our hypercover. It isn't complicated: we will essentially use the same definition we used last time, utilizing dg-Segal epimorphisms instead of $M_0$-epimorphisms.

\begin{defin}Let $F\in\dgS$ a functor that satisfies the dg-Segal conditions, and $F_*\to F$ an augmented simplicial object in $\dgS$. We say that $F_*$ is \textbf{a dg-Segal-hypercover of $F$} if for all $n\in\N$ the functor 
$$F_n\to F_*^{\partial\Delta^n}$$
is a dg-Segal epimorphism. In other words, using Proposition \ref{ch.1: computations_exp}, $F_*\to F$ is a dg-Segal-hypercover if 
$$F_0\to F $$
is a dg-Segal epimorphism and for all $n\geq 1$
$$F_n\to (\R\cosk_{n-1}\sk_{n-1}F)_{n} $$
is a dg-Segal epimorphism.
\end{defin}

Now that we have our definition of hypercover, our objective will be to construct for all $F$ a hypercover made of functors of the form $\Sing(T_i)$ with $T_i$ a free dg-category. Following exactly the proof of Theorem \ref{Ch. 2: existence hypercovers}, this would be the point where we would prove that for every $F\in\dgS$, there exists a free dg-category $T$ such that $\Sing(T)\to F$ is a dg-Segal epimorphism. Unluckily for us, while that result is undoubtedly true, utilizing it would later make the objects in our hypercover explode, and we do not want that. So we are going to fix the objects first.

\begin{nota}We denote the functor $\Sing(k)\in\dgS$ by $\underline{k}$. 
\end{nota}

\begin{defin}Let $\O$ be a set. We define \textbf{the category of dg-Segal spaces with fixed objects over $\O$} to be the full subcategory of $\coprod_\O\underline{k}/\dgS$ of $F$ dg-Segal spaces such that the morphism $\O\to \pi_0(F(k))$ is an isomorphism, and we denote it by $\dgS_\O$.
\end{defin}

\begin{defin}Let $\Phi:\dgS_\O\to \dgS$ be the forgetful functor. We define a dg-Segal epimorphism on $\dgS_\O$ to be a morphism $f$ in $\dgS_\O$ such that $\Phi(f)$ is a dg-Segal epimorphism.
\end{defin}

\begin{rem}It is easy to see that the first condition of the dg-Segal epimorphism is always true in $\dgS_\O$. Consequently, we won't have to check that condition as long as we are working on fixed objects.
\end{rem}

\begin{lema}\label{ch. 2 existence dg-epi} Let $F\in\dgS$ be a functor that satisfies the Segal conditions. We fix a set $\O=\pi_0(F(k))$. There exists a dg-category with fixed objects $T\in\dg_\O$ such that $\Sing(T)\to F$ is a dg-Segal epimorphism in $\dgS_\O$. 
\end{lema}

\begin{proof}

Most of the work for the construction of this dg-category has already been done, and we only have to put it together. We define a graph $G$ as follows:
\begin{itemize}
	\item $\Ob(G)=\O$.
	\item For all $x,y\in\O$, we have that $G(x,y)=F(x,y)$, the representing object of $F_{(x,y)}$.
\end{itemize}

We define then $T=L(G)$ as being the free category constructed from $G$.
\\
\\Let us now prove that the morphism $\Sing(T)\to F$ is a dg-Segal epimorphism in $\dgS_\O$, or equivalently, that its projection on $\dgS$ is a dg-Segal epimorphism.
\\
\\By the construction of $T=L(G)$, we have that for all $x,y\in\O$, 
$$L(G)(x,y)=\bigoplus_{m\in\N}\bigoplus_{(x_1,x_2, \ldots, x_m)\in\O^m}(G(x,x_1)\otimes\ldots \otimes G(x_m,y)), $$
and in particular, $G(x,y)=F(x,y)$ is a factor in $T(x,y)$. That means that for all $E\in\Ch$ there exists an inclusion $\Map(E,F(x,y))=F_{(x,y)}(E)\to\Map(E,T(x,y))$. It is easy to see, using the definition of homotopy fiber, that for all dg-category $T'$, $\Sing(T')_{(x,y)}=\Map(-,T'(x,y))$. Putting it all together, we have then that the morphism $\Sing(T)_{(x,y)}=\Map(-,T(x,y))\to F_{(x,y)}=\Map(-, F(x,y))$ is a split epimorphism and we have proven our result.
\\
\\The morphism $\Sing(T)\to F$ is a dg-Segal epimorphism and we have finished.
\\\end{proof}

Let us tackle now the hypercover result. 

\begin{teo}\label{ch. 2 dg-Segal hypercover}Let $F\in\dgS$ be a functor that satisfies the dg-Segal conditions and let $\O=\pi_0(F(k))$. Then there exists a simplicial object $T_*$ in $\dg_\O$ such that $F_*=\Sing(T_*)$ and a morphism $F_*\to F$ such that $F_*\to F$ is a dg-Segal hypercover in $\dgS_\O$.  
\end{teo}

\begin{proof}
Again, this construction is almost identical to that of Theorem \ref{Ch. 2: existence hypercovers}, and as such we won't be going into much detail. We will construct our hypercover by induction, by proving that for all $n\in\N$ there exists an $n$-truncated dg-Segal hypercover with fixed objects of $F$ where every level is in the image of $\Sing$. 
\begin{itemize}
	\item $n=0$ is true by Lemma \ref{ch. 2 existence dg-epi}. There exists a dg-category with fixed objects $T\in\dg_\O$ such that $\Sing(T)\to F$ is a dg-Segal epimorphism, and that creates a 0-truncated dg-Segal hypercover.
	\item $n\in\N$. By induction hypothesis there exists an $n$-truncated dg-Segal hypercover of $F$, named $\Sing(T_*)\to F$. Let us construct an $(n+1)$-truncated hypercover of $F$, that we will call $\Sing(T_*)\to F$ too. As the construction doesn't change the first $n$ terms, there is no ambiguity in the notation. 
\\
\\We define $V_*=\sk_{n+1}(\cosk_n \Sing(T_*))$. This simplicial set is $(n+1)$-truncated, but the term $n+1$ is not necessarily in the image of $\Sing$. By Lemma \ref{ch. 2 existence dg-epi} again, there exists a free dg-category $A$ with objects $\O$ and a morphism $\Sing(A)\to V_{n+1}$ that is a dg-Segal epimorphism. We define the $(n+1)$-truncated dg-Segal hypercover to be $\Sing(T_i)$ for all $i\leq n$ and 
$$\Sing(T_{n+1})=\Sing(A)\coprod \Sing(T_i)=\Sing(A\coprod T_i) $$
for $n+1$. This is possible because $\Sing$ commutes with finite coproducts.
\\
\\With the same morphisms as in Theorem \ref{Ch. 2: existence hypercovers}, we have a simplicial object in $\dgS_\O$. We now just have to prove that for all $i\leq n+1$, $\Sing(T_i)\to \Sing(T_*)^{\partial\Delta^i}$ is a dg-Segal epimorphism, which by construction gets instantly reduced to proving that 
$$ \Sing(T_{n+1})=\Sing(A)\coprod \Sing(T_i)\to V_{n+1}$$
is a dg-Segal epimorphism in $\dgS_\O$, or equivalently, that for all $x,y\in\O$, $\Sing(T_{n+1})_{(x,y)}=\Map(-, T_{n+1}(x,y))\to V_{n+1\\ (x,y)}$ is a split epimorphism.
\\
\\
\end{itemize}
We have created $F_*\to F$ a dg-Segal hypercover with fixed objects such that for all $n\in\N$, there exists a free dg-category of finite type $T_n$ such that $\Sing(T_n)=F_n$. Now, we have proven in Theorem \ref{ch. 2: fully faithfulness} that $\Sing$ is a fully faithful functor. In consequence, every morphism in the simplicial object $F_*$ comes from a morphism in $\dg_\O$, and there exists a simplicial object $T_*\in\dg_\O$ such that $\Sing(T_*)=F_*$. We have finished our proof.
\\\end{proof}

Now that we have constructed our hypercover $F_*\to F$, we are on the final stretch of the proof. Indeed, the last thing we need is to prove that for all $F\in\dgS$ there exists a dg-category $T$ such that $F$ is DK-equivalent to $\Sing(T)$, and we have our perfect candidate to do so. 

\begin{nota}Let $F\in\dgS$ be a functor satisfying the dg-Segal conditions. Let $\Sing(T_*)=F_*\to F$ be a dg-Segal hypercover with fixed objects as constructed in Theorem \ref{ch. 2 dg-Segal hypercover}. We call $T$ the homotopy colimit of the simplicial object $T_*$. In other words, we define
$$T=\hocolim(T_i). $$
\end{nota}

We need now to prove that $\Sing(T)\to F$ is a DK-equivalence. In order to do that, we will use two DK-equivalences that are easier to prove: $\hocolim(\Sing(T_i))\to \Sing(\hocolim(T_i))=\Sing(T)$ and $\hocolim(\Sing(T_i))\to F$.

\begin{lema}\label{ch. 2 dg-Segal to complex hyper}Let $F_*\to F$ be a dg-Segal hypercover in $\dgS_\O$ with $\O=\pi_0(F(k))$ constructed as before. Then, for all $x,y\in\O$, the augmented simplicial complex $T_*(x,y)\to F(x,y)$ is a split hypercover in $\Ch$.
\end{lema}

\begin{proof}
This follows from the definition of dg-Segal hypercover.
\\\end{proof}

\begin{prop}\label{ch.2: fixed objects dg-Segal}Let $F_*\to F$ be a dg-Segal hypercover constructed as in Theorem \ref{ch. 2 dg-Segal hypercover}. Then the homotopy colimit of $F_*$ is DK-equivalent to $F$ in $\dgS_\O$.
\end{prop}

\begin{proof}
The proof of this follows very closely the one in Lemma \ref{ch.2 hocolim DK}. Indeed, in order to prove something is a DK-equivalence, we have two conditions: firstly, that the morphism
$$\l\hocolim(F_i) \r\to \l F\r $$
is an essentially surjective. But the hypercover has been constructed to have fixed objects $\pi_0(F(k))$, so this condition is verified by construction.
\\
\\That leaves us with the second condition. We fix $x,y\in \O$. Do we have that 
$$(\hocolim F_{i})_{(x,y)}\to F_{(x,y)} $$
is a quasi-equivalence? By definition, $F_{i,(x,y)}=\Map(-,T_i(x,y))$ and $F_{(x,y)}=\Map(-,F(x,y))$, so we can rewrite this condition as wondering whether 
$$\hocolim \Map(-,T_i(x,y))\to \Map(-,F(x,y))$$
is a quasi-equivalence. By Lemma \ref{ch. 2 dg-Segal to complex hyper}, we know that $T_*(x,y)\to F(x,y)$ is a split hypercover of complexes, so for all $n\in\N$, $T_n(x,y)\to T_*^{\partial\Delta^n}(x,y)$ is a split epimorphism, and in consequence for all $E\in\Ch$ the morphism
$$\pi_0(\Map(E,T_n(x,y)))\to \pi_0(\Map(E,T_*(x,y)^{\partial\Delta^n}))\simeq \pi_0(\Map(E,T_*(x,y))^{\partial\Delta^n})  $$
is an epimorphism. That means that $\Map(E,T_*(x,y))\to \Map(E,F(x,y))$ is a hypercover of simplicial sets, and by Proposition \ref{Ch.2: hypercovers sSet}, we have a quasi-equivalence
$$\hocolim\Map(-,T_i(x,y))\simeq \Map(-,F(x,y)). $$
So the morphism $\hocolim(F_i)\to F$ is a DK-equivalence and we have finished our proof.
\\\end{proof}

Like in the case of the dg-categories, we have proved that the homotopy colimit is DK-equivalent to $F$, but only with fixed objects. But if we want to use it, we need to see that it is indeed true in $\dgS$.

\begin{prop}Let $F$ be a dg-Segal space and $F_*\to F$ a dg-Segal hypercover constructed as in Theorem \ref{ch. 2 dg-Segal hypercover}. Then the homotopy colimit of $F_*$ is DK-equivalent to $F$ in $\dgS$. 
\end{prop}

\begin{proof}
This proof is almost identical to the one in Proposition \ref{ch. 2: free colimitant}, and as such we won't spend too much time on its details. As was the case there, we have that the homotopy colimit of the hypercover is DK-equivalent to $F$, but only with fixed objects. So we're going to prove that the forgetful functor $\Phi:\dgS_\O\to \dgS$ commutes with those homotopy colimits. Let us denote the obvious forgetful functor by $\Xi: \dgS_\O\to \coprod \k/\dgS$. 
\\
\\ For all $G\in\dgS$ we then have a morphism of the form
$$\Map(\Phi(F), G)\to \Map(\coprod \k,G)\simeq \prod\Map(\k,G). $$
By getting the fiber of this morphism and doing the same with $\hocolim F_i$, we get the following diagram:
\begin{center}
	\begin{tikzcd}
		\Map(F,G)\ar[r]\ar[d,"\sim"]&\Map(\Xi(F),G)\ar[r]\ar[d, "f"]&\Map(\Phi(F),G)\ar[r]\ar[d]&\prod\Map(\k,G)\ar[d,"="]\\
		\holim\Map(F_i,G)\ar[r]&\holim\Map(\Xi(F_i),G)\ar[r]&\holim\Map(\Phi(F_i),G)\ar[r]&\prod\Map(\k,G)
	\end{tikzcd}
\end{center}

If we can prove that $f$ is a weak equivalence, we have finished. For that, let us prove that $\Xi$ is fully faithful. 
\\
\\We have an adjunction $\Xi:\Ho(\dgS_\O)\rightleftharpoons \Ho(\coprod \k/\dgS):\Gamma$ where the right adjoint $\Gamma$ is such that for all $H\in\coprod \k/\dgS$ and $L\in\FreeS$, $\Gamma(H)(L)$ is the following homotopy pullback
\begin{center}
	\begin{tikzcd}
		\Gamma(F)(L)\ar[d]\ar[r]\arrow[dr,phantom, "\ulcorner^h", very near start]&F(L)\ar[d]\\
		\coprod_{\O_L} F(k)\ar[r]&\coprod_{\O_L}\O 
	\end{tikzcd}
\end{center}
where $\O_L=\Ob(L)$ is the set of objects of $L$. It is trivial that with such a construction, $\Gamma\Xi=\Id_{\dgS_\O}$. We have then that $\Xi$ is fully faithful and in consequence that $\Map(\Xi,G)$ is a weak equivalence. By the 2-out-of-3 condition, we have that $f$ is a weak equivalence, and so we have that $\Phi$ commutes with colimits. By Proposition \ref{ch.2: fixed objects dg-Segal}, we know that $\hocolim(F_i)$ is DK-equivalent to $F$ in $\dgS_\O$, and that means that we have a DK-equivalence between $\hocolim(F_i)$ and $F$ in $\dgS$ too. We have finished the proof.
\\\end{proof}

Let us see now about $\Sing(T_i)\to \Sing(T)$. 

\begin{lema}Let $F_*\to F$ be a dg-Segal hypercover as constructed before, let $T_*$ be a simplicial object in $\dg_\O$ with $\O=\pi_0(F(k))$ such that $\Sing(T_*)=F_*$ and $T=\hocolim T_i$. Then the augmented simplicial object $T_*\to T$ is a $\FreeS$-hypercover of dg-categories.
\end{lema}

\begin{proof}
As we have fixed objects on this, by Lemma \ref{ch. 2 split hyper} if for all $x,y\in\O$, the augmented object $T_*(x,y)\to T(x,y)$ is a split hypercover of complexes, then $T_*\to T$ is a hypercover of dg-categories and we have finished.
\\
\\Let us fix $x,y\in\O$. By Lemma \ref{ch. 2 dg-Segal to complex hyper}, $T_*(x,y)\to F(x,y)$ is a split hypercover. But as it is a split hypercover, by Lemma \ref{ch.2: hyper complexes} we know that $T(x,y)=\hocolim(T_i(x,y))\simeq F(x,y)$. So the augmented object $T_*(x,y)\to T(x,y)$ is a split hypercover on $\Ch$, and in consequence $T_*\to T$ is a hypercover on $\dg$ and we have finished our proof.
\\\end{proof}

Now we just need to put everything together.

\begin{teo}\label{ch. 2: essential surjectivity}Let $F$ be a functor that satisfies the dg-Segal conditions. Then there exists a dg-category $T$ such that the morphism $\Sing(T)\to F$ is a DK-equivalence.
\end{teo}

\begin{proof}
Let $\Sing(T_*)=F_*\to F$ be a dg-Segal hypercover constructed following Theorem \ref{ch. 2 dg-Segal hypercover}. We have then a diagram of this form:
	\begin{center}
		\begin{tikzcd}
			\Sing(\hocolim(T_i))=\Sing(T)\ar[rr]&&F\\
			\hocolim(\Sing(T_i))=\hocolim(F_i)\ar[u,"\phi"]\ar[rru,"\zeta"]
		\end{tikzcd}
	\end{center}	
We have proven that the corresponding augmented object $T_*\to T$ is a hypercover of dg-categories, and by Lemma \ref{ch.2 hocolim DK}, the morphism $\phi$ is a DK-equivalence. We have also proven that $\hocolim(F_i)\to F$ is also a DK-equivalence, so by Remark \ref{ch. 2: pseudo 2oo3}, that means that our morphism $\Sing(T)\to F$ is also a DK-equivalence and we have finished.
\\\end{proof}

And we are done.

\begin{teo}Assume Hypothesis \ref{ch. 2: Re-zk} is true. Then, if we take the complete dg-Segal model structure on $\Fun^\S(\Free_\S\op,\sS)$, the functor $\Sing: \dg\to \dgSc$ is essentially surjective. Hence there exists an equivalence of categories of the form
$$\Ho(\dg)\to \Ho(\dgSc). $$
\end{teo}

\chapter{Future work}\label{Future work}

\epigraph{\textit{"  SUNDAY \\(The first day of the rest of their lives)"}}{---Terry Pratchett, Neil Gaiman, \textit{Good Omens}}

Now that we have finished talking about the results that are completely solved, it is time to tackle results that are in progress and those who we consider to be interesting prospects in a more general sense.

\section{Complete dg-Segal spaces and DK-equivalences}

We saw in the last chapter that the main result of this thesis hinges on one hypothesis. Let us remind the reader of it here.

\begin{hyp}[Hypothesis \ref{ch. 2: Re-zk}] Let $f:F\to G$ be a morphism between two functors satisfying the dg-Segal conditions. Then, $f$ is a DK-equivalence if and only if it is a weak equivalence in the complete dg-Segal model structure.  
\end{hyp}

Now, this is not a new or surprising condition to have. Indeed, a very similar situation appears in \cite[Theorem 7.7, proof in Section 14]{ComSegalSpacesREZK}: in the classic case the weak equivalences for the complete Segal model structure are exactly the Dwyer-Kan equivalences between complete Segal spaces. In that paper, this identification between Dwyer-Kan equivalences and complete Segal weak equivalences is done by the means of a completion functor such that for all Segal space $W$ there exists a map $i_W:W\to \widehat{W}$ with the following properties:

\begin{enumerate}
	\item The completion $\widehat{W}$ is a complete Segal space.
	\item The completion map $i_W$ is a weak equivalence in the complete Segal model structure.
	\item The completion map $i_W$ is a Dwyer-Kan equivalence.
\end{enumerate}

Using such a completion map, the author reduces the problem from asking whether a Dwyer-Kan equivalence $f:W\to W'$ between Segal spaces is a weak equivalence for the complete Segal model structure to whether its completion $\hat{f}$ is one, and some relatively straightforward arguments prove the if and only if from the theorem.
\\
\\The main issue with this proof is how to construct the completion functor mentioned above. Obviously, a simple fibrant replacement would be enough to assuage conditions 1 and 2, but not 3. Instead, the author constructs their complete Segal space by way of defining a certain simplicial space $E(m)$ for all $m\in\N$ and computing
$$\widehat{W}_n=diag([m]\to \Map(E(m),W^{F(n)}))=diag([m]\to (W^{E(m)})_n). $$

We expect the proof of our hypothesis to be something along these lines. Unfortunately, the construction won't be as smooth-sailing in our case: Rezk uses the direct product to define his exponentials, but if there is one thing that is well known about the monoidal structure of dg-categories, it's that it isn't compatible with its model structure (see \cite{Toen-dg}: the object $\D(1,0,1)$ is cofibrant in the model category of dg-categories, but it is easy to prove that $\D(1,0,1)\otimes\D(1,0,1)$ is not). We will be forced to define our own version of the monoidal structure.
\\
\\Let us start at the beginning. We have seen that we have the linearisation adjunction, $j_!:\Fun(\Delta\op,\sS)\rightleftharpoons \Fun^\S(\Free\op_\S,\sS):j^*$, and we have defined $E_k=j_!(E)$. That will be our best candidate in order to reproduce our $E(m)$ in this new context.

\begin{rem} Attention! In Rezk's case, $E$ was the nerve of the category with two objects and an isomorphism between them, $0\tilde{\to} 1$. But its equivalent is not true in our case: 
$$E_k\neq \Sing((0\tilde{\to}1) \otimes k). $$
And we should be grateful for that: indeed, $\Sing((0\tilde{\to}1) \otimes k)=\Sing(k)$, so $E_k$ being equal to it would mean that the complete dg-Segal model structure is in fact exactly the same as the dg-Segal model structure!
\end{rem}

\begin{defin}For all $n\in\N$, we have that
$$E_k(n)=\L j_!(0\tilde{\to} 1\tilde{\to}\ldots \tilde{\to} n). $$
\end{defin}

Following the example of \cite{ComSegalSpacesREZK}, we would want to define our dg-completion functor as follows:

\begin{defin}Let $X$ be a dg-Segal space. We call \textbf{the completion of $X$}, and we denote by $\hat{X}$, a complete dg-Segal space of the form
$$\hat{X}=\hocolim([i]\to X^{E_k(i)}).$$
\end{defin}

\begin{rem}We attract the reader's attention to the fact that, if this is defined correctly, its image by the forgetful functor $j^*$ will give us the completion of $j^*(X)$ in classical Segal spaces.
\end{rem}

The problem here is that we need to define this exponential, which asks from us to define a tensor product. We have that, for all $L\in\FreeS$
$$X^{E_k(n)}(L)\simeq \Map(\Sing(L)\otimes E_k(n),X). $$

But what does $\Sing(L)\otimes E_k(n)$ mean? Well, even if $E_k(n)$ itself is not of the form $\Sing(T)$ for a certain $T\in\dg$, we can write it as $E_k(n)=\hocolim\Sing(L_i)$ where $L_i\in\FreeS$ for all $i\in\N$. So we technically could write this tensor product as 
$$\Sing(L)\otimes E_k(n)=\hocolim \Sing(L)\otimes \Sing(L_i). $$

Now, our first instinct would be to write this as $\Sing(L\otimes L_i)$; but $L\otimes L_i$ isn't a free dg-category of finite type anymore. Maybe we could reiterate the process by writing $L\otimes L_i$ as a homotopy colimit of another $K_*$? We will see.

\begin{rem}It is important to notice that, should this method succeed (and we see no reason why it wouldn't), we not only have finished the proof of our theorem, but also defined a monoidal structure on $\dgSc$ that is compatible with the model structure, giving in its wake a possible monoidal model structure to the category of dg-categories. 
\end{rem}

\section{The linear simplex category}\label{linear simplex}

The free dg-categories of finite type were, sadly, not our first attempt at a model for dg-categories. Indeed, when we started we expected to have something a little closer to the simplex category we use in our definition of Segal spaces. The methods we wanted to use didn't work with them, though, which is why we opted for the free dg-categories instead; but as we haven't lost hope of being able to bring it back to it, we explain the linear simplex category here.

\subsection{First constructions}

In this text, we have used many different versions of dg-categories that were some version of "a finite number of objects with morphisms between them being $k[s]$". For example, when we discussed Tabuada's cofibrantly generated model structure for $\dg$ in \cite{TAB}, we constructed a series of objects of the form $\D(1,s,1)$, for all $s\in\Z$, dg-categories with two objects and $k[s]$ the cochain complex that is $k$ in degree $s$. We have also the linearisation functor, that gives us $\D(n,0,1)$, dg-categories with $n$ objects and $k$ the morphism between $i$ and $i+1$ for all $0\leq i\leq n-1$. 
\\
\\We are now going to expand on those definitions, and  make what will be the objects of our new linear simplex category.

\begin{defin}Fix $n\in\N$ a natural and let $s\in\Z^n$ and $d\in\N^n$. We call $\Dnsd$ a dg-category consisting of the following data:
\begin{itemize}
	\item A set of $n+1$ objects, that we will denote by $0,\ldots, n$ for clearness' sake.
	\item For all $0\leq i-1<i\leq n$, a complex $\Hom(i-1,i)$ concentrated in degree $s_i$, where it is $k^{d_i}$ i.e. a complex which is zero everywhere except for in degree $s_i$. In other words, for $0\leq i<j\leq n$, a complex $\Hom(i,j)$ of the form 
	$$\Hom(i,j)=\otimes_{i< k\leq j}\Hom(k-1,k). $$
\end{itemize}
In other words, a dg-category $\Dnsd$ is as follows:
\begin{center}
		\begin{tikzcd}
			0\ar[r, "k^{d_1}\l s_1\r"]\ar[loop left, "k"]& 1\ar[r, "k^{d_2}\l s_2\r"]\ar[loop below, "k"]&\ldots\ar[r]&n-1\ar[r,"k^{d_n}\l s_n\r"]\ar[loop below, "k"]&n\ar[loop right, "k"]
		\end{tikzcd}
	\end{center}
\end{defin}

%\begin{nota}Fix $n\in\N$ a natural and let $s\in\Z^n$ and $d\in \N^n$. We call $\Dnsd$ the dg-category consisting of $n+1$ objects, which we will call $0,1,\ldots,n$ for clearness sake, and for each $0< i\leq n$ $\Hom(i-1,i)$ is the complex concentrated in degree $s_i$, where it is $k^{d_i}$. If for some $0< i\leq n$ we have $d_i=0$, then the complex $\Hom(i-1,i)$ is the one which is $0$ everywhere. For all $0\leq i\leq n$, the complex $\Hom(i,i)$ is the one concentrated in degree $0$, where it is $k$. 
%\end{nota}

\begin{nota}We denote by $\D(0)$ or $k$ the dg-category $\D(0,s,d)$ with only one object and $k$ as its complex of endomorphisms. We remind the reader that this dg-category is the unit for the tensor product.
\end{nota}

\begin{defin}We call \textbf{the linear simplex category}, and we denote it by $\D$, the full subcategory of the category of dg-categories $\dg$ which has the dg-categories $\Dnsd$ as objects for all $n\in\N$, $s\in\Z^n$ and $d\in\N^n$. 
\end{defin}

One immediate advantage of the objects in $\D$ is that it is way easier to compute the morphism sets in $\dg$ from $\Dnsd$ to a general dg-category $T$ than it is to compute them from a general dg-category to another.

\begin{prop}\label{ch. 2:hom}Let $T$ be a object in $\dg$ and $\Dnsd$ in $\D$. Let $\phi_*:\Hom(\Dnsd, T)\to \Hom(\coprod \D(0),T)$ the morphism induced by the inclusion $\phi:\coprod_{n+1}\D(0)\to \Dnsd$ where we send each copy of $\D(0)$ to a different object in $\Dnsd$. Then, we have the following isomorphisms:
\begin{itemize}
	\item The set of morphisms from $\D(0)$ to $T$ is isomorphic to the objects in $T$, $\Hom(\D(0),T)\simeq \Ob(T)$. In other words, giving a morphism $\D(0)\to T$ is equivalent to fixing a point in $T$.
	 \item The fiber of $\phi$ is given by 
	$$\phi_*^{-1}(t)\simeq\prod \Hom_\Ch (k^{d_i}\left[ s_i\right], T(t_{i-1}, t_i))\simeq\prod (Z^{s_i}(\Hom_T(t_{i-1},t_{i})))^{d_i}, $$
	where $Z^{s_i}(\Hom_T(t_{i-1},t_{i}))$ is the set of cocycles of $\Hom(t_{i-1},t_{i})$ in degree $s_i$, i.e. the set of objects $a\in\Hom(t_{i-1},t_{i})^{s_i}$ such that $d(a)=0$. In other words, if we fix $n+1$ points in $T$, $t_0,\ldots,t_n$, then describing the set $\Hom(\Dnsd, T)$ comes down to describing the set of cocycles in degree $s_i$ of the cochain complex $T(t_{i-1}, t_i)$ for all $0< i\leq n$.
\end{itemize}
\end{prop}

\begin{proof}
Let us start with $\Hom(\D(0),T)\simeq\Ob(T)$. Let $f:\D(0)\to T$ be an object in $\Hom(\D(0),T)$. Giving such a morphism is equivalent to giving an object in $T$ (the image of $0$) and a morphism in $T$ (the image of $k=\Hom(0,0)$). We will define the isomorphism as 
$$F:\Hom(\D(0), T)\to \Ob(T)$$
$$ \ \ f\mapsto f(0). $$
The surjectivity is obvious. The only thing that we need to check is that $F$ is, in fact, injective: that if we fix a point $t\in T$ in $T$, there is only one possible morphism $\Hom(0,0)=k\to \Hom(t,t)$. Let $f:\D(0)\to T$ be a dg-functor such that $F(f)=t$. As it is a morphism in $\dg$, we know that the diagram 
\begin{center}
	\begin{tikzcd}
		k\arrow[rd, "e_t"]\arrow[r, "e_0"] & \Hom(0,0)=k \arrow[d, "f"]\\
		& \Hom(t,t)
	\end{tikzcd}
\end{center}
commutes. In particular, that means that the image of the identity of $\Hom(0,0)=k$ is fixed, and that fixes the entirety of the morphism $f$. We have that if we fix an object $t\in T$ there is only one possible morphism $\Hom(0,0)\to \Hom(t,t)$, and the functor $F$ is injective. 
\\
\\The functor $F$ is an isomorphism and $\D(0)\simeq \Ob(T)$. In particular, giving a morphism from $\D(0)$ to $T$ is equivalent to fixing a point in $T$.
\\
\\Now let us do the fiber, and let us start by the first isomorphism. By definition of $\phi$, and by the last part of this proof, looking at the fiber $\phi_*^{-1}$ is equivalent to fixing $n+1$ objects $t_0,\ldots,t_n$ in $T$, and then looking at the morphisms $f:\Dnsd\to T$ such that $f(i)=t_i$ for all $0\leq i\leq n$. 
$$\phi_*^{-1}\simeq\{f\in\Hom(\Dnsd, T)/\ f(i)=t_i\ \forall i\in\Ob(\Dnsd)\}. $$ 
We have fixed the images of all the objects of $\Dnsd$. This set is entirely determined by the morphisms of cochain complexes between the fixed objects. But we know that the morphisms of objects in $\Dnsd$ are entirely determined by the ones between $i-1$ and $i$, for all $1\leq i\leq n$. So we have the following isomorphisms:
$$\phi_*^{-1}\simeq\{(f_i)_{1\leq i\leq n}/\ f_i: \Dnsd(i-1,i)=k^{d_i}\left[s_i\right]\to T(t_{i-1}, t_{i})\}\simeq \prod \Hom_\Ch(k^{d_i}\left[s_i\right], T(t_{i-1},t_{i})). $$
\\We have proven the first isomorphism, reducing the computation of the set of morphisms between these dg-categories to a computation of morphisms of complexes. But we can go even simpler. Let us attack the second isomorphism. For legibility's sake, let us take a cochain complex $A\in\Ch$, and we'll prove that 
$$\Hom_\Ch(k^{d_i}\left[s_i \right], A)\simeq Z^ {s_i}(A)^{d_i}. $$
A morphism of complexes from $k^{d_i}\l s_i\r$ to $A$ is of the form 
\begin{center}
	\begin{tikzcd}
		\ldots\arrow[r]&0\arrow[r]\ar[d, hookrightarrow]& k^{d_i} \arrow[d, "f_{s_i}"]\ar[r]& 0\ar[d, hookrightarrow]\ar[r]& \ldots\\
		\ldots\ar[r]& A^{s_i-1}\ar[r]& A^{s_i}\ar[r]&A^{s_i+1}\ar[r]&\ldots
	\end{tikzcd}
\end{center}
We can easily see that this morphism is determined solely by $f_{s_i}$. By definition of a morphism of complexes and of cocycles, we have, then, that
$$\Hom_\Ch(k^{d_i}\left[s_i \right], A)\simeq\{f\in\Hom_\Mod(k^{d_i}, A^{s_i})/\ d\circ f=0\}\simeq \Hom_\Mod(k^{d_i},Z^{s_i}(A)).  $$
\\
\\The only thing we have left to do is prove that $\Hom_\Mod(k^{d_i}, Z^{s_i}(A))\simeq Z^{s_i}(A)^{d_i}$. We know that $k^{d_i}$ is a free $k$-module with a basis consisting of $d_i$ generators, $e_1, \ldots, e_{d_i}$. We define a morphism
$$F: \Hom_\Mod(k^{d_i}, Z^{s_i}(A))\to Z^{s_i}(A)^{d_i} $$ 
$$f\mapsto (f(e_1), \ldots, f(e_{d_i})).$$
It is a classical result in Algebra that this is a bijection. 
\\
\\We have constructed the isomorphisms
$$\phi_*^{-1}(x)\simeq\prod \Hom_\Ch (k^{d_i}\left[ s_i\right], T(x_{i-1}, x_i))\simeq\prod (Z^{s_i}(\Hom_T(x_{i-1},x_{i})))^{d_i}, $$
and we have finished this proof.
\\\end{proof}

Already we know that every dg-category is fibrant, but the objects in $\D$ are also cofibrant. Let us get a computational lemma out of the way first. 

\begin{lema}\label{ch. 2:z-surj}Let $g:A\to B$ be a trivial fibration in $\dg$, let $s_i\in\Z$ and $x_{i-1}, x_i\in\Ob(A)$ two objects of $A$. Then the induced morphism of sets 
$$ Z^{s_i}(A(x_{i-1},x_{i}))\to Z^{s_i}(B(g(x_{i-1}),g(x_{i})))$$
is surjective. 
\end{lema}

\begin{proof}
It is just a question of diagram-chasing. Let $\alpha\in Z^{s_i}(B(g(x_{i-1}),g(x_{i})))$. As $g$ is a fibration, in particular it is surjective on the complex of morphisms, and there exists $\beta\in A(x_{i-1},x_{i})$, but not necessarily a cocycle. At the same time, we know that $g$ is a weak equivalence, so there is an isomorphism of cohomology groups $H^{s_i}(A(x_{i-1}, x_{i}))\simeq H^{s_i}(B(g(x_{i-1}), g(x_{i})))$, which means that there exists a $\beta'\in Z^{s_i}(A(x_{i-1}, x_{i}))$ such that $g(\beta')$ is in the same equivalence class as $\alpha$, 
$$g(\beta')-\alpha=g(\beta')-g(\beta)=g(\beta'-\beta)\in \Im d_{B, i}.$$
There is, then, a $\gamma \in B^{s_i-1}(g(x_{i-1}), g(x_{i}))$ such that $d(\gamma)=g(\beta-\beta')$. On the other hand, using again the surjectivity of $g$, we have a $\gamma'\in A^{s_{i}-1}(x_{i-1}, x_{i})$ such that $g(\gamma')=\gamma$, and in consequence $d(g(\gamma'))=g(\beta-\beta')$. By the linearity of $g$, we have that $g(\beta'-d(\gamma'))=g(\beta)=\alpha$, and so $\beta'-d(\gamma')$ is a preimage of $\alpha$ in $Z^{s_i}(A(x_{i-1},x_{i}))$. 
\\
\\So the morphism is surjective and we have finished.
\\\end{proof}

\begin{coro}\label{ch. 3: Delta cofibrant}All objects in $\D$ are cofibrant in the model structure on $\dg$.
\end{coro}

\begin{proof}
The case of $\D(0)$ is trivial: from Theorem \ref{ch. 1:Tab-model} we know that $\0\to \D(0)$ is a generating cofibration, so in particular a cofibration. The dg-category $\D(0)$ is cofibrant.
\\
\\Let $n\in\N$, and let $s\in\Z^n$, $d\in\N^n$. We want to prove that $\Dnsd$ is cofibrant. For that, we will use the description of a cofibration as a morphism which has the left lifting property with respect to all trivial fibrations. Let $g:A\to B$ be a trivial fibration for the model structure for $\dg$, and $f:\Dnsd\to B$ a dg-morphism. We want to find a lift for this commutative square:
\begin{center}
	\begin{tikzcd}
		\0\ar[d]\ar[r]&A\ar[d, "g"]\\
		\Dnsd\ar[r, "f"]\ar[ru, dashrightarrow, "h"]&B
	\end{tikzcd}
\end{center} 
As $\0$ is the initial object in $\dg$, the existence of this lift is equivalent to the following condition: for every trivial fibration $g:A\to B$ and every $f:\Dnsd\to B$ there exists a factorization of $f$ through $g$, i.e. there exists a morphism $h:\Dnsd\to A$ such that $f=g\circ h$. This is, in turn, equivalent to asking that 
$$g_*:\Hom(\Dnsd,A)\to\Hom(\Dnsd,B) $$
is surjective.
\\
\\We know that such a morphism is surjective on the objects. As such, we can fix $n+1$ objects in $A$, $x_0,\ldots, x_n\in A$. Fixing the objects allows us to use the result in Proposition \ref{ch. 2:hom}: we have reduced the problem to proving that 
$$\prod_{x_{i-1},x_{i}\in A} Z^{s_i}(A(x_{i-1},x_{i}))^{d_i}\to \prod_{x_{i-1},x_{i}\in A} Z^{s_i}(B(g(x_{i-1}),g(x_{i})))^{d_i}$$
is surjective. As we are taking the product and the powers over the same families of indexes, proving that this morphism is surjective is equivalent to asking that the morphism 
$$ Z^{s_i}(A(x_{i-1},x_{i}))\to Z^{s_i}(B(g(x_{i-1}),g(x_{i})))$$
is. But we have already proven that that is the case in Lemma \ref{ch. 2:z-surj}.
\\
\\The morphism $g_*: Z^{s_i}(A(x_{i-1}, x_{i}))\to Z^{s_i}(B(g(x_{i-1}), g(x_{i})))$ is surjective. The morphism $\0\to \Dnsd$ is a cofibration, so $\Dnsd$ is cofibrant and we have finished.
\\\end{proof}

The computations for the morphism sets are not all we are going to need here. Indeed, we will be using the mapping spaces in the construction of our simplicial presheaf category, so it would be useful to have a similar result to the one in Proposition \ref{ch. 2:hom} but for the mapping spaces. We will, luckily, get one, but before that we will need a couple of preliminary results.

\begin{lema}\label{ch. 2:cof}Let $n\in\N$, $s\in\Z^n$ and $d\in\N^n$. The morphism $\phi:\coprod_n\D(0)\to \Dnsd $ is a cofibration in the model structure of $\dg$. 
\end{lema} 

\begin{proof}
Once again, to prove that $\phi$ is a cofibration we use the fact that it is a cofibration if and only if $\phi$ has the left lifting property with respect to all trivial fibrations. This time we don't have the luck to have the initial object as our top corner, but we can still see this property as some kind of surjectivity condition. Indeed, $\phi$ has the left lifting property if and only if for all $g:A\to B$ trivial fibration the morphism
$$g_*': \Hom(\Dnsd, A)\to \Hom(\coprod k, A)\times_{\Hom(\coprod k, B)}\Hom(\Dnsd, B) $$
is surjective. 
\\
\\We want to use Proposition \ref{ch. 2:hom} to compute those $\Hom$, but we only know how to compute $\Hom(\Dnsd, B)$ with fixed objects. He have, then, to take the coproduct over all possible selections of $n+1$ points. Like that, we get the following isomorphisms:
$$  \Hom(\coprod k, A)\times_{\Hom(\coprod k, B)}\Hom(\Dnsd, B)\simeq \Ob(A)^{n+1}\times_{\Ob(B)^{n+1}}\coprod_{\Ob(B)^{n+1}}\prod Z^{s_i}(B(x_{i-1},x_{i}))^{d_i}$$
 $$\simeq \coprod_{\Ob(B)^{n+1}}\prod Z^{s_i}(B(x_i,x_{i+1}))^{d_i}.$$
 We have reduced the problem to proving that 
 $$ \coprod_{\Ob(A)^{n+1}}\prod Z^{s_i}(A(x_{i-1},x_{i}))^{d_i}\to \coprod_{\Ob(B)^{n+1}}\prod Z^{s_i}(B(g(x_{i-1}),g(x_{i})))^{d_i}$$
is surjective. Once again, we can ignore the products and the coproducts in this, and reduce the problem even more by taking those out: the question we end up with whether for all $g:A\to B$ trivial cofibrations the induced morphism on cocycles
$$Z^{s_i}(A(x_i, x_{i+1}))\to Z^{s_i}(B(g(x_i), g(x_{i+1}))) $$
is surjective. But that is exactly the result of Lemma \ref{ch. 2:z-surj}, so we have our result. The morphism  $\phi:~\coprod_n\D(0)\to \Dnsd $ is a cofibration and we have finished.
\\\end{proof}

\subsection{A characterization of Maps}

\begin{lema}\label{Ch. 2: fiber}Let $T$ be a dg-category. Let $n\in\N$, $s\in\Z^n$ and $d\in\N^n$. There exists a homotopy pullback of simplicial sets of the following form:
\begin{center}
	\begin{tikzcd}
		\Map_{\coprod k/\dg}(\Dnsd, T)\ar[r]\ar[d, "U'"]\arrow[dr,phantom, "\ulcorner^h", very near start]&*\ar[d]\\
		\Map_{\dg}(\Dnsd, T)\ar[r, "\phi^*"]& \Map_\dg(\coprod_{n+1} k, T)
	\end{tikzcd}
\end{center}
where $\phi^*$ is the morphism induced by $\phi:\coprod_{n+1} k\to \Dnsd$ as above, and $U'$ is induced by the forgetful functor $U:\coprod_{n+1} k/\dg\to \dg$. 
\\
\\In particular, $\Map_{\coprod k/\dg}(\Dnsd, T)$ is the homotopy fiber of $\phi^*$.
\end{lema}

\begin{proof}
Let us take $C_*(T)$ a simplicial frame of $T$ in $\coprod k/\dg$. As $U$ is a right Quillen adjoint, the image $U(C_*(T))$ is a simplicial frame of $U(T)=T$ in $\dg$. By abuse of notation, we will still denote $U(C_*(T))$ by $C_*(T)$. We have proven that both $k=\D(0)$ and $\Dnsd$ are cofibrant, so we don't need to take cofibrant replacements. With those choices made, the diagram from the lemma could be rewritten by saying that there exists a homotopy pullback of the form 
\begin{center}
	\begin{tikzcd}
		\Hom_{k\coprod k/\dg}(\Dnsd, C_*(T))\ar[r]\ar[d, "U'"]&*\ar[d]\\
		\Hom_{\dg}(\Dnsd, C_*(T))\ar[r, "\phi^*"]& \Hom_\dg(k\coprod k, C_*(T)).
	\end{tikzcd}
\end{center}
The simplicial set $\Hom_{\coprod k/\dg}(\Dnsd, C_*(T))$ is literally the fiber of $\phi^*$. But here we are asking for a \textit{homotopy} fiber, not just a fiber. For that we will use the properness of $\sS$. We know by Corollary \ref{ch. 1:fiber} that on a right proper model category, if we have a diagram $X\to Z\leftarrow Y$ where at least one of these two arrows is a fibration, then the pullback is naturally weak equivalent to the homotopy pullback. We have proven in Lemma \ref{ch. 2:cof} that the morphism $\phi:\coprod_n\D(0)\to \Dnsd $ is a cofibration, and according to Proposition \ref{ch. 1:cof-to-fib}, if $\phi$ is a cofibration, the morphism $\phi^*$ is a fibration. So we have a pullback along a fibration, and by right properness of $\sS$, $\Hom_{k\coprod k/\dg}(\Dnsd, C_*(T))$ is a homotopy pullback of $\phi^*$ along a point. As all dg-categories are fibrant, we don't need to get a fibrant replacement of this object to get a homotopy fiber. The object $\Hom_{k\coprod k/\dg}(\Dnsd, C_*(T))$ is a homotopy fiber of $\phi^*$ over a point and we have our result. We have finished.
\\\end{proof}

\begin{prop}\label{ch. 2: computation map}Let $T$ be a dg-category and $\Dnsd$ in $\D$. Let $\phi_*:\Hom(\Dnsd, T)\to \Hom(\coprod \D(0),T)$ the morphism induced by the inclusion $\phi:\coprod_{n+1}\D(0)\to \Dnsd$ where we send each copy of $\D(0)$ to a different object in $\Dnsd$. Then, we have the following weak equivalences in $\sS$:
\begin{itemize}
	\item The mapping space from $\coprod\D(0)$ to $T$ is weak equivalent to the cartesian product of $n+1$ copies of the nerve of $F(\D(0),T)$, the category of quasi-representable $T\op$-modules with only weak equivalences as its morphisms. In other words,
	$$\Map(\coprod\D(0),T)\simeq \vert T\vert\times\ldots\times\vert T\vert=\times_{n+1}N(F(\D(0),T)). $$
	\item Let $x\in\Ob(T)^{n+1}$ be an $(n+1)$-uple of objects. The mapping space from $\Dnsd$ to $(T,x)$ in $\coprod k/\dg$, the category of dg-categories with $n+1$ fixed objects, is weak equivalent to the product over $1\leq i\leq n$ of the mapping spaces from $k^{d_i}\l s_i\r$ to $T(x_{i-1},x_i)$ in the category of cochain complexes. In other words,
	$$\Map_{\coprod k/\dg}((\Dnsd), (T,x))\simeq\prod_{i}\Map_\Ch(k^{d_i}\l s_i\r, T(x_{i-1},x_i)). $$
\end{itemize}
\end{prop}

\begin{proof}
The first weak equivalence is a direct consequence of Theorem \ref{ch. 1:map k}. Indeed, that result says that if we have two dg-categories, in our case $k$ and $T$, the simplicial set $\Map(k,T)$ is weak equivalent to the nerve of the category of quasi-representable $k\otimes T\op$-modules with equivalences as morphisms. But $k$ is the unit for the tensor product, so $F(k,T)$ is isomorphic to the category of quasi-representable $T\op$-modules with the appropriate morphisms. The Cartesian product is trivial.
\\
\\For the next part, we will prove it for $n=1$. Indeed, the reasoning stays the same, and the writing is easier. In that case, we define a Quillen adjunction
\begin{center}
	\begin{tikzcd}
		\Ch\ar[r, "\xi", shift left=0.5ex]& k\coprod k/ \dg\ar[l, "\beta", shift left=0.5ex]\\
		E\ar[r, mapsto]& \D(1,E)\\
		T(x_1, x_2)& (T,x_1, x_2)\ar[l, mapsto] 
	\end{tikzcd}
\end{center}
where $\D(1,E)$ denotes the dg-category with two objects, $k$ as the cochain complexes of endomorphisms, and $E$ as the complex of morphisms between the two objects. In particular, $\xi(k^d\l s\r)=\D(1,s,d)$.
\\
\\Let us take $C_*(T)$ a simplicial frame in $k\coprod k/\dg$. As $\beta$ is a right Quillen functor, the cochain complex $\beta(C_*(T))$ is still a simplicial frame. We have already proven in Corollary \ref{ch. 2:cof} that all $\D$ are cofibrant, and in Corollary \ref{ch. 1:hom-cofib} that all $\beta(\D(1,s,d)=k^{d}\l s\r$ are too. By definition of an adjunction, we have that
$$\Map_\Ch(k^d\l s\r, T)=\Hom_\Ch(k^d\l s\r, \beta(C_*(T)))\simeq \Hom_{k\coprod k/\dg}(\D(1,s,d),C_*(T))=\Map_{k\coprod k/\dg}(\D(1,s,d),T)$$
We have the weak equivalence and we have finished.
\\ \end{proof}

Finally, this last result gives us a way to characterize weak equivalences in $\dg$ using elements in $\D$. In fact, we don't even need the entirety of the objects in $\D$: it suffices with the ones with two objects and no degree.

\begin{teo}Let $T$ and $T'$ be two dg-categories and $f:T\to T'$ a dg-functor between them. Then $f$ is a weak equivalence in $\dg$ if and only if the induced morphisms
$$\Map(\D(0),f):\Map_\dg(\D(0),T)\to \Map_\dg(\D(0),T')$$
$$\forall s\in\Z,\ \ \Map(\D(1,s,1),f):\Map_\dg(\D(1,s,1),T)\to \Map_\dg(\D(1,s,1), T')$$
are weak equivalences in $\sS$.
\end{teo}

\begin{proof}
One implication is trivial. According to Proposition \ref{ch.1:Map Quillen}, if $Y$ is a fibrant object then $\Map(-,Y)$ is a right Quillen functor, and in particular it preserves weak equivalences between fibrant objects. But in $\dg$ every object is fibrant, so if $f$ is a weak equivalence then both $\Map(\D(0),f)$ and $\Map(\D(1,s,1),f)$ are weak equivalences.
\\
\\Now let us prove the other implication. Assume that both $\Map(\D(0),f)$ and $\Map(\D(1,s,1),f)$ for all $s\in\Z$ are weak equivalences in $\sS$, and let us prove that that's enough to get that the original dg-functor $f$ is a weak equivalence. A dg-functor in $\dg$ is a weak equivalence if it is quasi-essentially surjective and quasi-fully faithful.
\\
\\A dg-functor $f$ is quasi-essentially surjective if the induced map $\l f\r:\l T\r\to \l T' \r$ is essentially surjective. That is equivalent to proving that the map $(\l T\r/\cong)\to )\l T'\r/\cong)$ between the isomorphism classes of the homotopy categories is surjective. We have assumed that the map $\Map(k,f):\Map(k,T)\to \Map(k,T')$ is a weak equivalence in $\sS$. By definition of a weak equivalence in $\sS$, that means that for all $i\in\N$ and for all point $x\in T$ the induced morphism $\pi_i(\Map(k,f),x)$ is an isomorphism. In particular, it is an isomorphism if $i=0$. But we have seen in Corollary \ref{ch. 1:map to iso classes } that the set $\l k,T \r\simeq \pi_0(\Map(k,T), x)$ is isomorphic to the set of classes of isomorphisms in $\l T\r$. We have then that the map 
$$(\l T\r/\cong)\simeq \pi_0(\Map(k,T),x)\to \pi_0(\Map(k,T'),f(x))\simeq (\l T'\r/\cong) $$
is not only surjective, but an isomorphism. The dg-functor $f$ is quasi-essentially surjective.
\\
\\A dg-functor $f$ is quasi-fully faithful if for all $x,y\in T$ the induced morphism $T(x,y)\to T'(f(x),f(y))$ is weak equivalence in $\Ch$. That is, if we want to prove that $f$ is quasi-fully faithful we need to prove that for all $i\in\Z$ and for all $x,y\in T$ the induced morphism $H^i(f):H^i(T(x,y))\to H^i(T'(f(x),f(y)))$ is an isomorphism of groups. We have proven in Proposition \ref{Ch. 2: fiber} that we have a fibration $E=\Map(\D(1,s,1),T)\to B=\Map(k\coprod k, T)$ which has $F=\Map_{k\coprod k/\dg}(\D(1,s,1), T)$ as its homotopy fiber. That gives us a long exact sequence of the form 
\begin{center}
	\begin{tikzcd}
		\ldots\ar[r]&\pi_n(F,x)\ar[r]&\pi_n(E,x)\ar[r]&\pi_n(B,x)\ar[r]&\pi_{n-1}(F,x)\ar[r]&\ldots\\
		 &\ldots\ar[r]&\pi_1(E,x)\ar[r]&\pi_1(B,x)\ar[r]&\pi_0(F,x)&
	\end{tikzcd}
\end{center}
and another one of the same form for $E'=\Map(\D(1,s,1),T')\to B'=\Map(k\coprod k, T')$ with the fiber $F'=\Map_{k\coprod k/\dg}(\D(1,s,1), T')$. By using the morphisms $\Map(\D(1,s,1),f)$ and $\Map(k,f)$ we can induce a morphism of complexes of the form
\begin{center}
	\begin{tikzcd}
		\ldots\ar[r]&\pi_n(F,x)\ar[rr]\ar[d]&&\pi_n(E,x)\ar[rr]\ar[d]&&\pi_n(B,x)\ar[d]\ar[rr]&&\pi_{n-1}(F,x)\ar[r]\ar[d]&\ldots\\
		\ldots\ar[r]&\pi_n(F',x)\ar[rr]&&\pi_n(E',x)\ar[rr]&&\pi_n(B',x)\ar[rr]&&\pi_{n-1}(F',x)\ar[r]&\ldots
	\end{tikzcd}
\end{center}
By definition of a weak equivalence in $\sS$, the morphisms $\pi_i(\Map(\d(1,s,1),f),x)$ and $\pi_i(\Map(k,f),x)$ are isomorphisms for all $i\in\Z$. We have a morphism of long exact sequences which are isomorphisms in two out of every three. By the Five Lemma, that means that $\pi_i(\Map_{k\coprod k/\dg}(\D(1,s,1),f),x)$ is also an isomorphism for all $i\geq 1$. But we have already computed the simplicial sets $\Map_{k\coprod k/\dg}(\Dnsd,T)$ in Proposition \ref{ch. 2: computation map}; we have that 
$$\Map_{k\coprod k/\dg}((\D(1,s,1), (0,1)), (T,(x_0,x_1)))\simeq \Map_\Ch(k\l s\r, T(x_0,x_1)). $$
We know that for all $i\geq 0$,
$$\pi_i(\Map_\Ch(k\l s\r, T(x_0,x_1)),x)\simeq H^{-i-s}(T(x_0,x_1)),$$
So if we put the isomorphisms next to each other, we have, for all $s\in\Z$ and all $i\geq 1$ that 
\begin{center}
	\begin{tikzcd}
		H^{-i-s}(T(x_0,x_1))\simeq\pi_i(\Map_{k\coprod k/\dg}(\D(1,s,1),T),x)\ar[r, "\sim"]& \ldots\\
		\longrightarrow \pi_i(\Map_{k\coprod k/\dg}(\D(1,s,1),T'),f(x))\simeq H^{-i-s}(T'(f(x_0),f(x_1)))&
	\end{tikzcd}
\end{center}
and the morphism $f:T\to T'$ induces an isomorphism between cohomology groups of the form $H^{-i-s}$. As we have that isomorphism for every $s\in\Z$, by making $s$ vary we have the isomorphism on cohomology groups for every degree. The morphism $f$ is quasi-fully faithful. 
\\
\\The morphism $f:T\to T'$ is a weak equivalence in $\dg$, and we have finished.
\\\end{proof}

\subsection{Expected result}

It was at this moment when we realized that we would not be able to use the results in Section \ref{hypercovers} for $\D$. Indeed, the conditions on Theorem \ref{Ch. 2: existence hypercovers} are too strong: the linear simplex category is not closed for finite coproducts. But just because it isn't possible to use the hypercover results as we have put them it doesn't mean we cannot do it some other way. Indeed, free dg-categories have already very little relationships between the morphism complexes, so it's not that much of a stretch to consider that we could potentially rewrite complete dg-Segal spaces as some kind of simplicial functors from $\D$ to $\sS$. 

\begin{hyp}There exists a full subcategory of $\Fun^\S(\D\op, \sS)$ that is equivalent to $\dgSc$. 
\end{hyp}

\section{Other possible prospects}

Although the last two sections are the most obvious and direct prospects from this work, they are not the only ones. We add here some other possible applications and questions.

\subsection{Automorphisms of dg-cat}

At the end of his paper \cite{TOEN-infini}, Toën used his result about complete Segal spaces to compute the group of automorphisms of $\infty$-categories, which turned out to be just $\Z/2\Z$, containing only the dualisation as a non-identity automorphism. We could try to do the same thing for $\dg$. 
\\
\\Indeed, it makes sense to attempt to use our construction for this: if we think of it as a kind of "presentation by generators and relations", where the category of free dg-categories $\FreeS$ (or $\D$, if we have already proven the hypothesis in the last section) would be the generators and the weak equivalences we described as the conditions on complete dg-categories, it is common practice to define the morphisms from $\dg$ to another model category $C$ by looking at the morphisms from $\FreeS$ to  $C$ that turn the relations into weak equivalences. This method could work with any category, but in particular we can use it to see the automorphisms of $\dg$. 
\\
\\Sadly, in our case the automorphism group of $\dg$ will not be as simple as the one for $\infty$-categories. It will contain $\Z/2\Z$, of course, as the duality is still an invertible morphism that turns dg-categories into dg-categories, but it won't be enough. As we are working with $k$-linear categories, it is not unexpected that the automorphisms of the ring $k$ will be included in the group of automorphisms, and maybe some other groups too.
\\
\\It would also be interesting to explore what the higher homotopy groups of $Aut(\dg)$ would be: we expect the Hochschild homology to appear at some point. 

\subsection{Hypercover definitions}

As any invested reader will have noticed, even though the definitions are different, the construction of $M_0$-hypercovers and of dg-Segal hypercovers follow the exact same structure. In the same vein, the definitions might not be the same, but they follow the same pattern: we have a definition of some kind of epimorphism involving the Map functor, and then we ask that for all $n\in\N$,
$$\Phi_n: A_n\to A^{\partial\Delta^n} $$
is that kind of epimorphism. This is not a particularity of the hypercovers we have defined here. We have already seen something similar in the case of hypercovers in simplicial sets at Definition \ref{ch. 2: defin hypercovers sSet} and in \cite{HTT} Lurie defines a hypercover over an $\infty$-topos to be an augmented simplicial object such that $\Phi_n$ is an effective epimorphism for all $n$. With so many similar definitions around, there is a natural question that arises: is it possible to find a general definition that combines all of these?

\bibliographystyle{alpha}
\bibliography{bibliographie}

\end{document}